\documentclass{amsart}

\usepackage{paralist}
\usepackage{color}
\usepackage{mathrsfs}
\usepackage{amssymb}
\usepackage{amsmath}
\usepackage{amsfonts}
\usepackage{stmaryrd}
\usepackage[all]{xy}
\usepackage{tikz-cd}
\usetikzlibrary{shapes,arrows}
\usepackage{adjustbox}
\usepackage{enumitem}

\makeatletter
\@namedef{subjclassname@2020}{\textup{2020} Mathematics Subject Classification}
\makeatother

\newtheorem{theorem}{Theorem}[section]
\newtheorem{lemma}[theorem]{Lemma}

\newtheorem{claim}{Claim}[theorem]

\newtheorem{question}[theorem]{Question}
\newtheorem{task}[theorem]{Task}

\theoremstyle{definition}
\newtheorem{definition}[theorem]{Definition}
\newtheorem{example}[theorem]{Example}
\newtheorem{remark}[theorem]{Remark}

\newtheorem{exercise}[theorem]{Exercise}

\usepackage[pdftex]{hyperref}
\hypersetup{
    colorlinks=true, 
    linktoc=all,     
    linkcolor=blue,  
    allcolors=blue,
}

\usepackage[framemethod=tikz]{mdframed}

\tikzset{
  double arrow/.style args={#1 colored by #2 and #3}{
    -stealth,line width=#1,#2, 
    postaction={draw,-stealth,#3,line width=(#1)/3,
                shorten <=(#1)/3,shorten >=2*(#1)/3}, 
  }
}

\tikzset{
	MyPersp/.style={scale=1.1,x={(1.1cm,0cm)},y={(0.8cm,0.47cm)},
    z={(0cm,1cm)}},
	MyPoints/.style={fill=white,draw=black,thick}
		}

\title[Higher walks]{An introduction to higher walks}

\dedicatory{Dedicated to Justin Tatch Moore and Stevo Todorcevic.}

\author{Jeffrey Bergfalk}
\address{Departament de Matem\`{a}tiques i Inform\`{a}tica \\
Universitat de Barcelona \\
Gran Via de les Corts Catalanes, 585 \\ 08007 Barcelona, Catalonia}
\email{bergfalk@ub.edu}
\urladdr{www.jeffreybergfalk.com}


\thanks{This work was supported by Mar\'{i}a Zambrano and Marie Sk\l odowska Curie (Project 101110452: CatT) Fellowships at the University of Barcelona}

\begin{document}

\begin{abstract}
The following is an introduction to the study of \emph{higher walks}, by which we mean a family of higher-dimensional extensions of Todorcevic's method of walks on the ordinals. After a brief review of this method, including, for example, definitions of the classical functions $\mathrm{Tr}$ and $\rho_2$ induced by a choice of $C$-sequence, we record a shortlist of desiderata for such extensions, along with $(n+1)$-dimensional functions $\mathrm{Tr}_n$ and $\rho_2^n$ (induced by a choice of \emph{higher-dimensional} $C$-sequence) which we show to satisfy the bulk of them.
Much of the interest of these \emph{higher walks functions} lies in their affinity, as in the classical $n=1$ case, for the ordinals $\omega_n$ (we show, for example, that $\rho^n_2$ determines both $n$-dimensional linear orderings and $n$-coherent families on $\omega_n$, and that higher walks define nontrivial elements of the $n^{\mathrm{th}}$ cohomology groups of $\omega_n$), and in the questions that they thereby raise both about the combinatorics of the latter and about higher-dimensional infinitary combinatorics more generally; we collect the most prominent of these questions in our conclusion. These objects are also, though, of a sufficient combinatorial richness to be of interest in their own right, as we have underscored via an extended study of the first genuine novelty among them, the function $\mathrm{Tr}_2$.
\end{abstract}

\maketitle

\setcounter{tocdepth}{1}
\tableofcontents
\section{Introduction}

Among the most consequential developments in the study of infinitary combinatorics over the past forty years have been the arrival and elaboration of Todorcevic's method of walks on the ordinals.\footnote{The publication \cite[p.\ 288]{todpairs} traces the walks apparatus to Berkeley seminar notes Todorcevic drafted in 1984: walks indeed turn forty this year.}
We will briefly review its essentials in Section \ref{subsect:classical} below; recall most immediately, though, 
\begin{enumerate}
\item that the materials of this method are a family of interrelated functions taking nondecreasing pairs $(\alpha,\beta)$ of ordinals as inputs and deriving their outputs from the finite \emph{minimal walk} (determined by some background choice of $C$-sequence) from $\beta$ down to $\alpha$, and
\item that these functions are, on the countable ordinals most particularly, so simultaneously fundamental and adaptable as to collectively figure as a kind of ``master key'' to the combinatorics of $\omega_1$.
\end{enumerate}
The subject of the present work is a sequence of $(n+1)$-dimensional extensions of this method which we term, in aggregate, \emph{higher walks}. These neatly
generalize the classical $n=1$ case of item (1) above; higher walks of order $n=2$, for example, pass ``between'' nondecreasing triples $(\alpha,\beta,\gamma)$ of ordinals, inducing a family of derived functions comprising one main theme of our account below.

How or how far item (2) generalizes alongside these forms then figures as the subject's driving question, one lent particular charge and interest
by how persuasively, as we aim to show, the classical structures referenced in (1) do extend in higher dimensions to higher cardinals $\omega_n$.
Much of the difficulty of this question stems from the obscurity of the combinatorics of the higher cardinals $\omega_n$ themselves, and among the main stakes of this subject is, in fact, the prospect of a more methodical approach to that obscurity.

The idea of higher-dimensional extensions of walks techniques to higher cardinals is, broadly speaking, an old one.
The works \cite{cubes}, \cite[\S 10]{todwalks}, and \cite{Lopez-Abad}, for example, all in their own ways file under this heading, and the relation of what we term \emph{higher walks} to their various constructions is a good question, one to which we will return in our conclusion.
Two guiding principles do, however, sharply distinguish this article's approach from those of its predecessors:
\begin{enumerate}
\item Cohomological framings of classical walks phenomena have strongly influenced our approaches to their generalizations.
\item We have pursued, in senses formalized by criteria (n.1)--(n.6) of Section \ref{subsect:generalizations} below, as systematic and total a generalization of the classical walks apparatus as possible (the surprise has been in how far this program has proven practicable).
\end{enumerate}
We have also, as a corollary of the second item, deliberately minimized our appeal to assumptions beyond the $\mathsf{ZFC}$ axioms, at least in the initial development of this subject.
These emphases reflect a broader research context which it will be useful to review before proceeding.
Informing item (1), for example, are both the recognition, dating at the latest to \cite{Talayco}, of nontrivial coherence relations as cohomological in nature and the recognition in \cite{dimords} of the classical walks apparatus as the essential content of the $n=1$ instance of Mitchell's theorem \cite[Thm.\ 36.8]{Mitchell} (see Theorem \ref{thm:Mitchell} and the surrounding discussion below).
Three works subsequent to the latter recognition even more fundamentally shape our approach.
The first, \cite{CoOI}, joint with Chris Lambie-Hanson and posted in early 2019, recorded a cohomological approach to infinitary combinatorics in which nontrivial coherence figures as merely the first in a graded family of incompactness principles exhibiting a particular affinity for the cardinals $\omega_n$ (an update and expansion of this work, joint with Lambie-Hanson and Jing Zhang, is currently in preparation \cite{BLHZ}).
In the course of a new proof of the aforementioned theorem of Mitchell's, a second work \cite{TFOA} isolated many of the structures studied herein; see again Section \ref{subsect:generalizations} for further discussion.
Drafted just prior to that work, though, was the seed of the present one, a note ``intended firstly for Justin Moore and Stevo Todorcevic, who both at this point have probably heard more than enough from me about the existence of higher walks, and less than enough about what I believe them to be'' \cite{HighWalksNote}.
That they have, in the time since, heard more of the latter has not changed how much of my thinking in these matters is ultimately addressed to them.
This is much of the logic of this paper's dedication; see our Acknowledgements section below for more of it.

One major motivation for the above sequence of works was a body of contemporaneous research on the derived limits of inverse systems indexed by the partial order $({^\omega}\omega,\leq^*)$ (see \cite{sim1}, \cite{sim2}, \cite{PHn}), research which persistently underscores how much we still have left to learn about the higher-dimensional combinatorics of the cardinals $\omega_n$.
This we mention both to suggest some of the broader potential significance of higher walks and to reaffirm that substantial higher-dimensional combinatorial principles, like the family of $\mathrm{PH}_n$ principles isolated in \cite{PHn}, for example, do still remain to be discovered.

To sum up: this article is an introduction to a subject which is still in the early stages of its development; in it, we survey the main contours of what is both known and wondered about higher walks. It is simultaneously intended as a kind of progress report and, particularly in its later, more exploratory sections, as an invitation.

\section{The plan of the paper}
\label{sect:plan}

Our guiding conception is the following: walks on ordinals are a \emph{remarkably elementary phenomenon}, one isolated and leveraged by Todorcevic (in \cite{todpairs, todwalks, todcoh}) into a far-reaching \emph{combinatorial technique}, but one so versatile as to amount to a \emph{systematic method of study of the $\mathsf{ZFC}$ combinatorics of $\omega_1$, and of much beyond}.\footnote{Kanamori aptly and relatedly writes of walks as seeming to access ``the immanent structure of the uncountable, from which a wide range of combinatorial consequences flow'' \cite[p.\ 72]{Kanamori}.}

Part of what we wish to extend, in particular, is walks' elementarity; below, accordingly, we will tend to assume no more mathematical background of readers than a basic knowledge of the ordinals.
A main apparent exception will be our periodic references to cohomology groups; here these are invoked so essentially as a convenient shorthand, though, that their appearance imposes no real additional requirements upon the reader.
Most particularly, we presume no deep knowledge of walks themselves, and will briefly review what we need of the technique in Section \ref{subsect:classical} below.
For readers interested in learning more of the classical theory, we note that Todorcevic's most informative and influential treatments of the subject include his 1987 article ``Partitioning Pairs of Countable Ordinals'' \cite{todpairs}, his 2007 book \emph{Walks on Ordinals and Their Characteristics} \cite{todwalks}, and his 2012 CRS-Fields-PIMS Prize lecture \cite{todlecture}.
Further developments and applications of walks form too large a literature to adequately survey here; we note only as a few representative post-2007 references the works \cite{RTrans, Lopez-Abad, RZTrans}. Our brief review of the theory will do no justice to the wealth of consistency results involving walks, for the simple reason that a $\mathsf{ZFC}$ emphasis seems to us, in the present context, particularly orienting and suggestive.

We record our notational conventions in Section \ref{subsect:notation}.
Section \ref{subsect:generalizations} records the defining principles of a multidimensional generalization of the classical theory which we will pursue under the name of \emph{higher walks} for the remainder of the paper.
In Section \ref{sect:higherC} we describe the fundamental materials of higher walks, namely \emph{higher $C$-sequences} and \emph{internal walks}.
In Section \ref{sect:basics} we describe degree-$n$ walks' most basic representative, a function $\mathrm{tr}_n$ which takes nondecreasing $(n+1)$-tuples of ordinals as inputs and outputs $n$-branching finite trees, recovering classical walks in its $n=1$ instance.
In Section \ref{sect:Trn}, we describe a fuller upper trace function $\mathrm{Tr}_n$ along with a rho function $\rho_2^n$ which specializes to the classical function $\rho_2^1=\rho_2$, and briefly discuss some rationales for the higher walks forms we're proposing.
In Section \ref{sect:higher_coherence}, we introduce the property of \emph{$n$-coherence} and, using a lower trace function $\mathrm{L}_n$, show the fiber maps of $\rho_2^n$ on $\omega_n$ to possess it (along with the fiber maps of $\mathrm{Tr}_n$ itself, in some looser sense, with implications for any naturally associated rho function); for $n\leq 2$ we show them to possess it on higher cardinals as well under $\square$-like assumptions on the underlying $C$-sequence.
This is the paper's lengthiest section, and readers might profitably restrict their attention to its main results at a first pass.
In Section \ref{sect:explorations} we compute the upper regions of a sample walk $\mathrm{Tr}_2(+,\alpha,\beta,\gamma)$, the better to foreground the combinatorial richness of higher-dimensional walks and some of the distinctive questions which their analyses involve, and we plot these computations in a sequence of accompanying diagrams collected, for ease of reading, in our paper's concluding appendix.
We turn in Section \ref{sect:nontriviality} to \emph{nontrivial} $n$-coherent families of functions, and describe a function $r_2^n$ deriving from $\mathrm{Tr}_n$ whose fiber maps form a nontrivial $n$-coherent family on $\omega_n$; we also discuss nontriviality scenarios for extensions of more classical rho functions.
In Section \ref{sect:further}, we record a notion of \emph{higher-dimensional linear orders} and show how higher walks induce them, via a generalization of classical constructions of a Countryman line. On this point and several others, however, what we report at the moment amounts mainly just to promising avenues for further investigation, prospects which we summarize in a short list of tasks and questions and tasks in our concluding Section \ref{sect:conclusion}.

Readers familiar with \cite{TFOA} will note some overlap of its introduction of higher walks with portions of Sections \ref{subsect:classical}, \ref{subsect:internal_walks}, \ref{sect:basics}, and \ref{sect:Trn} herein. This is to some degree unavoidable. Note, however, that these overlapping treatments carry a different inflection in the present context: here, for one thing, they are conducted with reference to a much more expansive notion of higher $C$-sequence; for another, here the stress is on all $n>0$, not just the representative case of $n=2$.
\subsection{Notational conventions}
\label{subsect:notation}
We like to think of our notations as standard; in matters pertaining most directly to walks, we have endeavored, even in higher dimensions, to follow the organizing conventions of \cite{todwalks}. Further notational conventions are as follows: we sometimes write $A^{(\kappa)}$ for the direct sum of $\kappa$ many copies of an abelian group $A$, and $\mathbb{Z}/2$ for $\mathbb{Z}/2\mathbb{Z}$. We write $\mathrm{Sym}(X)$ and $\mathrm{Alt}(X)$ for the symmetric and alternating groups on a set $X$, respectively, and write $\mathrm{sgn}$ for the \emph{sign} or \emph{signum} function identifiable with the quotient map $\mathrm{Sym}(X)\to\mathrm{Sym}(X)/\mathrm{Alt}(X)\cong (\{-1,1\},\times)$.
All groups $A$ herein are topologically discrete, so that the \emph{locally constant} functions from a space $Y$ to $A$ are precisely the continuous ones.

We denote ordinals by $\xi$ or $\eta$ or an earlier letter of the Greek alphabet, reserving $\kappa$ or $\lambda$ for cardinals. Any of these is tacitly endowed with its order topology, as is standard. For any set $X$ of ordinals, $\mathrm{acc}(X)=\{\alpha<\mathrm{sup}\,X\mid\mathrm{sup}(\alpha\cap X)=\alpha\}$; similarly for $\mathrm{lim}(X)$, but with $\alpha\leq\mathrm{sup}\,X$ instead as the range of definition. We write $\mathrm{cf}(X)$ and $\mathrm{otp}(X)$ for the cofinality and ordertype of such sets $X$, write $\mathrm{Cof}(\kappa)$ for the class of ordinals of cofinality $\kappa$, and for any $k<n<\omega$ let $S^n_k=\mathrm{Cof}(\omega_k)\cap\omega_n$.

We write $[X]^n$ for the collection of size-$n$ subsets of $X$. When $X$ is a set of ordinals, we identify the elements of $[X]^n$ with their increasing enumerations, and hence with elements of the set $X^{[n]}$ of nondecreasing $n$-tuples of $X$, writing $\vec{\alpha}=(\alpha_0,\dots,\alpha_{n-1})$ for a typical element of the latter and $\vec{\alpha}^i$ for the $(n-1)$-tuple formed by deleting the $i^{\mathrm{th}}$ coordinate of $\vec{\alpha}$. Considerations of readability have uniformly prevailed over consistency in our choices between notations $\vec{\alpha}^\frown\vec{\gamma}$, $\vec{\alpha}\vec{\gamma}$, and $(\vec{\alpha},\vec{\gamma})$ for the concatenation of two such tuples.

\section{Classical walks and principles of their generalization}
In Section \ref{subsect:classical}, we briefly review the basic contours of classical walks. This is both for the reader's convenience and to foreground what seem to us walks' most essential features, the better to motivate the project of their generalization.
We compile generalized formulations of these features into a shortlist of desiderata in Section \ref{subsect:generalizations}; it is these which it will be our paper's main business thereafter to pursue.

\subsection{Classical walks}
\label{subsect:classical}

We now describe the fundamentals of walks on the ordinals; except where otherwise indicated, this material is standard and drawn from the references cited in Section \ref{sect:plan} above.

The initiating input is a \emph{$C$-sequence}:

\begin{definition} For any ordinal $\gamma$, a \emph{$C$-sequence on $\gamma$} is a family $\mathcal{C}=\langle C_\beta\mid\beta<\gamma\rangle$ in which each $C_\beta$ is a closed cofinal subset of $\beta$. For concision, we will sometimes term closed cofinal subsets of ordinals \emph{clubs}, and 
we will sometimes call a $C$-sequence in which $\mathrm{otp}(C_\beta)=\mathrm{cf}(\beta)$ for every $\beta<\gamma$ \emph{ordertype-minimal}.
\end{definition}

The assignments $C_\beta:=\beta$, of course, define the most trivial class of examples. In contrast, more interesting examples, and ordertype-minimal $C$-sequences
in particular, will tend to involve some application of the axiom
of choice.\footnote{Since, in principle, all further essential definitions are recursive on the input of a $C$-sequence, the latter may even be thought of as some materialization or delimitation of choice (or of even stronger assumptions) within the theory.}

With respect to a fixed C-sequence $\mathcal{C}=\langle C_\beta\mid\beta<\gamma\rangle$, for any $\alpha\leq\beta<\gamma$ the \emph{upper trace of the walk from $\beta$ down to $\alpha$} is recursively defined as follows:
\begin{align}
\label{eq:Tr}
\mathrm{Tr}(\alpha,\beta)=\{\beta\}\,\cup\,\mathrm{Tr}(\alpha,\min (C_\beta\backslash\alpha)),
\end{align}
with the \emph{boundary condition} that $\mathrm{Tr}(\alpha,\alpha)=\{\alpha\}$ for all $\alpha<\gamma$.
The \emph{walk} from $\beta$ to $\alpha$ is loosely identified with its upper trace, or with the collection of \emph{steps} between successive elements thereof, which is typically pictured as a series of arcs cascading in a downwards left direction, as on the left-hand side of Figure \ref{thepifigure} below.
Since the ordinals are well-founded, this series is finite.\footnote{Hence $C$-sequences, via walks, induce finitary relations between pairs of ordinals not far in spirit from a distance function; for further development of this point, see \cite[Ch. 3]{todwalks}.}
The \emph{number of steps} function $\rho_2$ sends any $\alpha$ and $\beta$ as above to $|\mathrm{Tr}(\alpha,\beta)|-1$. When $\gamma=\omega_1$ and each $C_\beta$ in $\mathcal{C}$ is of minimal possible order-type, the $\rho_2$ \emph{fiber maps} $\varphi_\beta(\,\cdot\,):=\rho_2(\,\cdot\,,\beta):\beta\to\mathbb{Z}$ form a \emph{nontrivial coherent family of functions}. More precisely, under these assumptions \begin{align}\label{c}\varphi_\beta\big|_\alpha - \varphi_\alpha=0\hspace{.8 cm}\textnormal{modulo locally constant functions}\end{align}
for all $\alpha\leq\beta<\gamma$, but there exists no $\varphi:\xi\to\mathbb{Z}$ such that
\begin{align}\label{nt}\varphi\big|_\alpha - \varphi_\alpha=0\hspace{.8 cm}\textnormal{modulo locally constant functions}\end{align}
for all $\alpha<\gamma$ (see \cite[Cor.\ 2.7]{CoOI}, a superficial variation on \cite[Lem.\ 2.4.2, 2.4.3]{todwalks}; proof-sketches of these relations also appear in Sections \ref{subsect:classical_argument} and \ref{sect:nontriviality} below). \emph{Coherence} broadly refers to relations like the first above; \emph{nontriviality} refers to relations like the second.

Complementary to the upper trace is the \emph{lower trace of the walk from $\beta$ to $\alpha$}, which is defined as follows. Enumerate $\mathrm{Tr}(\alpha,\beta)$ in descending order as $$(\beta_0,\beta_1,\dots,\beta_{\rho_2(\alpha,\beta)}=\alpha)$$ and let $\max\,\varnothing=0$. Under the above assumptions, the sequence
$$\mathrm{L}(\alpha,\beta)=\langle\, \max(\bigcup_{i\leq j} \alpha\cap C_{\beta_i})\mid j<\rho_2(\alpha,\beta)\,\rangle$$
is well-defined, ascending towards $\alpha$ from below in tandem with $\mathrm{Tr}(\alpha,\beta)$'s descent to $\alpha$ from above. Related considerations determine the \emph{maximal weight} function $\rho_1$, defined again under the above assumptions as
$$\rho_1(\alpha,\beta)=\max_{i<\rho_2(\alpha,\beta)} |\,\alpha\cap C_{\beta_i}|.$$
Just as for $\rho_2$, when $\gamma=\omega_1$ the fiber maps $\varphi_\beta(\,\cdot\,):=\rho_1(\,\cdot\,,\beta)$ form a family of functions which is nontrivially coherent --- i.e., which follows the pattern of relations in (\ref{c}) and (\ref{nt}) above --- but in this case it is with respect to the modulus of \emph{finitely supported functions} (\emph{mod finite}, hereafter, for short; see Section \ref{subsect:classical_argument} for more on the coherence and nontriviality of $\rho_1$).

This brings us to the first of several summary points: walks engender \emph{multiple forms} of nontrivial coherence, a fact which may be formalized as follows.

\begin{theorem}
\label{thm:cohomological_framing}
For any abelian group $A$ and ordinal $\gamma$ let
\begin{itemize}
\item $\mathcal{F}_A$ denote the presheaf on $\gamma$ given by $U\mapsto\bigoplus_{U} A$, and
\item $\mathcal{A}$ denote the presheaf on $\gamma$ given by $U\mapsto\{\text{locally constant functions }f:U\to A\}$.
\end{itemize}
Then by way of the assignments
$$(\alpha,\beta)\mapsto(\varphi_\beta\big|_\alpha - \varphi_\alpha)\hspace{.8 cm}(\alpha\leq\beta<\gamma),$$
any \emph{mod finite} or \emph{mod locally constant} coherent family
$$\Phi=\langle\varphi_\beta:\beta\to A\mid\beta<\gamma\rangle$$ represents 
a cohomology class in the \v{C}ech cohomology group $\check{\mathrm{H}}^1(\gamma;\mathcal{F}_A)$ or $\check{\mathrm{H}}^1(\gamma;\mathcal{A})$, respectively; such a $\Phi$ is, moreover, nontrivial if and only if its associated class is.
In addition, for all $n> 0$,
$$\check{\mathrm{H}}^n(\gamma;\mathcal{F}_A)\cong\check{\mathrm{H}}^n(\gamma;\mathcal{A})\cong\mathrm{H}^n(\gamma;\mathcal{A}),$$ where the latter denotes the sheaf cohomology of $\gamma$ with respect to $\mathcal{A}$.
In particular, $\mathrm{H}^1(\omega_1;\mathbb{Z})\neq 0$, as witnessed by the nontrivial coherent families $\langle\varphi_\beta\mid\beta<\omega_1\rangle$ deriving, as above, either via $\rho_1$ or $\rho_2$ from any ordertype-minimal $C$-sequence on $\omega_1$.
\end{theorem}

The theorem amalgamates \cite[Thm.\ 2.6.3]{dimords}, \cite[Lem.\ 2.22]{CoOI}, and \cite[Thm.\ 4.1]{Wiegand}. The $n=1$ instance of the leftmost isomorphism in its third displayed equation should be read as: \emph{wherever there is a $\rho_2$-type function, there is a $\rho_1$-type function as well, and vice versa}; since the isomorphism holds for all $n$, this will also apply for higher-dimensional analogues of these functions, insofar as we succeed in defining them.\footnote{That this isomorphism is induced, at the level of the cochain complexes, not by an isomorphism but by a cochain homotopy underscores the sometimes singular utility of homological perspectives in this area. To be clear, by a ``$\rho_2$-type function'' above we mean an $r:[\gamma]^2\to A$ whose fiber maps nontrivially cohere modulo locally constant functions; similarly for a ``$\rho_1$-type function'' but with respect to the modulus of finitely supported functions.}


As the above might suggest, walks and nontrivial coherence relations exhibit a particular affinity for the ordinal $\omega_1$, one, indeed, which it would be difficult to overstate.
Here we should note that while we have focused, for economy of presentation, on the rho functions $\rho_1$ and $\rho_2$, there exist well-studied $\rho_0$ and $\rho_3$ functions as well; in each case, the associated family of fiber maps exhibits nontrivial coherence relations determining that most characteristic of combinatorial structures on $\omega_1$, namely, an Aronszajn tree.

\begin{theorem}
Let $\mathcal{C}$ be an ordertype-minimal $C$-sequence on $\omega_1$. Then for any $i\in\{0,1,2,3\}$,
$$T(\rho_i):=(\{\rho_i(\,\cdot\,,\beta)\big|_\alpha\,\mid\alpha\leq\beta<\omega_1\},\sqsubseteq)$$ is an Aronszajn tree.
\end{theorem}

Consider next, for $i\in\{1,2\}$, the ordering $\triangleleft_i$ on the countable ordinals given by $\alpha\triangleleft_i\beta$ if and only if $\alpha\neq\beta$ and \begin{itemize}
\item $\rho_i(\,\cdot\,,\alpha)=\rho_i(\,\cdot\,,\beta)\big|_\alpha$, or
\item $\rho_i(\xi,\alpha)<\rho_i(\xi,\beta)$ at the least ordinal $\xi$ at which $\rho_i(\,\cdot\,,\alpha)$ and $\rho_i(\,\cdot\,,\beta)$ disagree.
\end{itemize}
This is the natural \emph{branch ordering} of the tree $T(\rho_i)$ --- and it is the most basic example of a \emph{Countryman line}, i.e., of an uncountable linear order $L$ whose square can be covered by countably many chains.\footnote{Similarly for $T(\rho_0)$ and, under mild further $\mathsf{ZFC}$ assumptions $\mathcal{C}$, for $T(\rho_3)$ as well; see \cite[\S 2]{todwalks}. For the case of $i=2$, see \cite[Thm.\ 4.27]{PengThesis}. Such $L$, and indeed such $T(\rho_i)$, play a critical role within the class of uncountable linear orders; see \cite{Moore_Five}.}

We could go on like this, applying walks to construct $L$-spaces \cite{LSpace}, multiple witnesses to $\omega_1\not\to[\omega_1]^2_\omega$ \cite{todpairs}, etc.: one may even, in the process, come to regard walks as ``the universal fun on $\omega_1$'', as the fun through which all other fun on $\omega_1$ factors:
\medskip

\begin{quote}
Despite its simplicity, this structure [i.e., the walks apparatus] can be used to derive virtually all known other structures that have been defined so far on $\omega_1$. \cite[p.\ 7]{todwalks}
\end{quote}
\medskip

\noindent This is \emph{the very good news} about walks and $\omega_1$. Equally provocative is \emph{the less good news} about walks and $\omega_1$, namely the fact that much of the best news about them ends there. More precisely, although $C$-sequences will determine walks and rho functions on any larger class of ordinals, without assumptions supplementary to the \textsf{ZFC} axioms, the structures they induce tend to lack punch:
\medskip

\begin{quote}
The first uncountable cardinal is the only cardinal on which the theory can be carried out without relying on additional axioms of set theory. (\cite[p.\ 7]{todwalks})
\end{quote}
\medskip

\noindent All of this accords with a view, common among set theorists, of the $\mathsf{ZFC}$ combinatorics of $\omega_1$ as much more rich and determinate than that of any larger $\kappa$. Much of the interest of higher walks stems from how, drawing on precisely this background, they trouble this view.

The better to underscore these points, we close this subsection with a representative theorem; this will also afford us occasion to recall one last main actor in the theory: the principle $\square(\kappa)$, which traces to \cite{Jensen} by way of \cite{todpairs}.

\begin{definition}
\label{def:square}
Let $\kappa$ be a regular uncountable cardinal. $\square(\kappa)$ is the assertion that there exists a $C$-sequence $\mathcal{C}=\langle C_\beta\mid\beta<\kappa\rangle$ in which:
\begin{enumerate}
\item $C_\beta\cap\alpha = C_\alpha$ for all $\beta<\kappa$ and $\alpha\in \mathrm{acc}(C_\beta)$, but
\item there exists no club $C\subseteq\kappa$ such that $C\cap\alpha = C_\alpha$ for all $\alpha\in \mathrm{acc}(C)$.
\end{enumerate}
\end{definition}

Note that $\square(\aleph_1)$ follows from the $\mathsf{ZFC}$ existence of  ordertype-minimal $C$-sequences (i.e., \emph{ladder systems}) on $\omega_1$.
Now recall from Theorem \ref{thm:cohomological_framing} that the existence of nontrivial coherent $A$-valued families on an ordinal $\gamma$ may be more concisely expressed by the equation $\mathrm{H}^1(\gamma;\mathcal{A})\neq 0$. For the \emph{P-Ideal Dichotomy}, see \cite{todPID}; here it suffices to recall that the principle follows, for example, from the Proper Forcing Axiom.

\begin{theorem}
\label{thm:square_PID}
Let $A$ be a nontrivial abelian group, and let $\gamma$ be an ordinal of cofinality $\kappa$.
\begin{itemize}
\item If $\square(\kappa)$ holds, then $\mathrm{H}^1(\gamma;\mathcal{A})\neq 0$.
\item If the P-Ideal Dichotomy holds, then $\mathrm{H}^1(\gamma;\mathcal{A})\neq 0$ if and only if $\kappa=\aleph_1$.
\end{itemize}
\end{theorem}
The first item follows from \cite[Thm.\ 6.3.2, Lem.\ 7.1.10]{todwalks}, together with the argument of Subsection \ref{subsect:internal_walks}; there is a subtlety in this deduction, though, which we will address when we revisit its underlying coherence mechanism in Section \ref{sect:higher_coherence}.
The second item is also due to Todorcevic; see \cite[Thm.\ 3.16]{CoOI} for a proof.\footnote{Alert readers of Theorem \ref{thm:square_PID} might next wonder whether the existence of a nontrivial coherent family on a cardinal $\kappa$ is tantamount to the existence of a $\square(\kappa)$-sequence. It is a strictly weaker assumption; see \cite[Rmk.\ 3.41]{CoOI}.}

\subsection{Generalizations}
\label{subsect:generalizations}
Let us summarize the preceding discussion in the following six fundamental points:
\begin{enumerate}[label=(1.\arabic*)]
\item The arguments of classical walks are pairs $\alpha\leq\beta$
of ordinals.

\item Walks associate to each such pair a finite set $\mathrm{Tr}(\alpha,\beta)$ in a manner which is recursive on the input of a \emph{C-sequence} $\mathcal{C}$. 

\item \emph{Rho functions} $\rho_i(\alpha,\beta)$ record basic features of these walks.

\item Mild conditions on $\mathcal{C}$ entail \emph{coherence relations} on each of \emph{the families} $\langle\rho_i(\,\cdot\,,\beta):\beta\to A\mid \beta\in\omega_1\rangle$ \emph{of fiber maps of the rho functions}. (These conditions are ``mild'' in the sense, in particular, that they are easily arranged in any model of the $\mathsf{ZFC}$ axioms.)

\item Further mild conditions on $\mathcal{C}$ entail the \emph{nontriviality} of such families.

\item Principles like $\square(\kappa)$ furnish $C$-sequences $\mathcal{C}$ whereby the phenomena of items (1.4) and (1.5) extend to cardinals $\kappa>\omega_1$.
\end{enumerate}
As we have seen, items (1.4) and (1.5) admit the following rephrasing:
\begin{enumerate}[label=(1.\arabic*)]
\setcounter{enumi}{3}
\item For standard $\mathsf{ZFC}$ choices of $C$-sequence $\mathcal{C}$, each of the induced families $\Phi_i=\langle\rho_i(\,\cdot\,,\beta):\beta\to \mathbb{Z}\mid \beta\in\omega_1\rangle$ $(i\in\{1,2,3\})$ represents a \emph{$1$-cocycle} in the cochain complex defining the cohomology of $\omega_1$ with respect to a natural choice of presheaf $\mathcal{P}_i$.\footnote{For $i=3$, let $\mathcal{P}_i=\mathcal{F}_\mathbb{Z}$, just as for $i=1$; the case of $i=0$, from which the other rho functions may be seen as deriving, is subtler, for while $\rho_0$ doesn't itself admit so neat an algebraic framing, it plainly exhibits more abstract (nontrivial) coherence relations, generalizations of which are discussed in Section \ref{subsect:further_analysis}.}
Briefly: $[\Phi_i]\in \check{\mathrm{H}}^1(\omega_1;\mathcal{P}_i)$ --- and $[\Phi_2]\in \mathrm{H}^1(\omega_1;\mathcal{A})$ for $\mathcal{A}=\mathbb{Z}$, in particular.
\item For such $\mathcal{C}$, these cocycles are each nontrivial: $[\Phi_i]\neq 0$.
\end{enumerate}
The remainder of this paper traces a view of the foregoing as merely the $n=1$ instance of a more multidimensional walks apparatus in which the above points generalize to $n\geq 1$ as follows.
\begin{enumerate}[label=(n.\arabic*)]
\item The arguments of \emph{$n$-dimensional walks} are nondecreasing $(n+1)$-tuples $\vec{\alpha}$ of ordinals.
\item \emph{Higher walks} associate to each such argument a finite set $\mathrm{Tr}_n(\vec{\alpha})$ in a manner which is recursive on the input of a \emph{higher $C$-sequence $\mathcal{C}$}.
\item \emph{Higher rho functions} $\rho_i^n(\,\cdot\,)$ articulate elementary features of these higher-dimensional
walks via definitions generalizing those of the classical rho functions $\rho_i(\,\cdot\,)$ ($=\rho^1_i(\,\cdot\,)$).
\item For standard $\mathsf{ZFC}$ choices of higher $C$-sequence $\mathcal{C}$, the \emph{families of fiber maps} $$\Phi_i=\langle\rho_i^n(\,\cdot\,,\vec{\beta})\mid \vec{\beta}\in [\omega_n]^{n}\rangle$$
are \emph{$n$-coherent} in the way in which cohomological considerations lead us to expect.
More briefly, $[\Phi_i]\in\check{\mathrm{H}}^n(\omega_n;\mathcal{P}_i)$; such families, in other words, are naturally viewed as $n$-cocycles.
\item For such $\mathcal{C}$ and suitable $i$, these families $\Phi_i$ are \emph{nontrivial}, again in the way in which cohomological perspectives lead us to expect. Briefly: the cohomology classes $[\Phi_i]$ --- and suitable sheaf cohomology groups $\mathrm{H}^n(\omega_n;\mathcal{A})$, in particular --- are nonzero.
\item The phenomena of (n.4) and (n.5) extend to higher cardinals $\kappa$ under assumptions like $\square(\kappa)$.
\end{enumerate}
In the following sections, we will describe functions which amply satisfy items (n.1) through (n.4) and the first half of (n.6). Item (n.5) will be satisfied as well, but, as the qualifier ``suitable'' suggests, in a weaker and more tentative manner than we might ultimately hope; whether it holds in stronger senses is, broadly, this subject's most fundamental open question.
More precisely, we will show that the fiber maps of generalizations $r_2^n$ of \emph{non-classical} rho functions $r_2$ on $\omega_1$ do define nontrivial $n$-coherent families of functions witnessing that $\mathrm{H}^n(\omega_n;\mathcal{A})\neq 0$ for $A=\mathbb{Z}^{(\omega_n)}$ or $(\mathbb{Z}/2)^{(\omega_n)}$ (hereabouts, we mean \emph{rho function} in the extended sense of: \emph{any function recursively deriving from the system of walks induced by any given $C$-sequence}).
We will tend to focus nevertheless on the generalized \emph{classical} rho functions $\rho_i^n$, and on $\rho_2^2$ in particular, both for their relative transparency and for the larger questions they involve, not least of which is that of whether $\mathrm{H}^n(\omega_n;\mathbb{Z})\neq 0$ for all $n\in\omega$ is a $\mathsf{ZFC}$ theorem.

Let us record without delay the function $r_2$, along with its simpler variant $r_1$; their discussion will naturally open onto one last useful word of background.

\begin{definition}
\label{def:r2}
For any given ordinal $\delta$ and $C$-sequence $\mathcal{C}$ on $\delta$, define $r_1:[\delta]^2\to [\delta]^{<\omega}$ by
$$r_1(\alpha,\gamma)=\{\beta<\gamma\mid\text{max L}(\beta,\gamma)=\alpha\}.$$
Relatedly, define $r_2:[\delta]^2\to [\delta^2]^{<\omega}$ by
$$r_2(\alpha,\gamma)=\{(\beta,\text{min Tr}(\beta,\gamma)\backslash(\beta+1))\mid\beta<\gamma\text{ and max L}(\beta,\gamma)=\alpha\}.$$
\end{definition}

Observe that $\gamma=\bigsqcup_{\alpha<\gamma}r_1(\alpha,\gamma)$; put differently, $r_1(\,\cdot\,,\gamma)$ partitions $\gamma$ into  a $\gamma$-indexed family of finite (and possibly empty) sets. Similarly, viewing the steps $\{\beta_i,\beta_{i+1}\}$ of walks from $\gamma$ down to lower ordinals $\beta$ as the edges of a graph $G_\gamma$ on $\gamma$, observe that $r_2(\,\cdot\,,\gamma)$ partitions the edges of $G_\gamma$ into a $\gamma$-indexed family of finite (and possibly empty) paths.
Note next that the symmetric difference operation $\Delta$ endows the codomain of either $r_i$ with the structure of an abelian group.

\begin{theorem}
\label{thm:ri}
For any ordertype-minimal $C$-sequence on $\omega_1$ and $i\in\{1,2\}$, the family
$$\Phi_i=\langle r_i(\,\cdot\,,\gamma):\gamma\to[(\omega_1)^i]^{<\omega}\mid\gamma<\omega_1\rangle$$
is nontrivially coherent mod finite.
Each, in other words, witnesses the nonvanishing of $\mathrm{H}^1(\omega_1;\mathcal{A})$ for $A=([\omega_1]^{<\omega},\Delta)$ or, equivalently, for $A=(\mathbb{Z}/2)^{(\omega_1)}$.
\end{theorem}

\begin{proof}
The second assertion follows from the first one by way of Theorem \ref{thm:cohomological_framing}.
The coherence of these families follows from standard arguments which we review in Section \ref{subsect:classical_argument}.
Since a trivialization of $\Phi_2$ would, via first-coordinate projections, induce a trivialization of $\Phi_1$, the proof reduces to showing that the latter is nontrivial. Thus suppose for contradiction that $\varphi:\omega_1\to [\omega_1]^{<\omega}$ trivializes $\Phi_1$. Then for every $\gamma\in S^1_0$ there exists an $\eta_\gamma<\gamma$ such that $\varphi\big|_{[\eta_\gamma,\gamma)}=r_1(\,\cdot\,,\gamma)\big|_{[\eta_\gamma,\gamma)}$; in particular, $\mathrm{min}\,C_\gamma\backslash(\eta+1)\in \varphi(\eta)$ for all $\eta\in C_\gamma\backslash \eta_\gamma$ and $\gamma\in S^1_0$. Fix a stationary set $S\subseteq S^1_0$ with $\eta_\gamma$ constantly equal to some $\beta$ for all $\gamma\in S$ and a $\xi\geq\beta$ such that the set $B:=\{\mathrm{min}\,C_\gamma\backslash(\xi+1)\mid\gamma\in S\text{ and }\xi\in C_\gamma\}$ is infinite. Since by assumption $\varphi(\xi)$ contains $B$, it is infinite --- a contradiction, as desired.\footnote{In fact the argument shows that each $\Phi_i$ is nontrivially coherent modulo locally constant functions, and nontrivial (in either sense) even with respect to functions $\varphi:\omega_1\to\mathcal{P}(\omega_1)$.}
\end{proof}
The function $r_2$ was (to the best of this author's knowledge) first perceived in \cite{TFOA}; $r_2(\alpha,\gamma)$ essentially records the supports of the function $\mathtt{f}_1(0,\gamma)\big|_{\{\alpha\}\otimes[\omega_1]^2}$ appearing in Section 6.1 therein.
That work, in turn, was a study of what in interested circles is known as \emph{Mitchell's theorem}, which we might in the present context formulate as follows.
\begin{theorem}
\label{thm:Mitchell}
For any ordinal $\delta$ and $n>0$, let $\mathrm{H}^n(\delta;\,\cdot\,)$ denote the functor $A\mapsto \mathrm{H}^n(\delta;\mathcal{A})$ from the category of abelian groups to itself. Then $\mathrm{H}^n(\delta;\,\cdot\,)$ is the zero functor for any $\delta$ of cofinality less than $\omega_n$, and $\omega_n$ is the least ordinal $\delta$ for which this functor is nonzero.
\end{theorem}
This framing of the theorem combines \cite[Cor.\ 36.9]{Mitchell} and \cite[Thm.\ 4.1]{Wiegand} with the analyses and reformulations of \cite{CoOI, TFOA}.
We will sketch two proofs of its nonvanishing portion in Section \ref{sect:nontriviality} below; the fact that $\mathrm{H}^n(\delta;\,\cdot\,)$ vanishes for smaller-cofinality $\delta$  is essentially due to Goblot \cite{Goblot}.

Mitchell's theorem is, to a set theorist, provocative in at least two ways:
\begin{itemize}
\item It attests (without quite specifying) otherwise unperceived nontrivial higher-dimensional $\mathsf{ZFC}$ combinatorics on $\omega_n$.
\item Its $n=1$ case may be persuasively viewed as little other than an algebraic incarnation of the classical walks apparatus on $\omega_1$ (see \cite[\S 8.1]{TFOA}).
\end{itemize}
A main aim of \cite{TFOA}, accordingly, was some distillation of the higher-$n$ cases into more purely combinatorial forms, the better to assess the bearing of these results on the theories both of walks and of the ordinals $\omega_n$.
It soon grew clear, as we'll show below, that these forms satisfy sufficiently much of items (n.1) through (n.6) listed above to merit the name \emph{higher walks}; they are, in essence, the subject of the remainder of this paper.
In perfect analogy with the classical case, their story begins with higher $C$-sequences, and the latter is the subject of the following section.

One last point: a more complete summary of Section \ref{subsect:classical} would arguably  have included an \emph{item (1.7): walks economically induce the most fundamental combinatorial structures on $\omega_1$}, to be generalized to an item (n.7) replacing the parameter $\omega_1$ with $\omega_n$, likely restricting attention to $(n+1)$-dimensional combinatorics as well.
This seems to us both too nebulous and extravagant a point to formally pursue; portions of it may be read between the lines of items (n.1) through (n.6), though, as one further guiding prospect, and form the animating concern of Section \ref{sect:further}.
\section{Higher $C$-sequences and internal walks}
\label{sect:higherC}
An \emph{order-2 $C$-sequence on $\gamma$} is a classical $C$-sequence $\langle C_\beta\mid\beta\in\gamma\rangle$ coupled with choices of closed and cofinal $C_{\alpha\beta}\subseteq\alpha\cap C_\beta$ for each $\beta\in\gamma$ and $\alpha\in C_\beta$; higher-order $C$-sequences are simply further iterations of this idea.
We precede their formal definition in Section \ref{subsect:Higher_C} with a brief discussion of \emph{internal walks}, both to motivate such \emph{relativizations} $\langle C_{\alpha\beta}\mid\alpha\in C_\beta\rangle$ of classical $C$-systems
and to introduce a main ingredient of higher walks.

\subsection{Internal walks}
\label{subsect:internal_walks}
In this section, we describe a simple mechanism for extending the Theorem \ref{thm:cohomological_framing} result that $\mathrm{H}^1(\omega_1;\mathbb{Z})\neq 0$ to the result that $\mathrm{H}^1(\gamma;\mathbb{Z})\neq 0$ all ordinals $\gamma$ of cofinality $\aleph_1$; the method applies more generally, of course, and reduces the first item of Theorem \ref{thm:square_PID}, for example, to the case of $\gamma=\mathrm{cf}(\gamma)=\kappa$ as well.

To this end, fix an ordertype-minimal $C$-sequence on $\omega_1$ along with an increasing enumeration $\langle\eta_i\mid i\in\omega_1\rangle$ of a closed cofinal subset of an ordinal $\gamma$. For $\alpha<\beta$ in $\gamma$ define the \emph{upper trace $\mathit{Tr}^\gamma$ of the $C_\gamma$-internal walk from $\beta$ to $\alpha$} as follows: let $\eta_i=\min C_\gamma\backslash\alpha$ and $\eta_k=\min C_\gamma\backslash\beta$ and let \begin{align*} \mathrm{Tr}^\gamma(\alpha,\beta)=\{\beta\}\,\cup\,\{\eta_j\mid j\in\mathrm{Tr}(i,k)\}.\end{align*}
In particular, $\mathrm{Tr}^\gamma(\eta_i,\eta_k)$ is the image of the walk $\mathrm{Tr}(i,k)$ under the order-isomorphism $\pi:\omega_1\to C_\gamma$ mapping each $i$ to $\eta_i$. Let $\rho_2[\gamma](\alpha,\beta)=|\mathrm{Tr}^\gamma(\alpha,\beta)|-1$. Observe then that the fiber maps $$\{\rho_2[\gamma](\,\cdot\,,\beta):\beta\to\mathbb{Z}\mid\beta\in\gamma\}$$
define a nontrivial coherent family of functions modulo locally constant functions, i.e., they witness that $\mathrm{H}^1(\gamma;\mathbb{Z})\neq 0$. Observe also that if $\gamma=\omega_1=C_{\omega_1}$ then $\mathrm{Tr}^\gamma=\mathrm{Tr}$ and $\rho_2[\gamma]=\rho_2$.

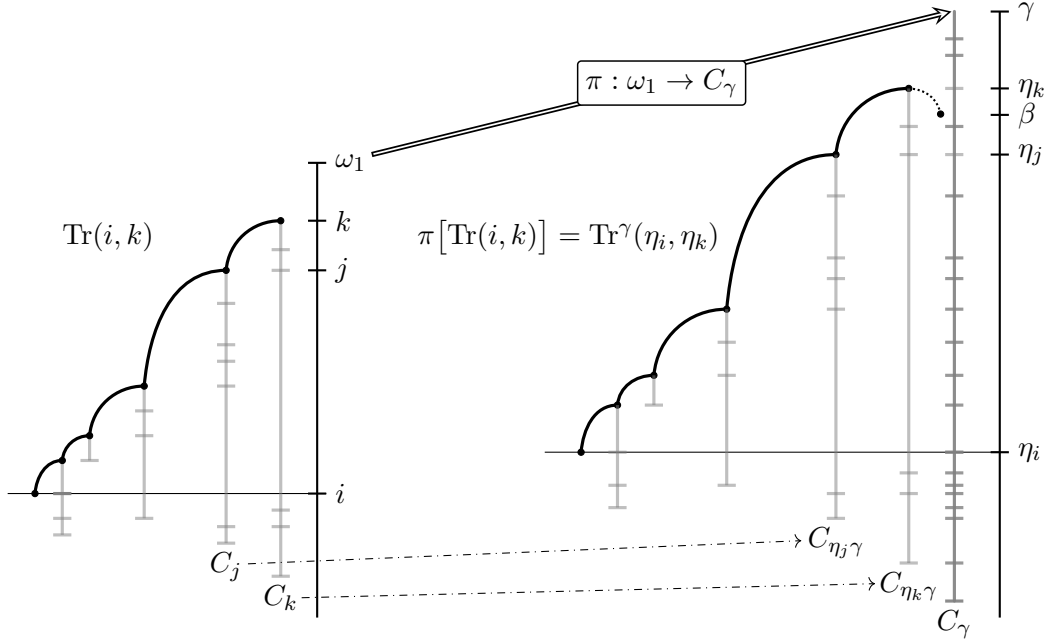
\begin{figure}
\centering
\begin{tikzpicture}[MyPersp,font=\large]
	\coordinate (A) at (-.3,0,1.5);
	\coordinate (B) at (3.2,0,1.5);
	\coordinate (C) at (3.1,0,5.5);
	\coordinate (D) at (3.1,0,0);
	
	\coordinate (E) at (5.6,0,2);
	\coordinate (F) at (10.6,0,2);
	\coordinate (G) at (10.1,0,7.33);
	\coordinate (H) at (10.1,0,0.2);	
	
	\draw (A)--(B);
	\draw[thick] (C)--(D);
	\draw (E)--(F);
	\draw[line cap=round, very thick, gray, opacity=.9] (G)--(H);	
	\draw[very thick, gray, opacity=.9] (10,0,0.2)--(10.2,0,0.2);
	\draw[thick] (10.6,0,7.33) -- (10.6,0,0);
	\draw[thick] (3,0,4.8)--(3.2,0,4.8);
	\draw[thick] (3,0,4.2)--(3.2,0,4.2);
	\draw[thick] (3,0,5.5)--(3.2,0,5.5);
	\draw[thick] (3,0,1.5)--(3.2,0,1.5);
	\draw (10.1,0,.2) node[below] {$C_\gamma$};
	\draw (2.7,0,4.8)[very thick] to[out=-180,in=85] (2.1,0,4.2);	
	\draw[very thick] (2.1,0,4.2) to[out=-180,in=85] (1.2,0,2.8);
		\draw[very thick] (1.2,0,2.8) to[out=-180,in=85] (.6,0,2.2);
		\draw[very thick] (.6,0,2.2) to[out=-180,in=85] (.3,0,1.9);
		\draw[very thick] (.3,0,1.9) to[out=-180,in=85] (0,0,1.5);
		\draw[very thick, gray, opacity=.5] (2.7,0,4.8)--(2.7,0,.5);
		\draw (2.7,0,.5) node[below] {$C_k$};
		\draw[very thick, gray, opacity=.5] (2.1,0,4.2)--(2.1,0,.9);
		\draw (2.1,0,.9) node[below] {$C_j$};
		\draw[very thick, gray, opacity=.5] (1.2,0,2.8)--(1.2,0,1.2);
		\draw[very thick, gray, opacity=.5] (.6,0,2.2)--(.6,0,1.9);
		\draw[very thick, gray, opacity=.5] (.3,0,1.9)--(.3,0,1);
		\draw (3.2,0,5.5) node[right] {$\omega_1$};
		\draw (3.2,0,4.8) node[right] {$k$};
		\draw (3.2,0,4.2) node[right] {$j$};
		\draw (3.2,0,1.5) node[right] {$i$};
		\draw (2.7,0,4.8) node[circle,draw,fill=black, scale=.25]{};
		\draw (2.1,0,4.2) node[circle,draw,fill=black, scale=.25]{};
		\draw (1.2,0,2.8) node[circle,draw,fill=black, scale=.25]{};
		\draw (.6,0,2.2) node[circle,draw,fill=black, scale=.25]{};
		\draw (.3,0,1.9) node[circle,draw,fill=black, scale=.25]{};
		\draw (0,0,1.5) node[circle,draw,fill=black, scale=.25]{};
		\draw[very thick, gray, opacity=.5] (2.6,0,4.2)--(2.8,0,4.2);
		\draw[very thick, gray, opacity=.5] (2.6,0,1.3)--(2.8,0,1.3);
		\draw[very thick, gray, opacity=.5] (2.6,0,1.1)--(2.8,0,1.1);
		\draw[very thick, gray, opacity=.5] (2.6,0,.5)--(2.8,0,.5);
		\draw[very thick, gray, opacity=.5] (2.6,0,4.45)--(2.8,0,4.45);
		\draw[very thick, gray, opacity=.5] (2,0,2.8)--(2.2,0,2.8);
		\draw[very thick, gray, opacity=.5] (2,0,3.8)--(2.2,0,3.8);
		\draw[very thick, gray, opacity=.5] (2,0,3.3)--(2.2,0,3.3);
		\draw[very thick, gray, opacity=.5] (2,0,3.1)--(2.2,0,3.1);
		\draw[very thick, gray, opacity=.5] (2,0,1.1)--(2.2,0,1.1);
		\draw[very thick, gray, opacity=.5] (2,0,.9)--(2.2,0,.9);
		\draw[very thick, gray, opacity=.5] (1.1,0,2.2)--(1.3,0,2.2);
		\draw[very thick, gray, opacity=.5] (1.1,0,2.5)--(1.3,0,2.5);
		\draw[very thick, gray, opacity=.5] (1.1,0,1.2)--(1.3,0,1.2);
		\draw[very thick, gray, opacity=.5] (.5,0,1.9)--(.7,0,1.9);
		\draw[very thick, gray, opacity=1] (.2,0,1.5)--(.4,0,1.5);
		\draw[very thick, gray, opacity=.5] (.2,0,1)--(.4,0,1);
		\draw[very thick, gray, opacity=.5] (.2,0,1.2)--(.4,0,1.2);
	\draw[very thick, gray, opacity=.8] (10,0,6.8)--(10.2,0,6.8);
	\draw[very thick, gray, opacity=.8] (10,0,7)--(10.2,0,7);
	\draw[very thick, gray, opacity=.5] (10,0,6.4)--(10.2,0,6.4);
	\draw[very thick, gray, opacity=.5] (10,0,5.6)--(10.2,0,5.6);
	\draw[very thick, gray, opacity=.5] (10,0,2)--(10.2,0,2);
	\draw[thick] (10.5,0,7.33)--(10.7,0,7.33);
	\draw[thick] (10.5,0,6.4)--(10.7,0,6.4);
	\draw[thick] (10.5,0,5.6)--(10.7,0,5.6);
	\draw[thick] (10.5,0,2)--(10.7,0,2);
	\draw[thick] (10.5,0,6.08)--(10.7,0,6.08);
	\draw (10.7,0,7.33) node[right] {$\gamma$};
		\draw (10.7,0,6.4) node[right] {$\eta_k$};
		\draw (10.7,0,5.6) node[right] {$\eta_j$};
		\draw (10.7,0,2) node[right] {$\eta_i$};
		\draw (10.7,0,6.07) node[right] {$\beta$};
	\draw (9.95,0,6.09)[densely dotted, thick] to[out=100,in=-2] (9.6,0,6.4);
	\draw (9.6,0,6.4)[very thick] to[out=-180,in=85] (8.8,0,5.6);	
	\draw[very thick] (8.8,0,5.6) to[out=-180,in=85] (7.6,0,3.73);
		\draw[very thick] (7.6,0,3.73) to[out=-180,in=85] (6.8,0,2.93);
		\draw[very thick] (6.8,0,2.93) to[out=-180,in=85] (6.4,0,2.57);
		\draw[very thick] (6.4,0,2.57) to[out=-180,in=85] (6,0,2);
				\draw (9.6,0,6.4) node[circle,draw,fill=black, scale=.25]{};
		\draw (8.8,0,5.6) node[circle,draw,fill=black, scale=.25]{};
		\draw (7.6,0,3.73) node[circle,draw,fill=black, scale=.25]{};
		\draw (6.8,0,2.93) node[circle,draw,fill=black, scale=.25]{};
		\draw (6.4,0,2.57) node[circle,draw,fill=black, scale=.25]{};
		\draw (6,0,2) node[circle,draw,fill=black, scale=.25]{};
		\draw (9.95,0,6.09) node[circle,draw,fill=black, scale=.25]{};
		\draw[very thick, gray, opacity=.5] (9.5,0,5.6)--(9.7,0,5.6);
		\draw[very thick, gray, opacity=.5] (9.5,0,5.94)--(9.7,0,5.94);
		\draw[very thick, gray, opacity=.8] (10,0,5.94)--(10.2,0,5.94);
		\draw[very thick, gray, opacity=.5] (9.6,0,6.4)--(9.6,0,.66);
		\draw (9.6,0,.66) node[below] {$C_{\eta_k\gamma}$};
		\draw[very thick, gray, opacity=.5] (9.5,0,.66)--(9.7,0,.66);
		\draw[very thick, gray, opacity=.5] (9.5,0,1.5)--(9.7,0,1.5);
		\draw[very thick, gray, opacity=.5] (9.5,0,1.75)--(9.7,0,1.75);
		\draw[very thick, gray, opacity=.5] (8.8,0,5.6)--(8.8,0,1.2);
		\draw (8.8,0,1.2) node[below] {$C_{\eta_j\gamma}$};
		\draw (.2,0,4.6) node[right] {$\mathrm{Tr}(i,k)$};
		\draw (4.1,0,4.6) node[right] {$\pi\big[\mathrm{Tr}(i,k)\big]=\mathrm{Tr}^{\gamma}(\eta_i,\eta_k)$};
		\draw[very thick, gray, opacity=.5] (8.7,0,1.2)--(8.9,0,1.2);
		\draw[very thick, gray, opacity=.5] (8.7,0,1.5)--(8.9,0,1.5);
		\draw[very thick, gray, opacity=.5] (8.7,0,3.73)--(8.9,0,3.73);
		\draw[very thick, gray, opacity=.5] (8.7,0,4.1)--(8.9,0,4.1);
		\draw[very thick, gray, opacity=.5] (8.7,0,4.35)--(8.9,0,4.35);
		\draw[very thick, gray, opacity=.5] (8.7,0,5.1)--(8.9,0,5.1);
		\draw[very thick, gray, opacity=.5] (7.6,0,3.73)--(7.6,0,1.6);
		\draw[very thick, gray, opacity=.5] (7.5,0,1.6)--(7.7,0,1.6);
		\draw[very thick, gray, opacity=.5] (7.5,0,2.93)--(7.7,0,2.93);
		\draw[very thick, gray, opacity=.5] (7.5,0,3.33)--(7.7,0,3.33);
		\draw[very thick, gray, opacity=.5] (6.8,0,2.93)--(6.8,0,2.57);
		\draw[very thick, gray, opacity=.5] (6.7,0,2.57)--(6.9,0,2.57);
		\draw[very thick, gray, opacity=.5] (6.4,0,2.57)--(6.4,0,1.33);
		\draw[very thick, gray, opacity=.5] (6.3,0,1.33)--(6.5,0,1.33);
		\draw[very thick, gray, opacity=.5] (6.3,0,2)--(6.5,0,2);
		\draw[very thick, gray, opacity=.5] (6.3,0,1.6)--(6.5,0,1.6);
				\draw[very thick, gray, opacity=.8] (10,0,.66)--(10.2,0,.66);
		\draw[very thick, gray, opacity=.8] (10,0,1.5)--(10.2,0,1.5);
		\draw[very thick, gray, opacity=.8] (10,0,1.75)--(10.2,0,1.75);
		\draw[very thick, gray, opacity=.8] (10,0,1.2)--(10.2,0,1.2);
		\draw[very thick, gray, opacity=.8] (10,0,3.73)--(10.2,0,3.73);
		\draw[very thick, gray, opacity=.8] (10,0,4.1)--(10.2,0,4.1);
		\draw[very thick, gray, opacity=.8] (10,0,4.35)--(10.2,0,4.35);
		\draw[very thick, gray, opacity=.8] (10,0,5.1)--(10.2,0,5.1);
		\draw[very thick, gray, opacity=.8] (10,0,1.6)--(10.2,0,1.6);
		\draw[very thick, gray, opacity=.8] (10,0,2.93)--(10.2,0,2.93);
		\draw[very thick, gray, opacity=.8] (10,0,3.33)--(10.2,0,3.33);
		\draw[very thick, gray, opacity=.8] (10.2,0,2.57)--(10,0,2.57);
		\draw[very thick, gray, opacity=.8] (10.2,0,1.33)--(10,0,1.33);
		\draw[very thick, gray, opacity=.8] (10,0,2)--(10.2,0,2);
	\draw[double arrow=2pt colored by black and white]
(3.7,0,5.6) -- node[midway,anchor=center,fill=white,draw=black,line width=.2mm,inner sep=3pt, rounded corners=.5mm]{$\pi:\omega_1\to C_\gamma$} (10.07,0,7.33);
\draw[->,dashdotted] (2.93,0,.24)--(9.23,0,.41);
\draw[->,dashdotted] (2.34,0,.65)--(8.42,0,.93);
\end{tikzpicture}
\caption{The order-isomorphism $\pi:\omega_1\to C_\gamma:k\mapsto\eta_k$ translating a $C$-sequence on $\omega_1$, and thereby a walk, to one on $C_\gamma$. On the left is the standard picture of a walk determined by a $C$-sequence (drawn in gray; notches depict representative elements of the associated $C_j$) on $\omega_1$; we term the walk on the right-hand side \emph{$C_\gamma$-internal}. For initial inputs $\beta\notin C_\gamma$ such walks require a first step ``up into'' $C_\gamma$; this we have depicted as well.}
\label{thepifigure}
\end{figure}

If for $\beta\in C_\gamma$ we let $C_{\beta\gamma}=\pi[C_{\pi^{-1}(\beta)}]$ then we may define the \emph{$C_\gamma$-internal walk} more directly: for $\beta\in C_\gamma$ we have
\begin{align}
\label{eq:internal_trace}
\mathrm{Tr}^\gamma(\alpha,\beta)=\{\beta\}\,\cup\,\mathrm{Tr}^\gamma(\alpha,\min(C_{\beta\gamma}\backslash\alpha)).\end{align}
To ground the recursion, let $\mathrm{Tr}^\gamma(\alpha,\min(C_\gamma\backslash\alpha))=\{\min(C_\gamma\backslash\alpha)\}$ for all $\alpha<\gamma$. (Similarly for $\beta\not\in C_\gamma$, though the notation grows cluttered.) Rather than mapping a walk on the countable ordinals to one on those of $C_\gamma$, this second framing maps the underlying $C$-sequence on the countable ordinals to the ordinals of $C_\gamma$, and walks thereon. See Figure \ref{thepifigure}. In one view, these \emph{internal walks} are the material of walks of the next higher order, as we'll see in Sections \ref{sect:basics} through \ref{sect:higher_coherence} below.

Note lastly that whether or not $C_{\beta\gamma}$ is in fact the $\pi$-image of $C_{\pi^{-1}(\beta)}$ is irrelevant to the functioning of equation \ref{eq:internal_trace}. One may, in other words, define $C_\gamma$-internal walks still more directly yet, simply by choosing cofinal subsets $C_{\beta\gamma}$ of $\beta\cap C_\gamma$ for each $\beta\in C_\gamma$; with this observation we arrive to general notion of a higher $C$-sequence.

\subsection{Higher $C$-sequences}
\label{subsect:Higher_C}
The disjoint union in the following definition merits a preliminary comment: the point is simply that although some $C_\alpha$ and $C_{\alpha\beta}$, for example, may denote identical subsets of $\alpha$, we will wish to continue to regard $C_\alpha$ and $C_{\alpha\beta}$ as distinct objects.
\begin{definition}
\label{def:higherCseqs}
For any nonzero $n\leq\omega$ and ordinal $\gamma$, an \emph{order-$n$ $C$-sequence on $\gamma$} is a family $\mathcal{C}=\bigsqcup_{i\in n}\mathcal{C}_i$ satisfying the following condition. To state it, first formally distinguish between $C_0$ and $C_\varnothing$ and let $C_0=\varnothing$, $C_\varnothing=\gamma$, and $\mathcal{C}_{-1}=\langle C_\varnothing\rangle$; the condition then is that for all 
$i\in n$, each element $C_{\alpha\vec{\beta}}$ of the family
$$\mathcal{C}_i=\langle C_{\alpha\vec{\beta}}\mid\alpha\in C_{\vec{\beta}}\textnormal{ and }\vec{\beta}\textnormal{ is among the indices of }\mathcal{C}_{i-1}\rangle$$
is a closed and cofinal subset of $\alpha\cap C_{\vec{\beta}}$.
We call the indices $\vec{\beta}$ appearing in $\mathcal{C}$ the \emph{$\mathcal{C}$-indices}, for short; note that these comprise a subset of $[\gamma]^{\leq n}$.
\end{definition}

Observe that an \emph{order-1} $C$-sequence is simply a classical one.
Conversely, classical $C$-sequences generate higher order $C$-sequences, as in the preceding subsection and the following example.

\begin{example}
\label{ex:compounding}
Let $\mathcal{D}=\langle C_\alpha\mid\alpha\in\gamma\rangle$ be a classical $C$-sequence, with order-isomorphisms $\pi_\alpha:\mathrm{otp}(C_\alpha)\to C_\alpha$ for each $\alpha\in\gamma$.
As noted, letting $\mathcal{C}_0=\mathcal{D}$ defines an order-1 $C$-sequence; just as in Section \ref{subsect:internal_walks}, letting $C_{\alpha\beta}=\pi_\beta\big[C_{\pi_\beta^{-1}(\alpha)}\big]$ for each $\beta\in\gamma$ and $\alpha\in C_\beta$ then defines a compatible $\mathcal{C}_1$ and, hence, an \emph{order-2} $C$-sequence as in Definition \ref{def:higherCseqs}.
We term this procedure a \emph{compounding of $\mathcal{D}$}; by iterating it, next letting $C_{\alpha\beta\gamma}=\pi_\gamma\big[C_{\pi_\gamma^{-1}(\alpha)\pi_\gamma^{-1}(\beta)}\big]$ for any $C_{\beta\gamma}\in\mathcal{C}_1$ and $\alpha\in C_{\beta\gamma}$, for example, one defines $C$-sequences $\mathcal{C}$ on $\gamma$ of arbitrary orders $n>0$.
Note that the $\mathcal{C}$ obtained in this fashion might only vacuously involve orders above some $j$: if, for example, the elements of $\mathcal{D}=\langle C_\alpha\mid\alpha\in\omega_1\rangle$ are all of minimal ordertype then $\mathcal{C}_1$ will contain no infinite elements, $\mathcal{C}_2$ will contain no nonempty elements, and $\mathcal{C}_i$, consequently, will be empty for every $i>2$.
We discuss notions of \emph{depth} for $C$-sequences in Section \ref{subsect:depth} below; here one natural notion of depth for classical sequences $\mathcal{D}$ suggests itself, namely the supremum (possibly $\infty$) of the indices $i$ of its nonempty compoundings $\mathcal{C}_i$.
Note that the depth, in this sense, of any ordertype-minimal $C$-sequence $\mathcal{D}$ on $\omega_n$ $(n\in\omega)$ is exactly $n+1$.
\end{example} 

The higher $C$-sequences structuring \emph{The first omega alephs} \cite{TFOA} are all instances of Example \ref{ex:compounding}, subject to two additional constraints:
\begin{enumerate}
\item each element of $\mathcal{D}$, and hence each element of each of its compoundings, is of minimal possible ordertype;
\item for any limit ordinal $\alpha$, the cofinality of the $\xi^{\mathrm{th}}$ element of $C_\alpha$ is $\mathrm{cf}(\xi)$ (as then holds for any $C_{\alpha\vec{\beta}}$ as well).
\end{enumerate}

These lend the induced higher $C$-sequences uniformities convenient for the applications appearing therein, and it is for this reason that the sequences $\mathcal{C}$ which have received the most sustained attention to date have tended to be of this form (see \cite[\S 4]{TFOA} for their fundamentals).
Such uniformities, on the other hand, are conceivably counterproductive when complexity is the aim, as with the nontrivial $n$-coherence and strong colorings, respectively, of Sections \ref{sect:nontriviality} and \ref{sect:further} below; we should therefore stress that Definition \ref{def:higherCseqs} accommodates a much wider range of possibilities.
$\square(\kappa)$ sequences, for example, generate higher order $C$-sequences possessing uniformities of structure rather different from those of Example \ref{ex:compounding}:

\begin{example}
\label{ex:highCfromsquare}
Let $\mathcal{D}$ be a $C$-sequence on a cardinal $\kappa$ satisfying item (1) of Definition \ref{def:square} and again let $\mathcal{C}_0=\mathcal{D}$.
A natural choice of $C_{\alpha\beta}\subseteq C_\beta$ for any $\alpha\in\mathrm{acc}(C_\beta)$ and $\beta<\kappa$ then presents itself, namely $C_{\alpha\beta}=C_\alpha$. Note that $C_{\alpha\beta}$ admits a canonical choice when $\alpha\in C_\beta\backslash\mathrm{acc}(C_\beta)$ as well, namely $\{\max C_\beta\cap\alpha\}$.
This approach is iterable, and in contrast with the compounding technique of Example \ref{ex:compounding}, will derive higher $C$-sequences with nontrivial $\mathcal{C}_i$ for arbitrarily large $i$ from, for example, a $\square(\omega_2)$-sequence $\mathcal{D}$.
\end{example}
Rich as these two classes of examples are, each might nevertheless be regarded as fundamentally ``one-dimensional'': each is, informationally, fully present already at the level of $\mathcal{C}_0$. Higher $C$-sequences which are not have yet to be seriously explored.

In what follows, a \emph{$C$-sequence} should always be read as at least potentially a higher order one; those which are not, we will term \emph{classical}. Similarly, we extend the epithet \emph{ordertype-minimal} to apply to any $C$-sequence in which $\mathrm{otp}(C_{\vec{\alpha}})=\mathrm{cf}(\alpha_0\cap C_{\vec{\alpha}^0})$ wherever $C_{\vec{\alpha}}$ is defined (here $C_\varnothing$ continues to equal $\gamma$, the domain of the $C$-sequence).
A \emph{compound $C$-sequence} will denote one deriving as in Example \ref{ex:compounding} from a classical one; as indicated, sequences of this sort, particularly when coupled with items (1) and (2) above, often make for good and simplifying test inputs.
Note, for example, that item (2) ensures that if $C_{\vec{\alpha}}$ is defined and $\alpha_0$ is a limit ordinal then $\mathrm{sup}(C_{\vec{\alpha}})=\alpha_0$. 
Note further that in $C$-sequences $\mathcal{C}$ of any sort, $C_{\vec{\alpha}}\subseteq C_{\vec{\beta}}$ whenever $\vec{\beta}$ is a tail of $\vec{\alpha}$ and $C_{\vec{\alpha}}$ (and hence $C_{\vec{\beta}}$ also) is defined; note lastly that $C_{\vec{\beta}}\backslash\alpha=\varnothing$ for some $\mathcal{C}$-index $\vec{\beta}$ and $\alpha\leq\beta_0$ if and only if $\beta_0=\mathrm{min}\,C_{\vec{\beta}^0}\backslash\alpha$.

\section{The basic form of higher walks}
\label{sect:basics}

Section \ref{sect:Trn} records our generalization $\mathrm{Tr}_n$ of the upper trace function $\mathrm{Tr}$.
It is natural to precede its consideration with a discussion of its simplified variant $\mathrm{tr}_n$, and this is the purpose of the present section.

Fix a $C$-sequence $\mathcal{C}$ on an ordinal $\delta$ and consider a nondecreasing $(n+1)$-tuple $\vec{\gamma}\in \delta^{[n+1]}$ for any $n>0$.
There then exists a $\mathcal{C}$-index $\tau(\vec{\gamma})$ of maximal length such that $\vec{\gamma}=\iota(\vec{\gamma})^\frown\tau(\vec{\gamma})$ for some nonempty $\iota(\vec{\gamma})=(\gamma_0,\dots,\gamma_j)$. We term this $\tau(\vec{\gamma})$ the \emph{$\mathcal{C}$-maximal proper tail} of $\vec{\gamma}$; these sequences, together with the associated value $\mathrm{min}(C_{\tau(\vec{\gamma})}\backslash\gamma_j)$ (if defined), are the basic ingredients of $\mathrm{tr}_n$.

\begin{definition}
\label{def:tr_n}
With respect to an order-$n$ $C$-sequence $\mathcal{C}$ on an ordinal $\delta$, for any $n>0$ the function $\mathrm{tr}_n$ is recursively defined as follows. For any $\vec{\gamma}\in \delta^{[n+1]}$, let $\iota(\vec{\gamma})$ and $\tau(\vec{\gamma})$ be as above; more precisely, let $\tau(\vec{\gamma})$ be the $\mathcal{C}$-maximal proper tail of $\vec{\gamma}$ and let $j=n-|\tau(\vec{\gamma})|$. If $C_{\tau(\vec{\gamma})}\backslash\gamma_j\neq\varnothing$ then
\begin{align*}
\mathrm{tr}_n(\vec{\gamma})=\{\mathrm{min}(C_{\tau(\vec{\gamma})}\backslash\gamma_j)\}\sqcup\bigsqcup_{i\in\{1,\dots,n+1\}\backslash\{j+1\}}\mathrm{tr}_n\big((\iota(\vec{\gamma}),\mathrm{min}(C_{\tau(\vec{\gamma})}\backslash\gamma_j),\tau(\vec{\gamma}))^i\big).
\end{align*}
Grounding the recursion is the boundary condition that $\mathrm{tr}_n(\vec{\gamma})=\varnothing$ if $C_{\tau(\vec{\gamma})}\backslash\gamma_j=\varnothing$. $\mathrm{tr}_n(\vec{\gamma})$ collects a family of ordinals which we term its \emph{outputs}; similarly, we call the $(n+1)$-tuples appearing in the course of its expansion its \emph{inputs}.
\end{definition}
As in Section \ref{subsect:Higher_C}, the disjoint union addresses the possibility of output repetition, a potential issue in our second, but not our first, example.
\begin{example}
\label{ex:tr1}
When $n=1$, so that $\vec{\gamma}=(\gamma_0,\gamma_1)$, there is only one possibility for $\tau(\vec{\gamma})$, namely $(\gamma_1)$. Hence for $\alpha<\beta$, Definition \ref{def:tr_n} specializes to
$$\mathrm{tr}_1(\alpha,\beta)=\{\mathrm{min}(C_\beta\backslash\alpha)\}\sqcup\mathrm{tr}_1(\alpha,\mathrm{min}(C_\beta\backslash\alpha)),$$
a sequence almost identical to $\mathrm{Tr}(\alpha,\beta)$ (compare equation \ref{eq:Tr}); more precisely,
$$\mathrm{tr}_1(\alpha,\beta)=\mathrm{Tr}(\alpha,\beta)\backslash\{\beta\}$$
for any $\alpha\leq\beta$. Absolute identity could, of course, be easily arranged, via revision of the boundary condition and of the leading output of $\mathrm{tr}_n(\vec{\gamma})$ to $\{\gamma_{j+1}\}$.
We've opted instead for outputs that tend, for higher $n$, to be more informative, but should stress that we regard such modifications as ultimately superficial,\footnote{A more mathematical rendering of this view is the observation that the families of $\rho_2$-fiber maps associating to any such modification are all cohomologous.} in the sense that any of these variations carries a roughly equal title to the name of \emph{higher walks}. We return to this matter of alternative formulations in Section \ref{subsect:alternative} below; at the risk of repetition, what counts in any of them is that they satisfy as much of (n.1) through (n.6) as possible, and that they do so with some of the versatility and grace of classical walks.
\end{example}

\begin{example}
The $n=2$ instance of Definition \ref{def:tr_n} unpacks as follows: for all $\alpha\leq\beta\leq\gamma<\delta$,
\begin{align*}
\mathrm{tr}_2(\alpha,\beta,\gamma)=
\begin{cases} 
      \{\min(C_\gamma\backslash\beta)\}\, \sqcup\, \mathrm{tr}_2(\alpha,\beta,\text{min}(C_\gamma\backslash\beta)) & \\ \hspace{.3 cm} \sqcup\: \mathrm{tr}_2(\alpha,\text{min}(C_\gamma\backslash\beta),\gamma) & \text{if }\beta\not\in C_\gamma \\
      & \\
\{\min(C_{\beta\gamma}\backslash\alpha)\} \,\sqcup\, \mathrm{tr}_2(\alpha,\text{min}(C_{\beta\gamma}\backslash\alpha),\beta) & \\ \hspace{.3 cm} \sqcup\: \mathrm{tr}_2(\alpha,\text{min}(C_{\beta\gamma}\backslash\alpha),\gamma) & \text{if }\beta\in C_\gamma, \\
   \end{cases}
\end{align*}
together with the following boundary conditions:
\begin{itemize}
\item if $\beta=\gamma$ then $\mathrm{tr}_2(\alpha,\beta,\gamma)=\varnothing$;
\item if $\beta\in C_\gamma$ and $C_{\beta\gamma}\backslash\alpha=\varnothing$ then $\mathrm{tr}_2(\alpha,\beta,\gamma)=\varnothing$.
\end{itemize}
Observe that any $\mathrm{tr}_2(\alpha,\beta,\gamma)$ is naturally viewed as a binary tree, one in fact generalizing the $1$-branching tree (i.e., the walk) associated to $\mathrm{tr}_1(\alpha,\beta)$ or $\mathrm{Tr}(\alpha,\beta)$. Depicted in Figure \ref{asampletr2} are the first two levels of the tree associated to $\mathrm{tr}_2(\alpha,\beta,\gamma)$ under generic assumptions on $\alpha$, $\beta$, and $\gamma$. The nodes of this tree are labeled with two sorts of data: as with the classical $\mathrm{Tr}$, the recursively defined function $\mathrm{tr}_2$ records an ordinal (appearing in the lower half of a node) then proceeds to new inputs (appearing in the top halves of successor nodes); each of these is displayed in Figure \ref{asampletr2} above. This is because unlike in the classical case, the collection of ordinals output and the collection of tuples input in the course of a higher-dimensional walk are no longer informationally equivalent; while the former may be better suited to combinatorial applications, deductions \emph{per se} may require the fuller data of the latter.

We now describe a few fundamental features of the function $\mathrm{tr}_2$, framing our discussion in terms of Figure \ref{asampletr2}. Observe that in the passage from any node to one directly below, exactly one element of the associated coordinate-triple is replaced. Observe also that this element is never the least one, $\alpha$. Descending along the rightmost branch, for example, it is the second coordinate that is always changing; observe that its pattern $\beta$, $\beta_\varnothing$, $\beta_1,\dots$ is that of the $C_\gamma$-internal walk from $\beta$ down to $\alpha$ (see Section \ref{subsect:internal_walks}). Along the leftmost branch, on the other hand, it is the third coordinate that is in motion; its pattern is visibly that of the classical walk from $\gamma$ down to $\beta$. Here it is natural to term the walks associated to leftwards paths through the tree \emph{external}. Any $\mathrm{tr}_2(\alpha,\beta,\gamma)$ may then be regarded as a structured family either of internal walks or of external walks, insofar as any binary tree is, as a set, simply the union either of its rightwards or leftwards branches.

Here a word of clarification is in order. While rightwards paths through $\mathrm{tr}_2(\alpha,\beta,\gamma)$ correspond precisely to internal walks, leftwards paths may properly contain classical walks in the following way: let $\beta'=\min (\mathrm{Tr}(\beta,\gamma)\backslash\{\beta\})$. If $\alpha'=\min(C_{\beta\beta'}\backslash\alpha)$ is defined then in the node below and to the left of $(\alpha,\beta,\beta')$ is $(\alpha,\alpha',\beta)$. The leftwards path out of this node will then describe a classical walk from $\beta$ down to $\alpha'$, possibly again initiating a further walk out of $\alpha'$ upon arrival, and so on. External walks correspond in this way to iterated descending chains of classical walks (``walks of walks'').
\end{example}

\begin{figure}
\centering
\begin{tikzpicture}[MyPersp]

\node[rectangle split,rectangle split parts=2,draw=black,line width=.2mm,inner sep=3pt, rounded corners=.5mm,text centered](A) at (4,0,7) {$(\alpha,\beta,\gamma)$ \nodepart{second} $\beta_{\varnothing}:=\min(C_\gamma\backslash\beta)$};

\node[rectangle split,rectangle split parts=2,draw=black,line width=.2mm,inner sep=3pt, rounded corners=.5mm,text centered](B) at (1,0,5) {$(\alpha,\beta,\beta_{\varnothing})$ \nodepart{second} $\beta_0:=\min(C_{\beta_{\varnothing}}\backslash\beta)$};

\node[rectangle split,rectangle split parts=2,draw=black,line width=.2mm,inner sep=3pt, rounded corners=.5mm,text centered](C) at (7,0,5) {$(\alpha,\beta_{\varnothing},\gamma)$ \nodepart{second} $\beta_1:=\min(C_{\beta_{\varnothing}\gamma}\backslash\alpha)$};

\node[rectangle split,rectangle split parts=2,draw=black,line width=.2mm,inner sep=3pt, rounded corners=.5mm,text centered](D) at (-.5,0,3) {$(\alpha,\beta,\beta_{0})$ \nodepart{second} $\beta_{00}$};

\node[rectangle split,rectangle split parts=2,draw=black,line width=.2mm,inner sep=3pt, rounded corners=.5mm,text centered](E) at (2.5,0,3) {$(\alpha,\beta_0,\beta_{\varnothing})$ \nodepart{second} $\beta_{01}$};

\node[rectangle split,rectangle split parts=2,draw=black,line width=.2mm,inner sep=3pt, rounded corners=.5mm,text centered](F) at (5.5,0,3) {$(\alpha,\beta_1,\beta_{\varnothing})$ \nodepart{second} $\beta_{10}$};

\node[rectangle split,rectangle split parts=2,draw=black,line width=.2mm,inner sep=3pt, rounded corners=.5mm,text centered](G) at (8.5,0,3) {$(\alpha,\beta_1,\gamma)$ \nodepart{second} $\beta_{11}$};

\draw (A) -- (B);
\draw (A) -- (C);
\draw (B) -- (D);
\draw (B) -- (E);
\draw (C) -- (F);
\draw (C) -- (G);
\end{tikzpicture}
\caption{First steps of $\mathrm{tr}_2(\alpha,\beta,\gamma)$. Associated to any $\mathrm{tr}_2$-input such as $(\alpha,\beta,\gamma)$ are an ordinal output, like $\beta_{\varnothing}$, and two further $\mathrm{tr}_2$-inputs, like $(\alpha,\beta,\beta_{\varnothing})$ and $(\alpha,\beta_{\varnothing},\gamma)$. Shaping the above diagram are the assumptions (made somewhat at random, simply for concreteness) that $\beta<\gamma$ and the outputs $\beta_\varnothing$, $\beta_0$, and $\beta_1$ defined as above are each meaningful and not equal to $\beta$. Lower outputs $\beta_\sigma$ may well correspond to undefined expressions, in which case they should be regarded as the empty set, marking the end of a branch.}
\label{asampletr2}
\end{figure}
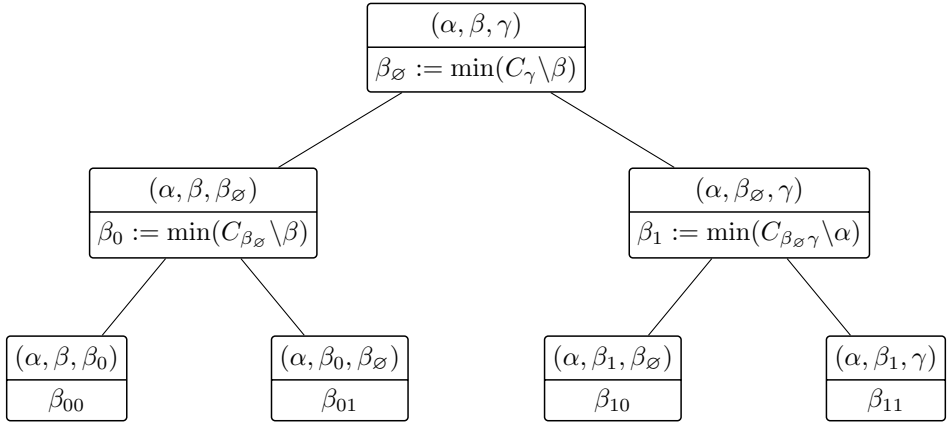
This view of $\mathrm{tr}_2$ as a binary tree generalizes: each $\mathrm{tr}_n$ sends $(n+1)$-tuples of ordinals to $n$-branching trees, and we will see in a moment that these trees are all finite. Just as above, ``hyperplanes'' through these trees (i.e., subtrees of smaller branching number determined by some fixed rule of descent) correspond to lower-order walks-structures and their relativizations. To give something of the flavor of these generalizations, the first step of such a $\mathrm{tr}_3(\alpha,\beta,\gamma,\delta)$ with respect to an order-3 $C$-sequence on $\varepsilon$ will depend on whether $\gamma\in C_\delta$ and, if so, on whether $\beta\in C_{\gamma\delta}$ and, if so, on the value, if defined, of $\min(C_{\beta\gamma\delta}\backslash\alpha)$. Relatedly, the $1$-branching walks comprising such higher $\mathrm{tr}_n$ will, in contrast to those of $\mathrm{tr}_2$, exhibit \emph{varying degrees of} internality.

\begin{theorem}\label{thm:finite} With respect to any $n>0$ and order-$n$ $C$-sequence on an ordinal $\delta$, the set $\mathrm{tr}_n(\vec{\gamma})$ is finite for any $\vec{\gamma}\in\delta^{[n+1]}$.
\end{theorem}

\begin{proof}
The argument's basic idea is simple: since any such $\mathrm{tr}_n(\vec{\gamma})$ is finitely branching, if it were infinite then it would have an infinite branch --- but this would entail, just as in the classical $n=1$ case, a strictly decreasing infinite sequence of ordinals, a contradiction.
Rigorously arguing this requires some care, however, because steps within higher walks can raise inputs' coordinate entries, as we'll discuss in greater detail below.
This prospect manifests in the first rightwards step of Figure \ref{asampletr2}, for example; to better gauge its extent, readers may find it edifying (and perhaps sufficient, in a first reading) to check the $n\leq 2$ cases of the basic idea by hand.

For more general $n$, see again Definition \ref{def:tr_n} and note that, in the view of $\mathrm{tr}_n(\vec{\gamma})$ as a tree of nodes labeled with the associated inputs $\vec{\beta}$, the immediate successors of any $\vec{\beta}$ are all of two sorts: $(\iota(\vec{\beta}),\mathrm{min}(C_{\tau(\vec{\beta})}\backslash\beta_j),\tau(\vec{\beta}))^i$ for some $i\leq j$, and for some $i>j+1$; note also that the second of these is the only possibility when $j=0$.
Hence any sequence of steps below $\vec{\beta}$ is the product of a sequence of choices $i_k$ of such $i$. Alongside this is a sequence of $j_k$ recording the $j$ in the above expression at step $k$.
These data determine a sequence of $\vec{\beta}^k=(\beta^k_0,\dots,\beta^k_n)$ in which $\vec{\beta}^0=\vec{\beta}$ and $\vec{\beta}^{k+1}$ is the successor of $\vec{\beta}^k$ associated to $i_k$; we will assume some such sequence infinite --- in other words, indexed by $k\in\omega$ --- and arrive to a contradiction by the following series of claims.
\begin{claim}
\label{clm:1}
If $i\leq j$ and $\mathrm{tr}_n((\iota(\vec{\beta}),\mathrm{min}(C_{\tau(\vec{\beta})}\backslash\beta_j),\tau(\vec{\beta}))^i)\neq\varnothing$ then $i=j$. Moreover, for any $k\in\omega$,
\begin{itemize}
\item if $j_k=i_k$ then $j_{k+1}<i_k$, while
\item if $j_k+1<i_k$ then $j_{k+1}<i_k-1$.
\end{itemize}
\end{claim}
\begin{proof}
If $i<j$ then $\tau((\iota(\vec{\beta}),\mathrm{min}(C_{\tau(\vec{\beta})}\backslash\beta_j),\tau(\vec{\beta}))^i)=(\mathrm{min}(C_{\tau(\vec{\beta})}\backslash\beta_j),\tau(\vec{\beta}))$, and $$C_{(\mathrm{min}(C_{\tau(\vec{\beta})}\backslash\beta_j),\tau(\vec{\beta}))}\backslash\beta_j$$
by definition equals $\varnothing$. This shows the first assertion. Similarly, if $i=j$ then $(\mathrm{min}(C_{\tau(\vec{\beta})}\backslash\beta_j),\tau(\vec{\beta}))$ forms a tail of $\tau((\iota(\vec{\beta}),\mathrm{min}(C_{\tau(\vec{\beta})}\backslash\beta_j),\tau(\vec{\beta}))^i)$; this implies the second assertion, and identical reasoning establishes the third.
\end{proof}
Observe that if $j_k=i_k$ then $\beta^{k+1}_i=\beta^k_i$ for all $i\neq i_k$, while $\beta^{k+1}_{i_k}>\beta^k_{i_k}$; in plainer English, the passage from $\vec{\beta}^k$ to $\vec{\beta}^{k+1}$ holds all coordinates constant except one, which it raises. It's this sort of conversion which renders our theorem's argument nontrivial.
If, on the other hand, $j_k+1<i_k$ then the passage from $\vec{\beta}^k$ to $\vec{\beta}^{k+1}$ strictly lowers an interval of coordinates while holding the others constant. More precisely:
\begin{claim}
\label{clm:2}
If $j_k< i< i_k$ then $\beta^{k+1}_i<\beta^k_i$; moreover, if there exists such an $i$ then $\beta^{k+1}_\ell=\beta^k_\ell$ for all $\ell$ outside the interval $(j_k,i_k)$.
\end{claim}
\begin{proof}
Since $\tau(\vec{\beta}^k)\in [\delta]^{\leq n}$ and the effect of omitting coordinate $i_k>j_k+1$ from $$(\iota(\vec{\beta}^k),\mathrm{min}(C_{\tau(\vec{\beta}^k)}\backslash\beta_{j_k}),\tau(\vec{\beta}^k))$$ is to replace
\begin{itemize}
\item $\beta^k_{j_k+1}$ with $\mathrm{min}(C_{\tau(\vec{\beta}^k)}\backslash\beta^k_{j_k})$, and
\item $\beta^k_i$ with $\beta^k_{i-1}$ for each $i$ strictly between $j_k+1$ and $i_k$,
\end{itemize}
the claim holds.
\end{proof}

\begin{claim}
\label{clm:2.1}
For any $\ell\in\omega$ there exists a $k\in [\ell,\ell+n)$ such that $j_k+1<i_k$.
\end{claim}
\begin{proof}
Suppose otherwise and swiftly attain, via Claim \ref{clm:1}, a contradiction: we must have $j_k=i_k$ for all $k\in [\ell,\ell+n)$, but this entails a length-$n$ strictly decreasing sequence of $j_k\in n$, hence $j_{\ell+n-1}=0\neq i_{\ell+n-1}$.
\end{proof}

By Claim \ref{clm:2.1}, there exists a maximal $i\leq n$ for which $j_k+1<i_k=i$ for all $k$ in some infinite $B\subseteq\omega$.
Thus there exists an $m\in B$ such that $i_k\leq i$ for all $k\geq m$.
It then follows from Claim \ref{clm:1} that $j_k<i-1$ for all $k\geq m$, and hence from Claim \ref{clm:2} and the analysis immediately preceding it that $\langle\beta^k_{i-1}\mid k\in B\backslash m\rangle$ is a strictly decreasing sequence of ordinals. This is our desired contradiction.
\end{proof}

Observe in conclusion that the functions $\mathrm{tr}_n$ satisfy the criteria (n.1) and (n.2) of Section \ref{subsect:generalizations}.

\section{$\mathrm{Tr}_n$ and its most basic associated functions}
\label{sect:Trn}

We turn now to our proposed generalizations of the upper trace function $\mathrm{Tr}$.
These $\mathrm{Tr}_n$ enrich the functions $\mathrm{tr}_n$ with the data of \emph{signs}, recorded as powers of $-1$, as follows:
\begin{definition}
\label{def:Tr_n}
With respect to an order-$n$ $C$-sequence $\mathcal{C}$ on an ordinal $\delta$, for any $n>0$ the \emph{order-$n$ upper trace function} $\mathrm{Tr}_n$ is recursively defined by letting
\begin{align*}
\label{eq:Tr_n}
\mathrm{Tr}_n((-1)^k,\vec{\gamma}) = & \{((-1)^{j+k},\mathrm{min}(C_{\tau(\vec{\gamma}))}\backslash\gamma_j)\}\,\sqcup \\ & \bigsqcup_{i\in\{1,\dots,n+1\}\backslash\{j+1\}}\mathrm{Tr}_n\big((-1)^{i+j+k},(\iota(\vec{\gamma}),\mathrm{min}(C_{\tau(\vec{\gamma})}\backslash\gamma_j),\tau(\vec{\gamma}))^i\big).
\end{align*}
for any $\vec{\gamma}\in \delta^{[n+1]}$, where $j$, $\tau(\vec{\gamma})$, and $\iota(\vec{\gamma})$ are all as in Definition \ref{def:tr_n}.
Grounding the recursion is the boundary condition that $\mathrm{Tr}_n((-1)^k,\vec{\gamma})=\varnothing$ if $C_{\tau(\vec{\gamma})}\backslash\gamma_j=\varnothing$, just as before.
\end{definition}
When convenient, we will record these signs as a simple $+$ or $-$, and we will regard the outputs of $\mathrm{Tr}_n$ as signed ordinals. The utility of these signs isn't immediately obvious, not least because they're unvarying, and hence superfluous, when $n=1$:
$$\mathrm{Tr}_1(\pm,\alpha,\beta)=\{\pm\,\text{min}(C_\beta\backslash\alpha)\}\,\sqcup\,\mathrm{Tr}_1(\pm,\alpha,\text{min}(C_\beta\backslash\alpha)).$$
Put differently, $\mathrm{Tr}_1$ is informationally equivalent to the classical function $\mathrm{Tr}$ (modulo the superficial issue of $\{\beta\}$, discussed in Example \ref{ex:tr1}), as we would hope.

This brings us to a first useful framing: just as subdivision of $n$-simplices is a reasonable heuristic for the iterative processes of $\mathrm{tr}_n$ (i.e., new $(n+1)$-tuples are formed from a new ordinal conjoined with elements of the boundary of the old one, viewed as an abstract simplex), geometric notions of orientation form a natural heuristic for the signs arising in $\mathrm{Tr}_n$; however, as in multivariable calculus or geometry, for example, these signs or orientations only assume their proper significance in settings of more than two coordinates.
It is for this reason we defer a motivational discussion of signs (to which we will repeatedly return) until our experience of higher-dimensional walks is a little richer.

\begin{example}
\label{ex:Tr2}
Consider next the case of $n=2$:
\[ 
\mathrm{Tr}_2(\pm,\alpha,\beta,\gamma)=
\begin{cases} 
      \{\mp\,\text{min}(C_\gamma\backslash\beta)\} \sqcup\, \mathrm{Tr}_2(\pm,\alpha,\beta,\text{min}(C_\gamma\backslash\beta)) & \\ \hspace{.1 cm} \sqcup\, \mathrm{Tr}_2(\pm,\alpha,\text{min}(C_\gamma\backslash\beta),\gamma) & \text{if }\beta\not\in C_\gamma \\
      & \\
\{\pm\,\text{min}(C_{\beta\gamma}\backslash\alpha)\} \sqcup\, \mathrm{Tr}_2(\mp,\alpha,\text{min}(C_{\beta\gamma}\backslash\alpha),\beta) & \\ \hspace{.1 cm} \sqcup\, \mathrm{Tr}_2(\pm,\alpha,\text{min}(C_{\beta\gamma}\backslash\alpha),\gamma) & \text{if }\beta\in C_\gamma, \\
   \end{cases}
\]
with boundary conditions just as for $\mathrm{tr}_2$. 

Signs in this case exhibit useful and interesting organizing effects, as they do within higher $\mathrm{Tr}_n$ more generally. For the duration of the down-and-leftwards movement in $\mathrm{tr}_2$ that we identified above with a classical walk, for example, inputs' and outputs' signs are both constant, until the last step. If here a further walk is initiated, signs flip, remaining constant for its duration, and so on. Similarly, inputs' signs along any rightwards, internal walk of $\mathrm{Tr}_2$ are constant on the branch's full length; outputs' signs are constant after a possible first step ``up into'' the internal walk (see again Figure \ref{thepifigure}, caption). Hence one may speak not only of the signs of nodes, but of the signs of eventually-rightwards branches as well; in fact the latter will form a main focus of Section \ref{subsubsect:pairings} below.
\end{example}
These observations should begin to suggest the number of interesting \emph{characteristics}, ``statistics,'' or rho functions which higher walks admit. One of the simplest of these, a generalization of the function $\rho_2$, was introduced in \cite{dimords}: for nondecreasing $(n+1)$-tuples of ordinals $\vec{\alpha}$, define the function $\rho_2^n(\vec{\alpha})$ as the ``charge'' of $\mathrm{Tr}_n(+,\vec{\alpha})$, i.e., as the ``pluses minus the minuses'':
\begin{align*}
\rho_2^n(\vec{\alpha})= & \textit{ the number of positive signed ordinals in }\mathrm{Tr}_n(+,\vec{\alpha})\\ & - \textit{the number of negative signed ordinals in }\mathrm{Tr}_n(+,\vec{\alpha}).\end{align*}
Observe in particular that $$\rho_2^1(\alpha,\beta)=|\mathrm{Tr}_1(+,\alpha,\beta)|=|\mathrm{Tr}(\alpha,\beta)|-1=\rho_2(\alpha,\beta)$$ for all $\alpha\leq\beta$.

Other extensions of $\rho_2$ suggest themselves, of course.
We might view $\rho_2(\alpha,\beta)$ as recording the tree-type of the $1$-branching tree $\mathrm{Tr}(\alpha,\beta)$, for example; from this perspective, its natural generalization is
\begin{align*}
\rho_{2,\mathrm{t}}^n(\vec{\alpha})= & \textit{ the tree-type of }\mathrm{Tr}_n(+,\vec{\alpha}),\end{align*}
and this indeed is a significant enough function that we accord it its own symbol $\rho_{2,\mathrm{t}}^n$.
More precisely, define $\rho_{2,\mathrm{t}}^n(\vec{\alpha})$ as the subset of $n^{<\omega}$ naturally indexing $\mathrm{Tr}_n(\pm,\vec{\alpha})$, as in the $n=2$ case of Figure \ref{asampletr2} (wherein $\{\varnothing,0,1\}\subseteq\rho_{2,\mathrm{t}}^2(\alpha,\beta,\gamma)\subseteq 2^{<\omega}$).
Note that this function encompasses the most naive generalization of $\rho_2$, namely the function $\vec{\alpha}\mapsto |\rho_{2,\mathrm{t}}^n(\vec{\alpha})|$, as well.
Among these possibilities, however, as we will see in Section \ref{sect:higher_coherence}, $\rho_2^n$ exhibits the strongest $n$-coherence relations, i.e., the best \emph{algebraic reflections of how the fiber maps of $\mathrm{Tr}_n$ interrelate}, and it is for this reason that it has our primary attention.

Observe in conclusion that the functions $\mathrm{Tr}_n$ satisfy criteria (n.1) and (n.2) of Section \ref{subsect:generalizations} because $\mathrm{tr}_n$ does; with the functions $\rho_2^n$ and $\rho_{2,\mathrm{t}}^n$ (we note further rho functions in Section \ref{sect:higher_coherence} below) we have now satisfied item (n.3) as well.

\subsection{Alternative formulations and views}
\label{subsect:alternative}

Whatever virtues the functions $\mathrm{Tr}_n$, $\rho_2^n$, and so on, may prove to have, the question will at some level remain of why it's these formulations and this approach to higher walks that we're pursuing, and not some other.
Let us record a few answers before continuing.

First, we should reiterate that a range of natural variations on our definitions would amount, for our purposes, to more or less the same thing: other choices of boundary conditions or output ordinals, for example, may in some circumstances make better sense than our own, without fundamentally altering our analysis.
Most particularly, the resulting functions will continue to satisfy suitable formulations of the criteria (n.1) through (n.6) of Section \ref{subsect:generalizations}. 
Still, it's conceivable that an apparatus substantially different from ours will as well, returning us to the question of what distinction the forms $\mathrm{tr}_n$ or $\mathrm{Tr}_n$ (or any of a family of minor variations on them)  ultimately carry.

Two main answers suggest themselves.
We alluded to the first of these above, but can be much more precise.
The refinement in \cite{TFOA} of Mitchell's original proof \cite{Mitchell} of what we've recast as Theorem \ref{thm:Mitchell} above may be summarized as follows: \emph{canonically associated to any order-$n$ ordertype-minimal $C$-sequence on $\omega_n$ is a witness to the nonvanishing of the functor $\mathrm{H}^n(\omega_n;\,\cdot\,)$}.\footnote{More precisely: any ordinal $\xi$ possesses a canonical image $\mathbf{\Delta}(\xi)$ in $\mathsf{Ab}^{\xi^{\mathrm{op}}}$, which in turn admits a canonical projective resolution with terms $\mathbf{P}_n(\xi)$; a key component of Mitchell's recognition was that the nonvanishing of the functor $\mathrm{lim}^{n+1}:\mathsf{Ab}^{\omega_n^{\mathrm{op}}}\to\mathsf{Ab}$ possesses a canonical witness, namely $\mathrm{lim}^{n+1}\,\mathbf{P}_{n+1}(\omega_n)$. This recognition was refined in \cite{TFOA} as follows: any ordertype-minimal $C$-sequence on $\omega_n$ induces a canonical basis for $\mathbf{d}_{n+1}\mathbf{P}_{n+1}(\omega_n)$, and hence a section $\mathbf{s}$ of $\mathbf{d}_{n+1}:\mathbf{P}_{n+1}(\omega_n)\to\mathbf{d}_{n+1}\mathbf{P}_{n+1}(\omega_n)$ which does not extend to $\mathbf{P}_n(\omega_n)$; canonically deriving from this $\mathbf{s}$ is the function $\mathtt{f}_n$ which, suitably construed, witnesses the nonvanishing of both $\mathrm{lim}^{n+1}\,\mathbf{P}_{n+1}(\omega_n)$ and $\mathrm{H}^n(\omega_n;\mathcal{A})$ for $A=\mathbb{Z}^{(\omega_n)}$ (see \cite[\S2 -- \S7]{TFOA} for details).}
The data of this witness, which goes by the name $\mathtt{f}_n$ in \cite{TFOA}, is fully captured by a pair of \emph{finite-output} functions, namely $\mathrm{Tr}_n$ together with a lower trace function $\mathrm{L}_n$ introduced in Section \ref{sect:higher_coherence} below.
These functions derive canonically from the functions $\mathtt{f}_n$ and, modulo the superficial modifications addressed above, correspond in the $n=1$ case to the functions $\mathrm{Tr}$ and $\mathrm{L}$ of the classical theory.\footnote{See the first paragraphs of \cite[\S 8.2 and \S 8.3]{TFOA} for more. Note, though, that the the $\mathrm{Tr}_n$ of Definition \ref{def:Tr_n} signs its outputs $(-1)^{j+k}$ rather than $(-1)^{j+k+1}$ and hence corresponds, at the level of outputs, to what in \cite{TFOA} would be denoted $-\mathrm{Tr}_n$ and $-\mathtt{f}_n$ (although the inputs of these two definitions of $\mathrm{Tr}_n$ continue to align). Again the modification is superficial and merely simplifying.}
To sum all this up yet further: in algebraic senses which can be made quite precise, the fundamental forms of what we're proposing as higher walks describe simultaneously canonical and significant invariants of the choice of underlying $C$-sequence.

These considerations touch on a second answer as well. A witness to the nonvanishing of $\mathrm{H}^n(\omega_n;\,\cdot\,)$ is, in one framing, simply a height-$\omega_n$ nontrivial $n$-coherent family of functions. We'll define these terms more precisely in Sections \ref{sect:higher_coherence} and \ref{sect:nontriviality}, but the case of $n=1$ is already familiar: using $\rho_1$, in particular, we have seen that the existence of mod finite nontrivial coherent families on $\omega_1$ is a $\mathsf{ZFC}$ theorem.
Let us try proving this theorem without the aid of walks.
Armed with the observation that any mod finite coherent family $\Phi=\langle\varphi_\beta:\beta\to\mathbb{Z}\mid\beta<\omega_1\rangle$ of \emph{finite-to-one} functions is nontrivial, we might attempt to recursively construct such a $\Phi$; the most natural first attempt is probably the following. At any successor step $\beta=\alpha+1$, let
\begin{equation*}
\varphi_\beta(\xi)=
    \begin{cases}
        0 & \text{if } \xi=\alpha ,\\
        \varphi_\alpha(\xi) & \text{if } \xi <\alpha .
    \end{cases}
\end{equation*}
At any limit step $\beta$, fix an increasing enumeration $\langle\beta_i\mid i\in\omega\rangle$ of a cofinal $C_\beta\subseteq\beta$, and let
\begin{equation*}
\varphi_\beta(\xi)=
    \begin{cases}
        i(\xi) & \text{if } \varphi_{\beta_{i(\xi)}}(\xi) < i(\xi) , \\
        \varphi_{\beta_{i(\xi)}}(\xi) & \text{otherwise,}
    \end{cases}
\end{equation*}
where $i(\xi)=\mathrm{min}\,\{i\in\omega\mid\beta_i>\xi\}$.
This works, as the reader may verify; note also that we've constructed an ordertype-minimal $C$-sequence $\mathcal{C}$ on $\omega_1$ along the way as well.

It is instructive to now compute $\varphi_\beta(\xi)$ for an arbitrary $\xi<\beta<\omega_1$.
If $\beta$ is a successor, e.g., if $\beta=\alpha+1$ for some $\alpha>\xi$ then $\varphi_\beta(\xi)=\varphi_\alpha(\xi)$. If $\alpha$ also is a successor well above $\xi$ then this reduction will repeat, descending through successor ordinals down to either $\xi+1$ (in which case $\varphi_\beta(\xi)=0$) or to a limit ordinal $\eta>\xi$.
In the latter case, $\varphi_\beta(\xi)=\varphi_\eta(\xi)=\max\,\{i(\xi),\varphi_{\eta_{i(\xi)}}(\xi)\}$, where $i(\xi)$ is computed with reference, of course, to $C_\eta$. Reasoning in this fashion, it is not difficult to compute that $\varphi_\beta(\xi)$ is exactly $\rho_1(\xi+1,\beta)$, defined with respect to $\mathcal{C}$; the more fundamental observation is that the steps of this computation amount to a classical walk from $\beta$ to $\xi+1$. And this observation and example suggest, in turn, a more overarching perspective that \emph{if ordinals are the underlying structures of transfinite inductions, then walks are how those inductions unwind.}

This brings us to a second distinction of the higher walks forms we're proposing.
In Section \ref{sect:nontriviality} we'll construct height-$\omega_n$ nontrivial $n$-coherent families of functions $\varphi_{\vec{\beta}}$  in two ways: one will employ higher walks, and the other, as above, will be by transfinite recursion.
The working out, in the latter case, of an arbitrary evaluation $\varphi_{\vec{\beta}}(\xi)$ will have a particular form, and much as above, this form will be recognizably that of what we're calling \emph{higher walks}.

\section{Coherence}
\label{sect:higher_coherence}

For concreteness, we'll focus for most of this section on the modulus of locally constant functions; other forms of $n$-coherence are defined, of course, simply by replacing the modulus in equation \ref{ncoh} below. We'll touch as well on how to turn families $n$-coherent with respect to one modulus into families $n$-coherent with respect to another in Lemma \ref{lem:ntc_conversion}.
\begin{definition}
\label{def:n_coh}
Let $A$ be an abelian group. For any $n>0$, a family of functions $$\Phi=\langle\varphi_{\vec{\alpha}}:\alpha_0\rightarrow A\mid\vec{\alpha}\in [\varepsilon]^{n}\rangle$$ is \emph{$n$-coherent} if 
		\begin{align}
			\label{ncoh}\sum_{i=0}^{n} (\text{-}1)^i\varphi_{\vec{\alpha}^i}= 0\hspace{.8 cm} \textnormal{modulo locally constant functions}
		\end{align}
		for all $\vec{\alpha}\in [\varepsilon]^{n+1}$ (implicit in this equation is the restriction of the function $\varphi_{\vec{\alpha}^0}$ to the domain of the other functions in the sum, i.e., to $\alpha_0$).
		\end{definition}
		
		As indicated, this is a cocycle condition; see \cite[\S 2]{CoOI}.\footnote{A more studiously cohomological perspective might prefer indices in $\varepsilon^{[n]}$ in Definition \ref{def:n_coh}. The additional indices $\vec{\gamma}$ in $\varepsilon^{[n]}\backslash [\varepsilon]^n$ are, however, essentially immaterial to $n$-coherence questions on the ordinals (any $[\varepsilon]^n$-indexed $n$-coherent family extends to an $\varepsilon^{[n]}$-indexed one simply by letting $\varphi_{\vec{\gamma}}$ equal the zero function for all such $\vec{\gamma}$, for example), and ignoring them can streamline arguments, as in Section \ref{subsect:higher_coherence} below.}
Observe that the $n=1$ instance of $n$-coherence coincides with the classical coherence relation
$$\varphi_\beta\big|_\alpha-\varphi_\alpha=0\hspace{.4 cm} \textnormal{modulo locally constant functions}\hspace{1 cm}\textnormal{for all }\alpha\leq\beta<\varepsilon$$
of equation \ref{c} above. As noted, the family $\langle\varphi_\beta\mid\beta\in\varepsilon\rangle$ of fiber maps $\varphi_\beta=\rho_2(\,\cdot\,,\beta)=\rho_2^1(\,\cdot\,,\beta)$ satisfies this condition when $\varepsilon=\omega_1$ and $\mathcal{C}$ is an ordertype-minimal $C$-sequence on $\varepsilon$. In fact the following is true, and the present section will largely be dedicated to its proof.
By way of this theorem, higher walks satisfy the generalization criterion (n.4) together with its corresponding portion of (n.6).\footnote{Readers might object that the theorem only rigorously establishes (n.4) to ``extend to higher cardinals $\kappa$ under assumptions like $\square(\kappa)$'' for $n\leq 2$. This is true, but while the details of the $n>2$ extensions (which fall in any case beyond the scope of this introduction, and whose precise elaboration figures among our conclusion's list of most immediate next tasks) have yet to be fully sorted out, that the principles of the $n\leq 2$ extensions described in Section \ref{subsect:higher_coherence} will more generally apply seems nevertheless clear.}
\begin{theorem}
\label{thm:n-coherence}
Fix $n>0$. Under either of the following assumptions on an order-$n$ $C$-sequence $\mathcal{C}$ on an ordinal $\varepsilon$, the associated family
\begin{align*}
\Phi^n_2(\mathcal{C}):=\langle\rho_2^n(\,\cdot\,,\vec{\beta}):\beta_0\to\mathbb{Z}\mid\vec{\beta}\in [\varepsilon]^n\rangle
\end{align*}
is $n$-coherent modulo locally constant functions.
\begin{enumerate}
\item $\varepsilon=\omega_n$ and $\mathcal{C}$ is ordertype-minimal.
\item $\mathrm{otp}(C_{\vec{\gamma}})\leq\omega$ for any length-$n$ $\mathcal{C}$-index $\vec{\gamma}$.
\end{enumerate}
In addition, under the following assumption, if $n\leq 2$ then $\Phi^n_2(\mathcal{C})$ is $n$-coherent modulo bounded functions; more particularly, for any $\vec{\gamma}\in [\varepsilon]^{n+1}$,
$$\sum_{i=0}^n(-1)^i\rho_2^n(\,\cdot\,,\vec{\gamma}^i)$$
is eventually constant below any limit ordinal $\alpha\leq\gamma_0$.
\begin{enumerate}
\setcounter{enumi}{2}
\item If $\vec{\gamma}\in[\varepsilon]^n$ is a $\mathcal{C}$-index and $\alpha\in\mathrm{lim}(C_{\vec{\gamma}})$ then $\alpha\cap C_{\vec{\gamma}}=C_\alpha$.
\end{enumerate}
\end{theorem}
Observe that each of the three items generalizes the classical conditions of $\rho_2$-coherence on $\omega_1$, as we will review in Section \ref{subsect:classical_argument}; item 3 is more precisely, of course, a generalization of the (coherence half of the) classical $\square(\kappa)$-based generalization of the walks apparatus beyond $\omega_1$ alluded to in Theorem \ref{thm:square_PID} above.
In fact, (1) is a special case of (2), by the lemma just below, and an $(n-1)$-fold compounding of an ordertype-minimal classical $C$-sequence on $\omega_n$ (see Example \ref{ex:compounding}) is, in turn, a special case of (1). In particular, for any $n>0$, the existence of $C$-sequences on $\omega_n$ satisfying conditions 1 or 2 follows easily from the $\mathsf{ZFC}$ axioms.
\begin{lemma}
\label{lem:item2to1}
If $\mathcal{C}$ is an ordertype-minimal $C$-sequence on $\lambda^{+n}$ then $\mathrm{otp}(C_{\vec{\gamma}})\leq\lambda$ for all length-$n$ $\mathcal{C}$-indices $\vec{\gamma}$.
\end{lemma}
\begin{proof}
By assumption, $\mathrm{otp}(C_{\gamma_{n-1}})\leq\lambda^{+(n-1)}$, a relation which by the same assumption descends the length of $\vec{\gamma}$ to imply that $\mathrm{otp}(C_{\vec{\gamma}})\leq\lambda$.
\end{proof}
A weak converse to the lemma is that if $\mathrm{otp}(C_{\vec{\gamma}})\leq\lambda$ for all length-$n$ $\mathcal{C}$-indices $\vec{\gamma}$, then $\mathcal{C}$, if of order at least $n$, is a $C$-sequence on some $\varepsilon\leq\lambda^{+n}$. Note that conditions 2 or 3 of Theorem \ref{thm:n-coherence} may be vacuously satisfied, condition 2 when $\mathcal{C}_{n-1}=\varnothing$, for example, as in the circumstances discussed in Example \ref{ex:compounding}; the theorem will in those cases continue to apply.
Observe lastly that
 any $\mathcal{C}$ as described in Example \ref{ex:highCfromsquare} will meet condition 3 of Theorem \ref{thm:n-coherence}.

The most basic nonclassical instance of Theorem \ref{thm:n-coherence} is that of item 1 when $n=2$; it then unpacks as the assertion that \emph{for all $\beta<\gamma<\delta<\omega_2$ the function $f_{\beta\gamma\delta}:\beta\to\mathbb{Z}$ given by}
\begin{align*}\xi\,\mapsto\,\rho_2^2(\xi,\gamma,\delta)-\rho_2^2(\xi,\beta,\delta)+\rho_2^2(\xi,\gamma,\delta)\end{align*}
\emph{is locally constant}. At the heart of this fact, in one view, are the $C_\alpha$-relativizations $(\alpha\in\text{Cof}(\omega_1))$ of its $n=1$ analogue that \emph{for all $\beta<\gamma<\omega_1$ the function $f_{\beta\gamma}:\beta\to\mathbb{Z}$ given by}
\begin{align*} \xi\,\mapsto\,\rho_2(\xi,\gamma)-\rho_2(\xi,\beta)\end{align*}
\emph{is locally constant}. Hence we should begin by recalling why the latter assertion is so.
\subsection{Classical coherence}
\label{subsect:classical_argument}
As this subsection functions almost entirely as a review, readers with a solid background in walks might consider proceeding directly to Section \ref{subsect:higher_coherence}.
\subsubsection{Coherence on $\omega_1$}
\label{subsubsect:omega_1}
Fix an ordertype-minimal $C$-sequence $\mathcal{C}$ on $\omega_1$ and suppose towards contradiction that for some $\beta<\gamma<\omega_1$ the function $f_{\beta\gamma}$ just above is not locally constant.
Then there exists a limit $\alpha\leq\beta$ and cofinal $C\subseteq\alpha$ such that
$$\rho_2(\xi,\gamma)-\rho_2(\xi,\beta)\neq\rho_2(\alpha,\gamma)-\rho_2(\alpha,\beta)$$
for any $\xi\in C$.
Observe, though, that
$$\eta:=\mathrm{max}\big(\mathrm{L}(\alpha,\gamma)\cup\mathrm{L}(\alpha,\beta)\big)<\alpha,$$
since for all $\gamma_k\in\mathrm{Tr}(\alpha,\gamma)\backslash\{\alpha\}$ and $\beta_j\in\mathrm{Tr}(\alpha,\beta)\backslash\{\alpha\}$ both $C_{\gamma_k}\cap\alpha$ and $C_{\beta_j}\cap\alpha$ are, by our ordertype assumptions, finite, and therefore bounded below $\alpha$. Moreover,
\begin{align*}
\mathrm{Tr}(\xi,\gamma) & = \mathrm{Tr}(\alpha,\gamma)\cup \mathrm{Tr}(\xi,\alpha),\text{ and} \\
\mathrm{Tr}(\xi,\beta) & = \mathrm{Tr}(\alpha,\beta)\cup \mathrm{Tr}(\xi,\alpha)
\end{align*}
for any $\xi\in (\eta,\alpha]$, for the reason that, by arrangement, no chances arise within the course of the walks from $\gamma$ and $\beta$ down to $\alpha$ to step any nearer to such a $\xi$. It follows that
\begin{align*}
\rho_2(\xi,\gamma)-\rho_2(\xi,\beta) = \rho_2(\alpha,\gamma)+\rho_2(\xi,\alpha)-\rho_2(\alpha,\beta)-\rho_2(\xi,\alpha) = \rho_2(\alpha,\gamma)-\rho_2(\alpha,\beta)
\end{align*}
for all such $\xi$, contradicting our assumptions.

The above sequence of recognitions is among the most fundamental in the theory of classical walks. Essentially the same arguments show, for example, that the fiber maps of $\rho_1$ are, under the same assumptions on $\mathcal{C}$, both coherent and finite-to-one and hence nontrivial; see \cite[\S 2, and particularly Remark\ 2.7]{LSpace} for a concise account.
Here in a nutshell is also the core theme of ``pass[ing walks] through'' specified or ``interesting places''; see \cite[\S 6]{todpairs}.
As we'll see shortly, the key motifs above --- control of the upper trace by the lower trace; end extensions, with attendant cancellations --- each do extend via higher walks to settings above $\omega_1$.
Let us briefly recall before proceeding, though, the ways that they more classically both do and do not.
\subsubsection{Coherence on internal walks}
\label{subsubsect:internal}
Consider next the more general setting of $C_\delta$-internal walks, where $\delta$ is any ordinal of cofinality $\omega_1$ and $\pi:\omega_1\to C_\delta$ is an order-isomorphism as in Section \ref{subsect:internal_walks}. Suppose again for contradiction that for some $\beta<\gamma<\delta$ there exists a limit $\alpha\leq\beta$ and cofinal $C\subseteq\alpha$ such that
$$\rho_2[\delta](\xi,\gamma)-\rho_2[\delta](\xi,\beta)\neq\rho_2[\delta](\alpha,\gamma)-\rho_2[\delta](\alpha,\beta)$$
for any $\xi\in C$. There are now two ways in which our assumption can break down:

\underline{Case 1}: $\alpha\in C_\delta$. In this case, the classical argument applies within $C_\delta$. In particular, let $\eta=\pi(\mathrm{max}(\mathrm{L}(\alpha',\gamma')\cup\mathrm{L}(\alpha',\beta'))),$
where $\alpha'$, $\beta'$, and $\gamma'$ are the $\pi$-preimages of $\alpha$, $\mathrm{min}(C_\delta\backslash\beta)$, and $\mathrm{min}(C_\delta\backslash\gamma)$, respectively; the contradiction is then just as before.

\underline{Case 2}: $\alpha\notin C_\delta$. Let $\eta=\mathrm{max}(C_\delta\cap\alpha)$ and observe that $\mathrm{Tr}^\delta(\,\cdot\,,\gamma)$ and $\mathrm{Tr}^\delta(\,\cdot\,,\beta)$ are each constant on the interval $(\eta,\alpha]$, again contradicting our assumption about $C$.

An identical argument applies to the $C_\delta$-internal walk induced by the elements $\{C_{\gamma\delta}\mid\gamma\in C_\delta\}$ of an order-$2$ ordertype-minimal $C$-sequence $\mathcal{C}$ on any $\varepsilon$ containing $\delta$, and under the same assumptions on $\mathcal{C}$ this is a main mechanism of, for example, the $2$-coherence of $\rho^2_2$ on $\varepsilon=\omega_2$, as we'll see in Section \ref{subsect:higher_coherence}.
\subsubsection{Above $\omega_1$ with half of a square}
\label{subsubsect:square}
In contrast to the above, settings of cofinality greater than $\omega_1$ will involve clubs $C_\beta$ which, unavoidably, contain limit ordinals $\alpha$, with the consequence that the key relation $\mathrm{max}\,\mathrm{L}(\alpha,\gamma)<\alpha$ in Section \ref{subsubsect:omega_1} will no longer so generally hold.
The two main classical strategies for conserving coherence in these contexts are as follows. In the first, we relax our modulus to one for which an argument only at the cofinality-$\kappa$ points below $\kappa^+$ suffices (as when $\kappa=\omega$); see Section \ref{subsubssect:ctblvar} below.
In the second, we let the $C$-sequence itself carry those uniformities which we formerly derived from the passage through to $\alpha$, and hence to $C_\alpha$, as in Section \ref{subsubsect:omega_1}; more precisely, we fix a classical $C$-sequence $\mathcal{C}$ on an ordinal $\varepsilon$ which satisfies the condition 1 of Definition \ref{def:square} that
$$C_\beta\cap\alpha = C_\alpha\text{ for all }\beta<\varepsilon\text{ and }\alpha\in \mathrm{acc}(C_\beta).$$
Let us show the associated $\rho_2$ fiber maps coherent modulo bounded functions, and in fact of eventually constant difference below any limit ordinal $\alpha$.
Fix $\beta,\gamma\geq\alpha$ as before. If
$\mathrm{max}(\mathrm{L}(\alpha,\gamma)\cup\mathrm{L}(\alpha,\beta))<\alpha$
then we may reason just as before.
So suppose instead that $\mathrm{max}\,\mathrm{L}(\alpha,\gamma)=\mathrm{max}\,\mathrm{L}(\alpha,\beta)=\alpha$; let then $\gamma'=\mathrm{min}\,\mathrm{Tr}(\alpha,\gamma)\backslash\{\alpha\}$ and $\beta'=\mathrm{min}\,\mathrm{Tr}(\alpha,\gamma)\backslash\{\alpha\}$ and observe that $\eta':=\mathrm{max}(\mathrm{L}(\gamma',\gamma)\cup\mathrm{L}(\beta',\beta))<\alpha$ and that for any $\xi\in(\eta',\alpha]$, the descending sequences $\mathrm{Tr}(\xi,\gamma)$ and $\mathrm{Tr}(\xi,\beta)$ end-extend $\mathrm{Tr}(\gamma',\gamma)$ and $\mathrm{Tr}(\beta',\beta)$, respectively. Moreover, our assumptions ensure that those end-extensions $\mathrm{Tr}(\xi,\gamma')\backslash\{\gamma'\}$ and $\mathrm{Tr}(\xi,\beta')\backslash\{\beta'\}$ are identical; in consequence, $f_{\beta\gamma}\big|_{(\eta',\alpha]}$ is constant, just as before.
Only one possibility remains: that represented (without loss of generality) by the scenario $\mathrm{max}\,\mathrm{L}(\alpha,\gamma)=\alpha$ while $\mathrm{max}\,\mathrm{L}(\alpha,\beta)<\alpha$. For $\gamma',\beta'$ as above, let $\eta'=\mathrm{max}(\mathrm{L}(\alpha,\gamma)\cup\mathrm{L}(\alpha,\beta))$, and observe that for any $\xi\in(\eta',\alpha)$, $\mathrm{Tr}(\xi,\gamma)$ steps from $\gamma'$ to $\mathrm{min}\,C_{\gamma'}\backslash\xi=\mathrm{min}\,C_\alpha\backslash\xi$, while $\mathrm{Tr}(\xi,\beta)$ steps from $\beta'$ first to $\alpha$, then to $\mathrm{min}\,C_\alpha\backslash\xi$, but that the two again walk identically thereafter.
Note that this implies $f_{\beta\gamma}$ constant on $(\eta',\alpha)$, but not on $(\eta',\alpha]$, but that this is sufficient for our claim.

This discontinuity at $\alpha$ is an utterly minor point, though, by the following lemma. For any abelian group $A$, call a function $f:\beta\to A$ \emph{locally semi-constant} if for all limit $\alpha$ there exists an $\eta<\alpha$ with $f\big|_{(\eta,\alpha)}$ constant.
\begin{lemma}
\label{lem:ntc_conversion}
There exists a nontrivial $n$-coherent family of functions \emph{mod locally semi-constant}
\begin{align*}
\Phi=\langle\varphi_{\vec{\beta}}:\beta_0\to A\mid\vec{\beta}\in [\varepsilon]^n\rangle
\end{align*}
if and only if there exists a nontrivial $n$-coherent family of functions \emph{mod locally constant}
\begin{align*}
\hat{\Phi}=\langle\hat{\varphi}_{\vec{\beta}}:\beta_0\to A\mid\vec{\beta}\in [\varepsilon]^n\rangle.
\end{align*}
More precisely, the quotient of the group of $n$-coherent such families by the trivial such families with respect to the first modulus is canonically isomorphic to the quotient of the group of $n$-coherent such families by the trivial such families with respect to the second.
\end{lemma}
Here we've anticipated a bit: see Definition \ref{def:n-triv} for the meaning of \emph{nontriviality} in higher-dimensional contexts. Moreover, we leave the lemma's rigorous proof to the interested reader. One approach is essentially that of \cite[Lem.\ 2.22]{CoOI}. A more intuitive approach generalizes the following device to the cases $n>1$. Let $\Phi=\langle\varphi_\beta:\beta\to A\mid\beta\in  \varepsilon\rangle$ be, mod locally semi-constant, nontrivially coherent, and define the functions $\hat{\varphi}_\beta$ on successors $\alpha+1$ by $\hat{\varphi}_\beta(\alpha+1)=\varphi_\beta(\alpha)$. For any limit $\alpha$ and $\gamma\geq\beta>\alpha$ let $\theta_{\beta\gamma}(\alpha)$ denote the eventually constant value of $\varphi_\gamma-\varphi_\beta$ below $\alpha$. Observe that $\theta_{\gamma\delta}(\alpha)-\theta_{\beta\delta}(\alpha)+\theta_{\beta\gamma}(\alpha)=0$ for all $\varepsilon>\delta\geq\gamma\geq\beta>\alpha$, and hence that there exists a family of $\theta_\beta(\alpha)\in A$ such that $\theta_\gamma(\alpha)-\theta_\beta(\alpha)=0$ for all $\gamma\geq\beta>\alpha$. Conclude the construction by letting $\hat{\varphi}_\beta(\alpha)=\theta_\beta(\alpha)$.

The lemma affords us a simplifying reading of Theorem \ref{thm:n-coherence}, namely that under suitable $C$-sequence assumptions, the mod locally constant $n$-coherence of $\rho^n_2$ on $\omega_n$ extends (for $n\leq 2$, and conjecturally for all $n$) to higher cardinals as well.
\subsection{Higher coherence}
\label{subsect:higher_coherence}
Recall that we have defined $\rho_{2,\mathrm{t}}^n(\vec{\gamma})$ for any $n>0$ as the subset of $n^{<\omega}$ naturally indexing $\mathrm{Tr}_n(+,\vec{\gamma})$.
Its elements $\sigma$ also comprise the indices of a lower-trace function $\mathrm{L}_n$ which, just as in the classical case, mirrors the upper trace $\mathrm{Tr}_n$ from below,\footnote{Todorcevic at times has offered the image of reflection in the surface of a pond which we should imagine, in the present context, to lie at the level of $\gamma_0$.} ascending as it descends.
It will be convenient to introduce the notations $\mathrm{Tr}_n(\pm,\vec{\gamma})(\sigma)$ and $\mathrm{L}_n(\vec{\gamma})(\sigma)$, for $\sigma\in\rho_{2,\mathrm{t}}^n(\vec{\gamma})$, to denote the $\sigma^{\mathrm{th}}$ element of a given $\mathrm{Tr}_n$ and $\mathrm{L}_n$, respectively.
These will specify outputs; it will be useful to record the $\mathrm{Tr}_n$ input appearing at stage $\sigma$ as well (note that the terminal inputs of $\mathrm{Tr}_n(\pm,\vec{\gamma})$ are unindexed by $\rho_{2,\mathrm{t}}^n(\vec{\gamma})$; these form a central concern in Lemmas \ref{lem:cyclic} and \ref{lem:pairing} below). The cases of most immediate interest derive from expressions like $\mathrm{Tr}_n(\pm,\beta,\vec{\gamma})$, and since in this expression's tree of inputs only the $n$ ordinal coordinates following $\beta$ are in motion, we will write $\vec{\gamma}^\sigma$ to refer just to that portion of the $\sigma^{\mathrm{th}}$ input of $\mathrm{Tr}_n(\pm,\beta,\vec{\gamma})$ (so that $\vec{\gamma}^\varnothing=\vec{\gamma}$, for example).
\begin{definition}
\label{def:Ln}
With respect to an order-$n$ $C$-sequence $\mathcal{C}$ on $\varepsilon$, the \emph{order-$n$ lower trace function} $\mathrm{L}_n$ is defined by letting
$$\mathrm{L}_n(\beta,\vec{\gamma})(\sigma)=\max_{j\leq |\sigma|}(\sup (\beta\cap C_{\vec{\gamma}^{\sigma\restriction j}}))$$
for any $\vec{\gamma}\in\varepsilon^{[n]}$ and $\beta\leq\gamma_0$ and $\sigma\in\rho_{2,\mathrm{t}}^n(\beta,\vec{\gamma})$, under the conventions that $\sup\varnothing=0$ and that $C_{\vec{\alpha}}=\varnothing$ if $\vec{\alpha}$ is not a $\mathcal{C}$-index.
\end{definition}
Observe that the $n=1$ instance is as it should be, namely, it is exactly the function $\mathrm{L}$.
Lemmas \ref{lem:max_L<limits} and \ref{lem:end-extension} below generalize the classical case as we would hope as well.
\begin{lemma}
\label{lem:max_L<limits}
Suppose that $n>0$ and $\mathcal{C}$ is an order-$n$ $C$-sequence on an ordinal $\varepsilon$ for which $\mathrm{otp}(C_{\vec{\gamma}})\leq\omega$ for any length-$n$ $\mathcal{C}$-index $\vec{\gamma}$.
Then $$\mathrm{max}\:\mathrm{L}_n(\beta,\vec{\gamma})<\beta$$ for any $\vec{\gamma}\in \varepsilon^{[n]}$ and nonzero $\beta\leq\gamma_0$.
\end{lemma}
\begin{proof}
It suffices to show that $\mathrm{L}_n(\beta,\vec{\gamma})(\sigma)$ is less than $\beta$ for any $\sigma\in\rho_{2,\mathrm{t}}^n(\beta,\vec{\gamma})$; this we show by induction on $\sigma$.
Note first that if it is defined then $\mathrm{L}_n(\beta,\vec{\gamma})(\varnothing)=0$ unless $\tau(\beta,\vec{\gamma})=\vec{\gamma}$ and $\beta<\gamma_0$, in which case $\mathrm{L}_n(\beta,\vec{\gamma})(\varnothing)=\mathrm{sup}\:\beta\cap C_{\vec{\gamma}}$; our assumption then completes the argument for $\sigma=\varnothing$.
Suppose next that we are at stage $\sigma$ and we have shown our claim at $\sigma\restriction|\sigma|-1$.
If $\tau(\beta,\vec{\gamma}^\sigma)\neq\vec{\gamma}^\sigma$ then $\mathrm{L}_n(\beta,\vec{\gamma})(\sigma)=\mathrm{L}_n(\beta,\vec{\gamma})(\sigma\restriction|\sigma|-1)$, while if $\tau(\beta,\vec{\gamma}^\sigma)=\vec{\gamma}^\sigma$ then $$\mathrm{L}_n(\beta,\vec{\gamma})(\sigma)=\mathrm{max}\big(\mathrm{L}_n(\beta,\vec{\gamma})(\sigma\restriction|\sigma|-1),\mathrm{sup}\,(\beta\cap C_{\vec{\gamma}^\sigma})\big)$$
and $\beta<\gamma_0$ (since $\sigma\in\rho_{2,\mathrm{t}}^n(\beta,\vec{\gamma})$), hence by our lemma's premise, $\mathrm{sup}\,(\beta\cap C_{\vec{\gamma}^\sigma})<\beta$, completing the proof.
\end{proof}
\begin{lemma}
\label{lem:end-extension}
For all $n>0$ and $\vec{\gamma}\in\varepsilon^{[n]}$ and $\beta\leq\gamma_0$ and $\alpha\in(\mathrm{max}\:\mathrm{L}_n(\beta,\vec{\gamma}),\beta]$, the tree $\mathrm{Tr}_n(\pm,\alpha,\vec{\gamma})$ end-extends the tree $\mathrm{Tr}_n(\pm,\beta,\vec{\gamma})$.
\end{lemma}
The trees in question, of course, are those of the outputs of $\mathrm{Tr}_n$; at the level of the inputs, the lemma is asserting that an initial portion of $\mathrm{Tr}_n(\pm,\alpha,\vec{\gamma})$ is identical to $\mathrm{Tr}_n(\pm,\beta,\vec{\gamma})$, only with $\alpha$ everywhere replacing $\beta$ in the first ordinal coordinate.
\begin{proof}
Observe first that, for any $\sigma\in\rho_{2,\mathrm{t}}^n(\beta,\vec{\gamma})$, if $\tau(\beta,\vec{\gamma}^\sigma)\neq\vec{\gamma}^\sigma$ then the difference between $\alpha$ and $\beta$ is immaterial to the step $\sigma$ of the expansion of $\mathrm{Tr}_n(\pm,\beta,\vec{\gamma})$.
If, on the other hand, $\tau(\beta,\vec{\gamma}^\sigma)=\vec{\gamma}^\sigma$, then for any $\alpha$ in the interval $(\mathrm{max}\:\mathrm{L}_n(\beta,\vec{\gamma}),\beta]$, if $\mathrm{Tr}_n(\pm,\beta,\vec{\gamma})(\sigma')=\mathrm{Tr}_n(\pm,\alpha,\vec{\gamma})(\sigma')$ for all proper initial segments $\sigma'$ of $\sigma$, then $\mathrm{Tr}_n(\pm,\beta,\vec{\gamma})(\sigma)=\mathrm{Tr}_n(\pm,\alpha,\vec{\gamma})(\sigma)$ by a reasoning identical to that of the classical case, namely because $\mathrm{L}_n(\beta,\vec{\gamma})(\sigma)$ certifies that there is no ordinal between $\alpha$ and $\beta$ for $\mathrm{Tr}_n(\pm,\beta,\vec{\gamma})$ at stage $\sigma$ to step to.
The lemma thus follows by induction; note that no particular assumptions on the underlying $C$-sequence were required.
\end{proof}
Recall now the function $f_{\beta\gamma\delta}:\beta\to\mathbb{Z}$ of this section's introduction, given for any $\beta<\gamma<\delta$ in $\omega_2$ by
\begin{equation*}
\xi\mapsto\rho_2^2(\xi,\gamma,\delta)-\rho_2^2(\xi,\beta,\delta)+\rho_2^2(\xi,\beta,\gamma),
\end{equation*}
defined with respect to an order-$2$ ordertype-minimal $C$-sequence on $\omega_2$. We claim, as an instance of Theorem \ref{thm:n-coherence}, that $f_{\beta\gamma\delta}$ is locally constant; let us see how far the above lemmas take us in the direction of our claim.
Suppose, as classically, that our claim is false, so that there exists a limit $\alpha\leq\beta$ and a cofinal subset $C\subseteq\alpha$ such that $f_{\beta\gamma\delta}(\xi)\neq f_{\beta\gamma\delta}(\alpha)$ for all $\xi\in C$.
Let
\begin{align}
\label{eq:max_L_ns}
\eta=\mathrm{max}\,(\mathrm{L}_2(\alpha,\gamma,\delta)\cup\mathrm{L}_2(\alpha,\beta,\delta)\cup\mathrm{L}_2(\alpha,\beta,\gamma)).
\end{align}
Observe that $\eta<\alpha$ by Lemma \ref{lem:max_L<limits}, and that by Lemma \ref{lem:end-extension}, the portion of
\begin{equation}
\label{eq:Tr_2_array}
\mathrm{Tr}_2(+,\xi,\gamma,\delta)\sqcup\mathrm{Tr}_2(-,\xi,\beta,\delta)\sqcup
\mathrm{Tr}_2(+,\xi,\beta,\gamma)
\end{equation}
greater than or equal to $\alpha$ is unchanging as $\xi$ ranges within the interval $(\eta,\alpha]$.
We will attain a contradiction as desired if for all $\xi\in (\eta,\alpha)$ the numbers of positive signed ordinals and negative signed ordinals appearing below $\alpha$ in expression \ref{eq:Tr_2_array} are equal.
This indeed holds, and follows from the stronger fact that for every $\xi\in(\eta,\alpha)$ the signed ordinals $\zeta\in [\xi,\alpha)$ appearing in expression \ref{eq:Tr_2_array} arise in opposite-signed pairs $\{+\zeta,-\zeta\}$.

This fact will follow from the $n=2$ instances of Lemmas \ref{lem:cyclic} and \ref{lem:pairing} below.
These lemmas will necessitate some further definitions and notations, which we will ground in the $n=2$ case as we proceed.
First, with respect to any $n>0$ and order-$n$ $C$-sequence $\mathcal{C}$ on an ordinal $\varepsilon$, call any $\mathrm{Tr}_n(x)$ with $x$ a signed tuple $\vec{\gamma}\in\varepsilon^{[n+1]}$ a \emph{$\mathrm{Tr}_n$-term}.
As above, we will speak of the \emph{expansion} of $\mathrm{Tr}_n$-terms, which transpires via \emph{steps} of two sorts; these consist in the replacement of further $\mathrm{Tr}_n$-terms deriving from $\mathrm{Tr}_n(x)$ via repeated applications of Definition \ref{def:Tr_n} with either (1) an output together with yet further $\mathrm{Tr}_n$-terms, or (2) the empty set. We call steps of the former type \emph{nondegenerate}.
Those cuing the latter replacement are, of course, the terms in the expansion of $\mathrm{Tr}_n(x)$ satisfying the boundary conditions of Definition \ref{def:Tr_n}, and we therefore denote their collection by $\mathtt{b}\mathrm{Tr}_n(x)$; here and between  pairs of brackets $\{$ and $\}$ below, we allow for the possibility that a collection may include multiple copies of a given object.
Next, for any $n>1$, let $\{\lfloor\vec{\gamma}\rfloor\mid \vec{\gamma}\in\varepsilon^{[n]}\}$ denote the set of generators of the free abelian group on $\varepsilon^{[n]}$; the following then determines a well-defined map from the collection of finite disjoint unions of $\mathrm{Tr}_n$-terms to $\bigoplus_{\varepsilon^{[n-1]}}\mathbb{Z}$ (we should perhaps stress that this map takes as arguments \emph{formal objects}, i.e., well-formed expressions made from ``$\mathrm{Tr}_n(x)$'' and ``$\sqcup$'', where $x$ is a signed nondecreasing $(n+1)$-tuple of ordinals below $\varepsilon$):
\begin{align*}
\partial_n\mathrm{Tr}_n((-1)^k,\xi,\vec{\gamma}) & =\sum_{i=0}^{n-1}(-1)^{i+k}\lfloor\vec{\gamma}^i\rfloor\,,\text{ and} \\
\partial_n(\mathrm{Tr}_n(x)\,\sqcup\,\mathrm{Tr}_n(y)) & =\partial_n\mathrm{Tr}_n(x)\,+\,\partial_n\mathrm{Tr}_n(y).
\end{align*}
For example: $\partial_2\mathrm{Tr}_2(+,\xi,\beta,\gamma) =\lfloor\gamma\rfloor-\lfloor\beta\rfloor$ and $\partial_2\mathrm{Tr}_2(-,\xi,\beta,\gamma)=\lfloor\beta\rfloor-\lfloor\gamma\rfloor$.

\begin{definition} A collection of $\mathrm{Tr}_n$-terms $\{\mathrm{Tr_n}(x_0),\dots,\mathrm{Tr}_n(x_k)\}$ is \emph{cyclic} if $$\partial_n(\mathrm{Tr}_n(x_0)\,\sqcup\,\dots\,\sqcup\,\mathrm{Tr}_n(x_k))=0.$$
\end{definition}

\begin{example}\label{ex:cyclic2} The collection $\{\mathrm{Tr}_2(+,\xi,\gamma,\delta),\mathrm{Tr}_2(-,\xi,\beta,\delta),\mathrm{Tr}_2(+,\xi,\beta,\gamma)\}$ corresponding to expression (\ref{eq:Tr_2_array}) above is cyclic.
\end{example}

\begin{lemma}\label{lem:cyclic} For any $n>1$, if a collection $\mathcal{T}$ of $\mathrm{Tr}_n$-terms is cyclic, so too is the collection formed from $\mathcal{T}$ by replacing one of its terms $\mathrm{Tr}_n(x)$ by the $n$ terms deriving from $\mathrm{Tr}_n(x)$ in a nondegenerate step. In consequence, if $\mathcal{T}$ is cyclic, then so too is
\begin{align}\bigsqcup_{\mathrm{Tr}_n(x)\in\mathcal{T}}\mathtt{b}\mathrm{(Tr}_n(x))
\end{align}
\end{lemma}

\begin{proof}
It should be clear that the second assertion follows from iterated applications of the first.
Reference Definition \ref{def:Tr_n} for the first, which follows from the fact that $\partial_n(\mathrm{Tr}_n((-1)^k,\iota(\vec{\gamma}),\tau(\vec{\gamma}))$ equals
$$\partial_n\Bigg(\bigsqcup_{i\in\{1,\dots,n+1\}\backslash\{j+1\}}\mathrm{Tr}_n\big((-1)^{i+j+k},(\iota(\vec{\gamma}),\mathrm{min}(C_{\tau(\vec{\gamma})}\backslash\gamma_j),\tau(\vec{\gamma}))^i\big)\Bigg).$$
To see this, observe that the disjoint union of $\mathrm{Tr}_n((-1)^{k+1},\iota(\vec{\gamma}),\tau(\vec{\gamma}))$ with the argument of the displayed term above may be written as
$$\bigsqcup_{i\in\{1,\dots,n+1\}}\mathrm{Tr}_n\big((-1)^{i+j+k},(\iota(\vec{\gamma}),\mathrm{min}(C_{\tau(\vec{\gamma})}\backslash\gamma_j),\tau(\vec{\gamma}))^i\big),$$
and that $\partial_n$ of this term equals zero, by the usual simplicial arithmetic whereby the composition of two boundary maps equals zero.
\end{proof}
\begin{example}
As noted, $\partial_2\mathrm{Tr}_2(+,\xi,\beta,\gamma)=\lfloor\gamma\rfloor-\lfloor\beta\rfloor$. If $\beta\notin C_\gamma$ and $\eta=\min C_\gamma\backslash\beta$ then $\mathrm{Tr}_2(+,\xi,\beta,\gamma)$ steps to $\mathrm{Tr}_2(+,\xi,\beta,\eta)\,\sqcup\,\mathrm{Tr}_2(+,\xi,\eta,\gamma)$, with $\partial_2$-image $\lfloor\eta\rfloor-\lfloor\beta\rfloor+\lfloor\gamma\rfloor-\lfloor\eta\rfloor=\lfloor\gamma\rfloor-\lfloor\beta\rfloor$. If $\beta\in C_\gamma$ and $\eta=\min C_{\beta\gamma}\backslash\xi$ then the $\partial_2$-image of the step is again $\lfloor\eta\rfloor-\lfloor\beta\rfloor+\lfloor\gamma\rfloor-\lfloor\eta\rfloor=\lfloor\gamma\rfloor-\lfloor\beta\rfloor$; in all cases the overall sum is unchanged.
\end{example}
This property of ``conserving cyclicity'' is among the most cogent of justifications for the sign-rules determining $\mathrm{Tr}_n$.

Continuing with our $n=2$ example, what does it mean --- in particular for our argument that $\rho_2^2$ is $2$-coherent, i.e., that $f_{\beta\gamma\delta}$ is locally constant --- that
\begin{equation}
\label{eq:bTr_2_array}
\mathtt{b}\mathrm{Tr}_2(+,\alpha,\gamma,\delta)\sqcup\mathtt{b}\mathrm{Tr}_2(-,\alpha,\beta,\delta)\sqcup
\mathtt{b}\mathrm{Tr}_2(+,\alpha,\beta,\gamma)
\end{equation}
is cyclic? Since the inputs $(\pm,\alpha,\beta',\gamma')$ therein all satisfy the boundary conditions, they are all of one of two forms: (i) $\beta'=\gamma'$, or (ii) $\beta'=\mathrm{min}\,C_{\gamma'}\backslash\alpha$.
Note in the first case that $\mathrm{Tr}_2(\pm,\alpha,\beta',\gamma')$ cannot have derived from any higher $\mathrm{Tr}_2$-term in 
\begin{equation}
\label{eq:Tr_2_alpha_array}
\mathrm{Tr}_2(+,\alpha,\gamma,\delta)\sqcup\mathrm{Tr}_2(-,\alpha,\beta,\delta)\sqcup
\mathrm{Tr}_2(+,\alpha,\beta,\gamma),
\end{equation}
hence two ordinals among $\{\beta,\gamma,\delta\}$ must equal $\beta'$ and $\gamma'$ and are thus themselves equal; in this case, however, $f_{\beta\gamma\delta}$ is constantly $0$ (alternately, simply observe that $(\beta,\gamma,\delta)\notin [\omega_2]^3$).
Without loss of generality, we may therefore suppose all inputs to be of the second form.
Fix now some $\mathrm{Tr}_2(+,\alpha,\beta',\gamma')$ in (\ref{eq:bTr_2_array}) (the argument is identical for $\mathrm{Tr}_2(-,\alpha,\beta',\gamma')$) and a minimal cyclic subcollection $c$ of (\ref{eq:bTr_2_array}) containing $\mathrm{Tr}_2(+,\alpha,\beta',\gamma')$, and let $\gamma''$ denote the highest ordinal appearing among the inputs of $c$. As $c$ is cyclic, $\gamma''$  must appear in at least two opposite-signed $\mathrm{Tr}_2$-terms; under our assumptions, however, these terms' ordinal inputs are fully determined: each is of the form $\mathrm{Tr}_2(\pm,\alpha,\mathrm{min}\,C_{\gamma''}\backslash\alpha,\gamma'')$.
The removal of such a pair would therefore yield a proper subcycle of $c$ containing $\mathrm{Tr}_2(+,\alpha,\beta',\gamma')$ --- a contradiction --- unless $\mathrm{min}\,C_{\gamma''}\backslash\alpha=\beta'$ and $\gamma''=\gamma'$; it follows that $c=\mathrm{Tr}_2(+,\alpha,\beta',\gamma')\sqcup\mathrm{Tr}_2(-,\alpha,\beta',\gamma')$. Iterate this argument to partition the expression (\ref{eq:bTr_2_array}) into a collection of pairs of opposite-signed repeated $\mathrm{Tr}_2$-terms.

Consider now the implications for $f_{\beta\gamma\delta}$. We have seen that for $\eta$ as in equation \ref{eq:max_L_ns}, the portion of (\ref{eq:Tr_2_array}) greater than or equal to $\alpha$ is unchanging as $\xi$ ranges within the interval $(\eta,\alpha]$. We have also observed that the inputs of (\ref{eq:Tr_2_array}) whose last two ordinals are greater than or equal to $\alpha$  are exactly those of (\ref{eq:Tr_2_alpha_array}) with the first ordinal coordinates all replaced with $\xi$. This holds in particular for the inputs appearing in (\ref{eq:bTr_2_array}), and since these fully determine (\ref{eq:Tr_2_array}) below $\alpha$ and occur, by the above reasoning, in matched canceling pairs, $f_{\beta\gamma\delta}$ is constant on the interval $(\eta,\alpha]$ and is, more generally, locally constant, as desired.

By the following lemma together with the others of this section, this reasoning holds for all $n\geq 1$, and this essentially establishes the first two items of Theorem \ref{thm:n-coherence}.
Two last points might be underscored in the $n=2$ case, though, before proceeding:
\begin{enumerate}
\item Any $\mathrm{Tr}_2(+,\alpha,\beta',\gamma')$ in (\ref{eq:bTr_2_array}) of type (ii) attests to a $C_{\gamma'}$-internal walk within (\ref{eq:Tr_2_alpha_array}), and the pairing of such $\mathrm{Tr}_2$-terms induces a pairing of opposite-signed $C_{\gamma'}$-internal walks within (\ref{eq:Tr_2_alpha_array}) which, although beginning from possibly different inputs, will agree on the passage through $\alpha$ to any $\xi$ above $\eta$, by the logic of Subsection \ref{subsubsect:internal}.
In this view, the mechanism of the $2$-coherence of $\rho^2_2$ is that of the $1$-coherence of $\rho^1_2$ manifesting along these paired internal walks above the relevant lower trace values, and above $\eta$ in particular.
\item Although the general principles of the pairings we're describing in this section's lemmas aren't ultimately very complicated, their working out in practice can be quite subtle, as Subsection \ref{subsubsect:pairings} below will suggest.
\end{enumerate}

\begin{lemma}
\label{lem:pairing}
Fix $n>0$ and an ordinal $\varepsilon$ and an order-$n$ $C$-sequence on $\varepsilon$.
For any $\vec{\gamma}\in[\varepsilon]^{n+1}$ and $\alpha\leq\gamma_0$, the collection of arguments appearing in
\begin{align}
\label{eq:lemma_pairing}
\bigsqcup_{i\leq n}\mathtt{b}\mathrm{Tr}_n((-1)^i,\alpha,\vec{\gamma}^i)
\end{align}
admits a partition into pairs $\{+\vec{\zeta},-\vec{\zeta}\}$ of opposite-signed $(n+1)$-tuples.
\end{lemma}
\begin{proof}
The $n=1$ case is clear, and we have argued the $n=2$ case just above. For $n>2$, proceed as above: pick a $\mathrm{Tr}_n(+,\alpha,\vec{\beta})$ in (\ref{eq:lemma_pairing}) and a smallest subscycle $c$ of (\ref{eq:lemma_pairing}) containing $\mathrm{Tr}_n(+,\alpha,\vec{\beta})$.
Let $\triangleleft$ denote the right-to-left lexicographic order on $\varepsilon^{[n]}$, so that $\vec{\beta}'\triangleleft\vec{\beta}''$ if and only if $\beta_i'<\beta_i''$ for some $i$ with $\beta_j'=\beta_j''$ for all $j>i$, and let $\vec{\xi}$ be $\triangleleft$-maximal among those $\vec{\beta}'$ appearing among the arguments $(\pm,\alpha,\vec{\beta}')$ in $c$.
Observe that $$\lfloor((\alpha,\vec{\xi})^{|\iota(\alpha,\vec{\xi})|})^0\rfloor$$
(the superscripts record the removal of $\alpha$ and the least element of $\tau(\alpha,\vec{\xi})$) is a summand of $\partial_n\mathrm{Tr}_n(+,\alpha,\vec{\xi})$, and hence must be a summand of $\partial_n$ of some other $\mathrm{Tr}_n(+,\alpha,\vec{\xi}')$ in $c$. This, however, implies that $\tau(\alpha,\vec{\xi}')=\tau(\alpha,\vec{\xi}')$, and hence that $(\alpha,\vec{\xi})=(\alpha,\vec{\xi}')$: in the only alternative, $\tau(\alpha,\vec{\xi}')=\tau(\alpha,\vec{\xi})^0$ --- but $(\alpha,\vec{\xi}')$ could only then satisfy the boundary condition if its $|\iota(\alpha,\vec{\xi}')|^{\mathrm{th}}$ coordinate were greater than that of $(\alpha,\vec{\xi})$, violating our assumptions about $\vec{\xi}$.
Only one conclusion is now possible, namely that $c$ equals $\mathrm{Tr}_n(+,\alpha,\vec{\beta})\sqcup\mathrm{Tr}_n(-,\alpha,\vec{\beta})$.
\end{proof}
Observe that as the $\vec{\gamma}$ of the lemma's premise ranges through most of $\varepsilon^{[n+1]}$, our proof will continue to apply. We restricted to $[\varepsilon]^{n+1}$ simply to exclude from consideration those $\mathrm{Tr}_n(\pm,\alpha,\vec{\gamma})$ which are themselves cyclic, as in the $\mathrm{Tr}_2(\pm,\alpha,\gamma',\gamma')$ example above. This is a minor enough issue that we leave its details to the interested reader; the main point is that such $\mathrm{Tr}_n(\pm,\xi,\vec{\gamma})$ will not themselves vary as $\xi$ ranges below $\alpha$.

Let us now turn to the proof of Theorem \ref{thm:n-coherence}.
\begin{proof}[Proof of Theorem \ref{thm:n-coherence}]
First note that the theorem's $n=1$ case is classical, so that we may restrict our attention to $n>1$.
By Lemma \ref{lem:item2to1}, the theorem's item 1 follows from its item 2. For the latter, the sequence of lemmas \ref{lem:max_L<limits}, \ref{lem:end-extension}, \ref{lem:cyclic}, and \ref{lem:pairing} applies exactly as in the $n=2$ case; more precisely, this sequence shows that for any $\vec{\gamma}\in [\varepsilon]^{n+1}$ the analogue $f_{\vec{\gamma}}$ of the function $f_{\beta\gamma\delta}$ is locally constant, implying the $n$-coherence of $\Phi^n_2(\mathcal{C})$, as desired.

Item 3 of the theorem builds on this reasoning, but will require a little more work.
Having addressed the case of $n=1$ in Section \ref{subsubsect:square}, let us focus on the case of $n=2$: it will suffice to show for an arbitrary $(\beta,\gamma,\delta)\in [\varepsilon]^3$ and limit $\alpha\leq\beta$ that, under our assumptions, the function $f_{\beta\gamma\delta}$ is constant on some nontrivial interval $(\eta',\alpha)$.
Again consider the expansion of (\ref{eq:Tr_2_alpha_array}) and observe that Lemma \ref{lem:pairing} continues to apply, so that the expression (\ref{eq:again}) below again decomposes into opposite-signed matching pairs. 
Observe that if the $\eta$ of (\ref{eq:max_L_ns}) is less than $\alpha$, then we may argue just as before.
More generally, consider for $\vec{\gamma}=(\beta,\gamma,\delta)$ and each $i\leq 2$ the tree
$$T_i:=\{\sigma\in\rho_{2,\mathrm{t}}^2(\alpha,\vec{\gamma}^i)\mid\mathrm{L}_2(\alpha,\vec{\gamma}^i)(\sigma)<\alpha\},$$
and note that each $\rho^2_{2,\mathrm{t}}$ decomposes into the downwards-closed subtree $T_i$ together with the union $U_i$ of those $\sigma$ for which $\mathrm{L}_2(\alpha,\vec{\gamma}^i)(\sigma)=\alpha$.
Note also that 
\begin{align}
\label{eq:eta}
\eta:=\max_{i\leq 2} \max\{\mathrm{L}_2(\alpha,\vec{\gamma}^i)(\sigma)\mid\sigma\in T_i\}<\alpha
\end{align}
and that the structure of the $T_i$-indexed portion of $\mathrm{Tr}_2((-1)^i,\xi,\vec{\gamma}^i)$ is stable as $\xi$ ranges within the interval $(\eta,\alpha]$; only the first coordinates of its inputs, in particular, are in motion. Put differently, the only arguments in
\begin{align}
\label{eq:again}
\mathtt{b}\mathrm{Tr}_2(+,\alpha,\gamma,\delta)\sqcup\mathtt{b}\mathrm{Tr}_2(-,\alpha,\beta,\delta)\sqcup
\mathtt{b}\mathrm{Tr}_2(+,\alpha,\beta,\gamma)
\end{align}
which can possibly complicate our desired conclusion are those descending from inputs indexed by some minimal $\sigma\in\bigcup_{i\leq 2}U_i$.
Write $X$ for the collection of such inputs $(\pm,\alpha,\beta'',\gamma'')$ and note that their expansion is always as on the left-hand side of Figure \ref{fig:halfsquarecones}, with $\alpha$ appearing in the second coordinate of exactly two boundary inputs $(\mp,\alpha,\alpha,\beta')$, $(\pm,\alpha,\alpha,\gamma'')$, with a possible chain of boundary inputs intermediating between them (corresponding to the steps of $\mathrm{Tr}(\beta',\beta'')\subseteq\mathrm{Tr}(\alpha,\beta'')$).

A further observation, on which the caption of Figure \ref{fig:halfsquarecones} expands, is in order. There exists an $\eta'<\alpha$ such that the portion of the cone of $(\pm,\alpha,\beta'',\gamma'')$ which is bounded by $(\pm,\alpha,\alpha,\beta')$ and $(\pm,\alpha,\alpha,\gamma'')$ is stable as $\xi$ ranges within $(\eta',\alpha]$; again, nothing but the first ordinal coordinate of that stretch of the cone (colored in black, in Figure \ref{fig:halfsquarecones}) is in motion. In fact, $\mathrm{L}(\beta',\beta'')$ is such an $\eta'$ (if $\beta'=\beta''$ then our claim here is trivially true).
Fix an $\eta^*<\alpha$ above both the $\eta$ of equation \ref{eq:eta} and any $\eta'$ associated as above to some $(\pm,\alpha,\beta'',\gamma'')\in X$.

\vspace{1 cm}

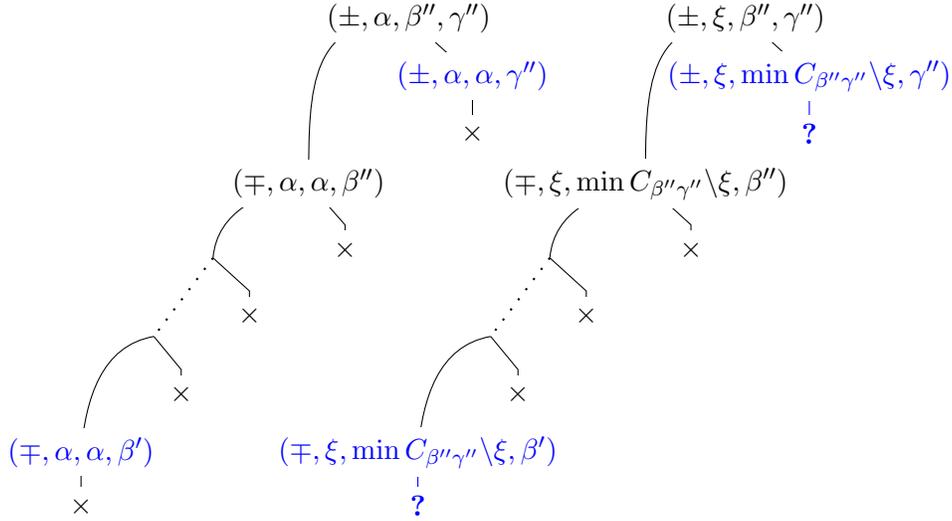
\begin{figure}[h]
\centering
\begin{tikzpicture}[MyPersp,font=\large,fill=white]
	\coordinate (A) at (.5,0,.65);
	\coordinate (A1) at (.5,0,0);
	\coordinate (B) at (1.3,0,2.05);
	\coordinate (B1) at (1.6,0,1.65);
	\coordinate (B2) at (1.6,0,1.35);
	\coordinate (C) at (1.95,0,3);
	\coordinate (C1) at (2.35,0,2.6);
	\coordinate (C2) at (2.35,0,2.3);
	\coordinate (D) at (3,0,3.9);
	\coordinate (D1) at (3.4,0,3.4);
	\coordinate (D2) at (3.4,0,3.1);
	\coordinate (F) at (4.1,0,5.9);
	\coordinate (F1) at (4.8,0,5.2);
	\coordinate (F2) at (4.8,0,4.5);
	
	\coordinate (A') at (4.2,0,.65);
	\coordinate (A1') at (4.2,0,0);
	\coordinate (B') at (5,0,2.05);
	\coordinate (B1') at (5.3,0,1.65);
	\coordinate (B2') at (5.3,0,1.35);
	\coordinate (C') at (5.65,0,3);
	\coordinate (C1') at (6.05,0,2.6);
	\coordinate (C2') at (6.05,0,2.3);
	\coordinate (D') at (6.7,0,3.9);
	\coordinate (D1') at (7.2,0,3.4);
	\coordinate (D2') at (7.2,0,3.1);
	\coordinate (F') at (7.8,0,5.9);
	\coordinate (F1') at (8.5,0,5.2);
	\coordinate (F2') at (8.5,0,4.5);	

\draw (B)--(B1)--(B2);
\draw (C)--(C1)--(C2);
\draw (B2) node[rectangle,fill] {$\times$};
\draw (C2) node[rectangle,fill] {$\times$};
\draw (D)--(D1)--(D2);
\draw (A)--(A1);
\draw (A1) node[rectangle,fill] {$\times$};
\draw (D2) node[rectangle,fill] {$\times$};
\draw (F)--(F1)--(F2);
\draw (F2) node[rectangle,fill] {$\times$};
\draw (F) to[out=180,in=90,distance=1.2cm] (D);
 \draw (D) to[out=-170,in=85] (C);
 \draw[thick, loosely dotted] (C)--(B);
 \draw (B) to[out=-170,in=85] (A);
	\draw (F) node[rectangle,fill] {$(\pm,\alpha,\beta'',\gamma'')$};
	\draw (D) node[rectangle,fill] {$(\mp,\alpha,\alpha,\beta'')$};
	\draw (A) node[rectangle,fill] {{\color{blue} $(\mp,\alpha,\alpha,\beta')$}};
	\draw (F1) node[rectangle,fill] {{\color{blue} $(\pm,\alpha,\alpha,\gamma'')$}};

\draw (B')--(B1')--(B2');
\draw (C')--(C1')--(C2');
\draw (B2') node[rectangle,fill] {$\times$};
\draw (C2') node[rectangle,fill] {$\times$};
\draw (D')--(D1')--(D2');
\draw[blue] (A')--(A1');
\draw (A1') node[rectangle,fill] {$\times$};
\draw (D2') node[rectangle,fill] {$\times$};
\draw (F')--(F1');
\draw[blue] (F1')--(F2');
\draw (F2') node[rectangle,fill] {{\color{blue}\textbf{?}}};
	\draw (F1') node[rectangle,fill] {{\color{blue}$(\pm,\xi,\mathrm{min}\,C_{\beta''\gamma''}\backslash\xi,\gamma'')$}};
 \draw (F') to[out=180,in=90,distance=1.2cm] (D');
 \draw (D') to[out=-170,in=85] (C');
 \draw[thick, loosely dotted] (C')--(B');
 \draw (B') to[out=-170,in=85] (A');
	\draw (A1') node[rectangle,fill] {{\color{blue}\textbf{?}}};
	\draw (A') node[rectangle,fill] {{\color{blue} $(\mp,\xi,\mathrm{min}\,C_{\beta''\gamma''}\backslash\xi,\beta')$}};
	\draw (D') node[rectangle,fill] {$(\mp,\xi,\mathrm{min}\,C_{\beta''\gamma''}\backslash\xi,\beta'')$};
	\draw (F') node[rectangle,fill] {$(\pm,\xi,\beta'',\gamma'')$};
	\end{tikzpicture}
\caption{The (visually downwards, but upwards, in the natural tree-ordering) cone of inputs below a $(\pm,\alpha,\beta'',\gamma'')$ indexed by some minimal $\sigma\in U_i$, together with its image when the initial ordinal coordinate $\alpha$ is replaced by $\xi$. Marked with descent to an ``$\times$'' are terminal inputs, those branching off from the external walk from $\beta''$ to $\alpha$ predominating among them: the rightwards successor of $(\mp,\alpha,\alpha,\beta'')$, for example, is the terminal input $(\mp,\alpha,\mathrm{min}\,C_{\beta''}\backslash\alpha,\beta'')$. Observe that if $\mathrm{min}\,C_{\beta''}\backslash\alpha\neq\alpha$ then $\mathrm{sup}(C_{\beta''}\cap\alpha)<\alpha$ and hence $(\pm,\xi,\mathrm{min}\,C_{\beta''}\backslash\alpha,\beta'')$ is terminal for any $\xi\in (\mathrm{sup}(C_{\beta''}\cap\alpha),\alpha)$ as well, and similarly for the length of the external walk (recall our convention that $\sup\varnothing=0$).
Hence for high enough $\xi<\alpha$, only the right-hand cone's further expansion at the question marks is in question, and this is why it is what is marked in blue --- the inputs $(\pm,\alpha,\alpha,\gamma'')$ and $(\pm,\alpha,\alpha,\beta')$ and their analogues on the right --- that holds, in the general analysis, our main attention.}
\label{fig:halfsquarecones}
\end{figure}

\newpage
The key observations now are
\begin{enumerate}
\item that as $\xi$ ranges within $(\eta^*,\alpha)$, an initial portion of (\ref{eq:Tr_2_array}) is structurally identical to (\ref{eq:Tr_2_alpha_array}) and, consequently,
\item that our task reduces to showing that the portion of (\ref{eq:Tr_2_array}) descending from inputs associated to boundary inputs of the form $(\pm,\alpha,\alpha,\beta^*)$ in (\ref{eq:Tr_2_alpha_array}) is, from a $\rho^2_2$ perspective, stable as $\xi$ ranges within $(\eta^*,\alpha)$.
\end{enumerate}

Let us therefore consider the possibilities:
\begin{itemize}
\item $(\pm,\alpha,\alpha,\beta^*)$ descends from no index $\sigma\in\bigcup_{i\leq 2}U_i$. In this case, the corresponding input in (\ref{eq:Tr_2_array}) will be $(\pm,\xi,\alpha,\beta^*)$, with $\alpha\in C_{\beta^*}$, and there are two sub-possibilities.
Either $\alpha\neq\mathrm{sup}\,C_{\beta^*}\backslash\alpha$ and $(\pm,\xi,\alpha,\beta^*)$ is terminal for sufficiently high $\xi$, or $\alpha=\mathrm{sup}\,C_{\alpha\beta^*}$ and the input expands to the terminal input $(\mp,\xi,\mathrm{min}\,C_\alpha\backslash\xi,\alpha)$ together with $(\pm,\xi,\mathrm{min}\,C_\alpha\backslash\xi,\beta^*)$.\item $(\pm,\alpha,\alpha,\beta^*)$ descends from some such $\sigma$-indexed $(\pm,\alpha,\beta'',\gamma'')$ (to be clear, though, here we are assuming no particular relation between these inputs' signs). In this case, the corresponding input in (\ref{eq:Tr_2_array}) will be
$$(\pm,\xi,\mathrm{min}\,C_{\beta''\gamma''}\backslash\xi,\beta^*)=(\pm,\xi,\mathrm{min}\,C_{\alpha}\backslash\xi,\beta^*),$$ and there are again two sub-possibilities.
If $\alpha\neq\mathrm{sup}\,C_{\beta^*}\backslash\alpha$ then for sufficiently high $\xi$ this input expands as the two terminal inputs $(\pm,\xi,\alpha,\beta^*)$ and $(\pm,\xi,\mathrm{min}\,C_\alpha\backslash\xi,\alpha)$.
If, on the other hand, $\alpha=\mathrm{sup}\,C_{\alpha\beta^*}$, then the only non-terminal inputs descending from the corresponding input in (\ref{eq:Tr_2_array}) descend from $(\pm,\xi,\mathrm{min}\,C_{\alpha}\backslash\xi,\beta^*)$, just as in the previous item. 
\end{itemize}
These observations together imply that there exists an $\eta''<\alpha$ such that the opposite-signed pairings of such $(\pm,\alpha,\alpha,\beta^*)$ within  (\ref{eq:bTr_2_array}) induce opposite-signed matchings of the descendants of the corresponding inputs in (\ref{eq:Tr_2_array}) whenever $\xi\in (\eta'',\alpha)$ --- with only, for each such pairing, the possible exception of a single extra ``step'' (which is independent of the choice of such $\xi$) needed to bring such descendants into alignment. This concludes the proof.
\end{proof}
Due to these potential extra steps, the continuity of the function $f_{\beta\gamma\delta}$ at $\alpha$ will depend on how uniformly $\alpha$ falls among the accumulation points of the relevant clubs; note, though, that this slight nuisance is, by the discussion of Section \ref{subsubsect:square}, as literal a generalization of the classical case as one could demand.
\subsection{Further analysis}
\label{subsect:further_analysis}
A sharper approach to the coherence phenomena described in Section \ref{subsubsect:omega_1} is given by the \emph{full lower trace function $\mathrm{F}$}, which associates to any $\beta\leq\gamma<\omega_1$ a finite set $\mathrm{F}(\beta,\gamma)$ such that
\begin{align}
\label{eq:Fac}
\mathrm{Tr}(\alpha,\gamma) & = \mathrm{Tr}(\mathrm{min}\,\mathrm{F}(\beta,\gamma)\backslash\alpha,\gamma)\cup\mathrm{Tr}(\alpha,\mathrm{min}\,\mathrm{F}(\beta,\gamma)\backslash\alpha),\text{ and} \\
\label{eq:Fab}
\mathrm{Tr}(\alpha,\beta) & = \mathrm{Tr}(\mathrm{min}\,\mathrm{F}(\beta,\gamma)\backslash\alpha,\beta)\cup\mathrm{Tr}(\alpha,\mathrm{min}\,\mathrm{F}(\beta,\gamma)\backslash\alpha)
\end{align}
for any $\alpha\leq\beta$.
To see that such a finite set exists, let $$\mathrm{F}'(\beta,\gamma)=\{\xi<\beta\mid\{\xi\}=\mathrm{Tr}(\xi,\gamma)\cap\mathrm{Tr}(\xi,\beta)\},$$
assume this set to be infinite, and derive a contradiction just as in Section \ref{subsubsect:omega_1}; alternatively, adopt the recursive definition and approach of \cite[\S 2.1]{todwalks} (the former definition is that of \cite{Hudson}).
The coherence of the rho functions is so essentially immediate from (\ref{eq:Fac}, \ref{eq:Fab}) and the finitude of $\mathrm{F}(\beta,\gamma)$, or, more colloquially, from the existence of finite sets mediating between the fiber maps of $\mathrm{Tr}$, that the latter might reasonably be viewed as the more abstract or ``master'' coherence principle of the walks function $\mathrm{Tr}$ itself.
Further significance accrues to the function $\mathrm{F}$ from its centrality to the argument, for example, that $T(\rho_0)$ is Countryman (see again \cite[\S 2.1]{todwalks} or \cite{Hudson}).

As Section \ref{subsect:higher_coherence} should suggest, a similar logic will apply in higher dimensions.
And while a precise formulation of \emph{higher full lower trace functions} $\mathrm{F}_n$ seems more cumbersome  than the present context would merit, they do in principle exist, and warrant further study.\footnote{Another cheap variant of $\mathrm{F}$ is $\mathrm{F}''(\beta,\gamma)=\bigcup_{\xi<\beta}\mathrm{Tr}(\xi,\beta)\triangle\mathrm{Tr}(\xi,\gamma)$, and it is immediate from our arguments above that the higher-dimensional variants of this function will also output finite sets. Optimally, higher $\mathrm{F}_n$ will do more than this: they will also satisfy higher subadditivity relations on the pattern of \cite[Lem.\ 2.1.8]{todwalks}.}
Let us approximate them in the meantime with the observation that the argument of items 1 and 2 of Theorem \ref{thm:n-coherence} was at least as much about the function $\mathrm{Tr}_n$ as the function $\rho_2^n$, sufficiently so to imply suitably formulated coherence relations on other natural $(n+1)$-dimensional generalizations of classical rho functions $\rho_i$. 
Here we will largely leave explorations of the latter to the interested reader.
Observe, for example, that $\mathrm{Tr}_n(+,\vec{\gamma})$ with respect to an ordertype-minimal order-$n$ $C$-sequence on $\kappa^+$ is readily coded by a labeling of the nodes of $\rho_{2,\mathrm{t}}^n(\vec{\gamma})$ with elements of $\kappa$, suggesting natural framings of both $\rho_0^n$ and $\rho_1^n$.
Even under conditions 1 or 2 of Theorem \ref{thm:n-coherence}, this most naive generalization of $\rho_1$, even when suitably algebraicized, will not in general be $n$-coherent in the sense of Definition \ref{def:n_coh} (in particular, it will not be $n$-coherent mod finite).
A more careful framing of $\rho_1^n$, however, one keeping track of the internal branches in which maximal weights appear, will be.





Both for later use and to see the kinds of implications we're describing in action, let us now record the higher-dimensional variants $r_2^n$ of the function $r_2$ introduced in Definition \ref{def:r2}; their $n$-coherence on $\omega_n$, as we will see, follows almost immediately from the recognitions of Section \ref{subsect:higher_coherence} (the same will hold for higher-dimensional variants of $r_1$, but as their nontriviality is less clear, we will ignore them in what follows).
For this purpose it will be useful to revise the conventions introduced just before Definition \ref{def:Ln} to let $(\beta,\vec{\gamma})^\sigma$ to refer the full (unsigned) $\sigma^{\mathrm{th}}$ input of $\mathrm{Tr}_n(\pm,\beta,\vec{\gamma})$.

A natural first attempt to define $r_2^n(\alpha,\vec{\gamma})$ is as $\{(\beta,\vec{\gamma})^\sigma\mid\mathrm{L}_n(\beta,\vec{\gamma})(\sigma)=\alpha\}$.
Note that the $n=1$ case indeed coincides with that of Definition \ref{def:r2}.
For higher $n$, however, the possibility of repetitions together with the more unavoidably additive character of $n$-coherence suggest the following algebraic variation.
\begin{definition}
\label{def:r2n}
For any $n>0$ and ordinal $\varepsilon$, define the function $r_2^n:[\varepsilon]^{n+1}\to\bigoplus_{\varepsilon^{[n+1]}}\mathbb{Z}$ as follows:
$$r_2^n(\alpha,\vec{\gamma})=\sum_{(\beta,\sigma)\in s(\alpha,\vec{\gamma})}(-1)^{i(\beta,\sigma)}\lfloor(\beta,\vec{\gamma})^\sigma\rfloor,$$
where
$$s(\alpha,\vec{\gamma})=\{(\beta,\sigma)\in (\alpha,\gamma_0)\times n^{<\omega}\mid \sigma\in\rho_{2,\mathrm{t}}^n(\beta,\vec{\gamma})\text{ and }\mathrm{L}_n(\beta,\vec{\gamma})(\sigma)=\alpha\}$$
and $i(\beta,\sigma)$ records the parity of $(\beta,\vec{\gamma})^\sigma$ in $\mathrm{Tr}_n(+,\beta,\vec{\gamma})$.
\end{definition}
\begin{lemma}
\label{lem:r^n_2_coherence}
Under the $C$-sequence conditions 1 or 2 of Theorem \ref{thm:n-coherence}, the $r_2^n$-fiber maps
$$\langle r_2^n(\,\cdot\,,\vec{\gamma}):\gamma_0\to\bigoplus_{\omega_n^{[n+1]}}\mathbb{Z}\mid\vec{\gamma}\in[\omega_n]^n\rangle$$
form a mod finite $n$-coherent family of functions.
\end{lemma}
\begin{proof}
Suppose not. Then there exist a $\vec{\delta}\in [\omega_n]^{n+1}$ and increasing $\langle\alpha_j\mid j\in\omega\rangle$ with supremum $\alpha\leq\delta_0$ such that
\begin{align}
\label{eq:r2nsum}
\sum_{i=0}^n(-1)^i r_2^n(\alpha_j,\vec{\delta}^i)\neq 0\end{align}
for all $j\in\omega$. By Lemma \ref{lem:max_L<limits},
$$\eta:=\max_{i\leq n}\,(\mathrm{max}\,\mathrm{L}_n(\alpha,\vec{\delta}^i))<\alpha,$$
which, by Lemma \ref{lem:pairing}, implies that for any $\eta<\xi<\beta\leq\alpha$, those $(\beta,\vec{\delta}^i)^\sigma$ with $\mathrm{L}_n(\beta,\vec{\delta}^i)(\sigma)=\xi$ appearing among the arguments of $$\bigsqcup_{i\leq n}\mathrm{Tr}_n((-1)^n,\beta,\vec{\delta}^i)$$
arise in opposite-signed pairs. It follows that the expression in (\ref{eq:r2nsum}) equals zero for all $\alpha_j>\eta$, a contradiction.
\end{proof}
Composing $r^n_2$ with a quotient of its codomain by $\bigoplus_{\omega_n^{[n+1]}}2\mathbb{Z}$
yields a free-$\mathbb{Z}/2$-module-valued function, one readily identified with a set-valued function by identifications as in Theorem \ref{thm:ri}, and our $r^n_2$ analysis here and in Section \ref{sect:nontriviality} will apply to this function as well (in particular, under the same assumptions, its fiber maps will also form a nontrivial $n$-coherent family on $\omega_n$).
Passages to $\mathbb{Z}/2$ like this can often, by paring back the algebraic paraphernalia of higher walks, render their combinatorial essentials plainer; as here, for example, they tend to allow one to bypass the issue of signs.
If we have declined to take such an approach earlier, though, it is because we regard signs themselves as combinatorially significant information: note that one cannot define the function $\rho_2^n$ without them.
\subsubsection{Countable variation and larger moduli}
\label{subsubssect:ctblvar}
One somewhat surprising upshot of the analysis of Section \ref{subsect:higher_coherence} is the following: when defined with respect to an order-$n$ $\mathcal{C}$-sequence as in item 2 of Theorem \ref{thm:n-coherence}, the upper trace fiber maps $\alpha\mapsto\mathrm{Tr}_n(+,\alpha,\vec{\gamma})$ (and hence those of the associated rho functions as well) are all \emph{countably varying}.
This follows from the reasoning behind Lemmas \ref{lem:max_L<limits} and \ref{lem:end-extension}; more precisely, fix a cofinality-$\omega_1$ $\beta\leq\gamma_0$ and apply the first line of the proof of Lemma \ref{lem:end-extension} together with the observation that $\mathrm{sup}(C_{\vec{\gamma}^\sigma}\cap\beta)<\beta$ for any other $\vec{\gamma}^\sigma$ arising in $\mathrm{Tr}_n(+,\beta,\vec{\gamma})$. Conclude that $\mathrm{Tr}_n(\,\cdot\,,\vec{\gamma})$ is eventually constant below $\beta$.

In the course, in other words, of generalizing various other features of classical walks on $\omega_1$, a main instance of the function $\mathrm{Tr}_n$ on $\omega_n$ generalizes the countable variation of their fiber maps as well.
Unlike those of the former, however, when $n>1$ fiber maps of the latter will have uncountable domains, with the consequence that they are constant on long (i.e., uncountable) intervals within them. From a cohomological perspective, such countable variation witnesses to $\mathrm{H}^n(\omega_n;\mathcal{A})\neq 0$, or countable support witnesses to $\check{\mathrm{H}}^n(\omega_n;\mathcal{F}_A)\neq 0$, are to be expected; fuller explanation here would take us too far afield, but the issue receives its due attention in \cite{BLHZ}.
The more immediate point, though, is simply that this is one of several features lending the combinatorics of higher $\mathrm{Tr}_n$ a distinctive flavor, one in which pigeonhole conditions like \emph{finite-to-one}-ness, for example, seem to play less decisive roles than than classically.\footnote{For another relevant sense in which \emph{finite-to-one} arguments break down above $\omega_1$, see \cite[Thm.\ 3.7]{Konig}.}

A not unrelated point is the following: while Section \ref{subsubsect:square} hinted at how $\mathsf{ZFC}$ nontrivial coherence arguments lift to higher cardinals under expanded assumptions, Section \ref{subsect:higher_coherence} showed how those arguments lift to higher cardinals via expansions of dimension.
The focus in either instance was on \emph{mod locally constant} coherence, for the reason that this modulus, being essentially finitary (mod finite coherence is a special case), carries the strongest implications.
As noted above, though, a third approach to extending walks techniques above $\omega_1$ is by relaxing the modulus; an ordertype-minimal classical $C$-sequence on $\omega_2$, for example, induces both \emph{modulo countable differences} and \emph{modulo countably varying functions} nontrivial coherent families on $\omega_2$, for the reason that at cofinality-$\omega_1$ points below $\omega_2$, the classical arguments of Section \ref{subsubsect:omega_1} will continue to apply.
Note that suitably adapted versions of Lemmas \ref{lem:max_L<limits} and \ref{lem:end-extension} then extend these arguments to $\mathrm{Tr}_2$ on $\omega_3$, for example, and more broadly imply the following generalization of Theorem \ref{thm:n-coherence}'s item 1:
\begin{theorem}
For any $n\geq k>0$, if $\mathcal{C}$ is an order-$k$ ordertype-minimal $C$-sequence on $\omega_n$ then the fiber maps of the associated function $\rho_2^k:[\omega_n]^{k+1}\to\mathbb{Z}$ form a $k$-coherent family of functions with respect to the modulus of functions varying less than $\omega_{n-k}$ many times.
\end{theorem}
The proof is a straightforward enough adaptation of this section's arguments that it is left to the interested reader.
\subsubsection{Pairings in practice}
\label{subsubsect:pairings}

Lemma \ref{lem:pairing}, and more particularly the remark numbered (1) just before it, described an opposite-signed pairing of internal walks within the triple
\begin{equation*}
\mathrm{Tr}_2(+,\alpha,\gamma,\delta)\sqcup\mathrm{Tr}_2(-,\alpha,\beta,\delta)\sqcup
\mathrm{Tr}_2(+,\alpha,\beta,\gamma)
\end{equation*}
of expression \ref{eq:Tr_2_alpha_array}. It is instructive to consider the mechanics of this pairing more closely. Such a pairing of $C_\delta$-internal walks plainly appears at the only place it can --- at the outset of the expansion of this triple --- and is denoted with dotted blue arrows in Figure \ref{fig:pairs1}.
Less obvious is where a negative-signed input with $\gamma$ as last coordinate will appear in the expansion of (\ref{eq:Tr_2_alpha_array}) to pair with $(+,\alpha,\beta,\gamma)$, and it is to this question that we now turn our attention.

The natural first place to seek such an input is in the expansion of $\mathrm{Tr}_2(+,\alpha,\gamma,\delta)$ --- and indeed, letting $\mathrm{Tr}(\gamma,\delta)=\{\delta_0,\dots,\delta_k\}$ (in descending enumeration), if $C_{\gamma\delta_{k-1}}\backslash\alpha\neq\varnothing$ then letting $\alpha'=\min(C_{\gamma\delta_{k-1}}\backslash\alpha)$, the input $(-,\alpha,\alpha',\gamma)$ does appear within $\mathrm{Tr}_2(+,\alpha,\gamma,\delta)$ at the end of the external walk from $\delta$ down to $\gamma$, as desired. See again Figure \ref{fig:pairs1}.

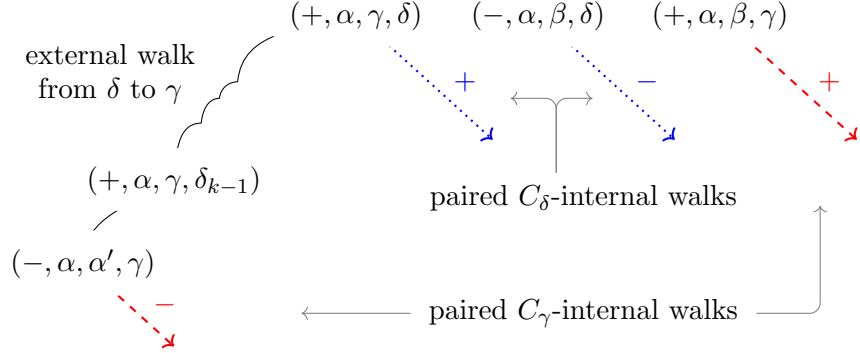
\begin{figure}[h]
\centering
\begin{tikzpicture}[MyPersp,font=\large,fill=white]
	\coordinate (A) at (4,0,5);
	\coordinate (B) at (6,0,5);
	\coordinate (C) at (8,0,5);

	\coordinate (A') at (5.5,0,3.5);
	\coordinate (B') at (7.5,0,3.5);
	\coordinate (C') at (9.5,0,3.5);
	
	\coordinate (D) at (2,0,3);
	\coordinate (E) at (1,0,2);	
	\coordinate (F) at (2,0,1);	
	
	\draw[blue, thick, dotted, ->] (A)--(A');
	\draw[blue, thick, dotted, ->] (B)--(B');
	\draw[red, thick, dashed, ->] (C)--(C');	
	
	\draw (D) to[out=-150,in=90] (E);
	\draw[red, thick, dashed, ->] (E)--(F);
		\draw (A) to[out=-165,in=85] (2.7,0,4.2);
		\draw (2.7,0,4.2) to[out=-180,in=85] (2.5,0,4);
				\draw (2.5,0,4) to[out=-180,in=85] (2.3,0,3.7);
				\draw (2.3,0,3.7) to[out=-180,in=85] (2,0,3);
	\draw (A) node[ellipse,fill] {$(+,\alpha,\gamma,\delta)$};
	\draw (B) node[ellipse,fill] {$(-,\alpha,\beta,\delta)$};
	\draw (C) node[ellipse,fill] {$(+,\alpha,\beta,\gamma)$};
	\draw (D) node[ellipse,fill] {$(+,\alpha,\gamma,\delta_{k-1})$};
	\draw (E) node[ellipse,fill] {$(-,\alpha,\alpha',\gamma)$};
	\draw (1.3,0,4.3) node[align=center] {external walk \\
	 from $\delta$ to $\gamma$};
	 \draw (6.5,0,2.8) node[align=center] {paired $C_\delta$-internal walks};
	 \draw (1.9,0,1.5) node[align=center] {\color{red} $-$};
	 	 \draw (9.2,0,4.2) node[align=center] {\color{red} $+$};
	 	 	 \draw (5.2,0,4.2) node[align=center] {\color{blue} $+$};
	 	 	 	 \draw (7.2,0,4.2) node[align=center] {\color{blue} $-$};
	 \draw[gray,->] (4.6,0,1.4)--(3.4,0,1.4);
	 	 \draw[gray, rounded corners, ->] (8.4,0,1.4)--(9.1,0,1.4)--(9.1,0,2.7);
	 	 \draw[gray, rounded corners, ->] (6.2,0,3.1)--(6.2,0,4)--(5.7,0,4);
	 	 \draw[gray, rounded corners, ->] (6.2,0,3.1)--(6.2,0,4)--(6.6,0,4);
	 	 \draw (6.5,0,1.4) node[align=center] {paired $C_\gamma$-internal walks};
	\end{tikzpicture}
	\caption{The negative-signed $C_\gamma$-internal walk pairing with that deriving from $(+,\alpha,\beta,\gamma)$ when $C_{\gamma\delta_{k-1}}\backslash\alpha\neq\varnothing$: a mapping of relevant inputs.}
	\label{fig:pairs1}
	\end{figure}

If on the other hand $C_{\gamma\delta_{k-1}}\backslash\alpha=\varnothing$ (or in other words if $C_{\delta_{k-1}}\cap [\alpha,\gamma)=\varnothing$) then $(+,\alpha,\gamma,\delta_{k-1})$ is a terminal node. We mark such nodes with a descending line to an ``$\times$''. If in this case $C_{\delta_{k-2}}\cap [\alpha,\gamma)\neq\varnothing$ then, letting $\alpha''=\min(C_{\delta_{k-1}\delta_{k-2}}\backslash\alpha)$, observe that the input $(-,\alpha,\alpha'',\gamma)$ must appear within $\mathrm{Tr}_2(+,\alpha,\gamma,\delta)$ at the last branch off of the classical walk from $\delta$ down to $\gamma$, as desired. See Figure \ref{fig:pairs2}.

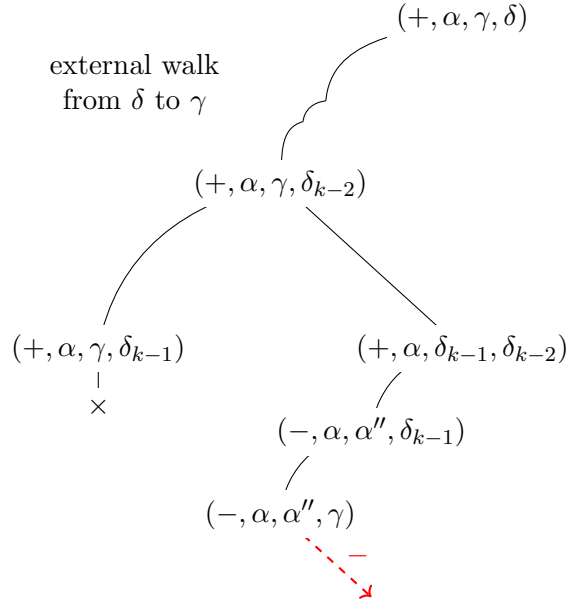
\begin{figure}[h]
\centering
\begin{tikzpicture}[MyPersp,font=\large,fill=white]
	\coordinate (A) at (6,0,7.5);
	\coordinate (B) at (4,0,5.5);
	\coordinate (C) at (2,0,3.5);

	\coordinate (A') at (6,0,3.5);
	\coordinate (B') at (5,0,2.5);
	\coordinate (C') at (4,0,1.5);
	
	\coordinate (D) at (5,0,.5);

	\draw[red, thick, dashed, ->] (C')--(D);	
	\draw (B)--(A');
	\draw (B) to[out=-170,in=80] (C);
	\draw (A) to[out=-165,in=85] (4.5,0,6.5);
	\draw (4.5,0,6.5) to[out=-175,in=85] (4.25,0,6.25);
				\draw (4.25,0,6.25) to[out=-165,in=85] (B);
				\draw (A') to[out=-180,in=85] (B');
				\draw (B') to[out=-180,in=85] (C');
				\draw (C)--(2,0,2.8);
	\draw (A) node[rectangle,fill] {$(+,\alpha,\gamma,\delta)$};
	\draw (B) node[rectangle,fill] {$(+,\alpha,\gamma,\delta_{k-2})$};
	\draw (C) node[rectangle,fill] {$(+,\alpha,\gamma,\delta_{k-1})$};
	\draw (A') node[rectangle,fill] {$(+,\alpha,\delta_{k-1},\delta_{k-2})$};
	\draw (B') node[rectangle,fill] {$(-,\alpha,\alpha'',\delta_{k-1})$};
	\draw (C') node[rectangle,fill] {$(-,\alpha,\alpha'',\gamma)$};
	\draw (2,0,2.8) node[rectangle,fill] {$\times$};
	\draw (4.85,0,1) node[align=center] {\color{red} $-$};
	\draw (2.4,0,6.7) node[align=center] {external walk \\ from $\delta$ to $\gamma$};
	\end{tikzpicture}
\caption{Marked again with a dashed red arrow is the $C_\gamma$-internal walk pairing with that deriving from $(+,\alpha,\beta,\gamma)$ when $C_{\gamma\delta_{k-1}}\backslash\alpha=\varnothing$ but $C_{\delta_{k-2}}\cap [\alpha,\gamma)\neq\varnothing$.}
\label{fig:pairs2}
\end{figure}

Similarly, if $C_{\delta_{k-1}}\cap [\alpha,\gamma)=C_{\delta_{k-2}}\cap [\alpha,\gamma)=\varnothing$ then suppose $C_{\delta_j}\cap [\alpha,\gamma)\neq\varnothing$ for some least $\delta_j\in\mathrm{Tr}(\gamma,\delta)\backslash\{\gamma\}$. Then again $\gamma$ appears as the last coordinate of a negative-signed input on the last branch off of the classical walk from $\delta$ down to $\gamma$, as desired. See Figure \ref{fig:pairs3}.

\begin{figure}[h]
\centering
\begin{tikzpicture}[MyPersp,font=\large,fill=white]
	\coordinate (A) at (.5,0,3.5);
	\coordinate (B) at (2,0,5);
	\coordinate (C) at (3.3,0,6.3);
	\coordinate (D) at (4.2,0,6.9);
	\coordinate (E) at (5.3,0,8);
	\coordinate (F) at (6.5,0,9);
	\coordinate (F') at (7.1,0,9.6);
	\coordinate (G') at (7.5,0,9.9);
	\coordinate (G) at (8.5,0,10.5);
	\coordinate (H) at (9,0,7);
	\coordinate (I) at (8.5,0,6);
	\coordinate (J) at (7.6,0,5);
	\coordinate (J') at (7,0,4.2);
	\coordinate (K) at (6.4,0,3.5);
	\coordinate (L) at (5.5,0,2.5);
	\coordinate (L') at (6.8,0,1.35);
	\coordinate (M) at (4.5,0,1.5);
	\coordinate (N) at (5.5,0,.5);
	\coordinate (O) at (3.7,0,3.5);

	\draw (F)--(H);
	\draw (B)--(O);
	\draw[thick, loosely dotted] (J')--(K);
	\draw (A)--(.5,0,2.8);
	\draw (O)--(3.7,0,2.8);
	\draw (.5,0,2.8) node[rectangle,fill] {$\times$};
	\draw (3.7,0,2.8) node[rectangle,fill] {$\times$};
		\draw (C)--(4.3,0,5.3)--(4.3,0,4.8);
	\draw (D)--(5.2,0,5.9)--(5.2,0,5.4);
	\draw (4.3,0,4.8) node[rectangle,fill] {$\times$};
		\draw (K)--(7.4,0,2.5)--(7.4,0,2);
	\draw (7.4,0,2) node[rectangle,fill] {$\times$};
		\draw (J')--(7.8,0,3.4)--(7.8,0,2.9);
	\draw (7.8,0,2.9) node[rectangle,fill] {$\times$};
	\draw (J)--(8.3,0,4.3)--(8.3,0,3.8);
	\draw (8.3,0,3.8) node[rectangle,fill] {$\times$};
	\draw (5.2,0,5.4) node[rectangle,fill] {$\times$};
	\draw (E)--(6.3,0,7)--(6.3,0,6.5);
	\draw (6.3,0,6.5) node[rectangle,fill] {$\times$};
	\draw (B) to[out=-180,in=85] (A);
	\draw (C) to[out=-180,in=85] (B);
	\draw (F') to[out=-180,in=85] (F);
	\draw (G) to[out=-180,in=85] (G');
	\draw (E) to[out=-170,in=85] (D);
	\draw (F) to[out=-170,in=85] (E);
	\draw (J) to[out=-170,in=85] (J');
	\draw (H) to[out=-180,in=85] (I);
	\draw[thick, loosely dotted] (D)--(C);
	\draw[thick, loosely dotted] (G')--(F');
	\draw (I) to[out=-170,in=85] (J);
	\draw (K) to[out=-180,in=85] (L);
	\draw (L) to[out=-180,in=85] (M);
	\draw[red, thick, dashed, ->] (M)--(5.5,0,.5);	
	\draw (G) node[rectangle,fill] {$(+,\alpha,\gamma,\delta)$};
	\draw (F) node[rectangle,fill] {$(+,\alpha,\gamma,\delta_j)$};
	\draw (H) node[rectangle,fill] {$(+,\alpha,\delta_{j+1},\delta_j)$};
	\draw (I) node[rectangle,fill] {$(-,\alpha,\alpha''',\delta_{j+1})$};
\draw (L)--(L')--(6.8,0,.65);	
	\draw (L) node[rectangle,fill] {$(-,\alpha,\alpha''',\delta_{k-1})$};
		\draw (L') node[rectangle,fill] {$(-,\alpha,\gamma,\delta_{k-1})$};
		\draw (6.8,0,.65) node[rectangle,fill] {$\times$};
	\draw (M) node[rectangle,fill] {$(-,\alpha,\alpha''',\gamma)$};
	\draw (B) node[rectangle,fill] {$(+,\alpha,\gamma,\delta_{k-2})$};
	\draw (A) node[rectangle,fill] {$(+,\alpha,\gamma,\delta_{k-1})$};
	\draw (O) node[rectangle,fill] {$(+,\alpha,\delta_{k-1},\delta_{k-2})$};
	\draw (5.35,0,1) node[align=center] {\color{red} $-$};
	\draw (2.5,0,8.8) node[align=center] {external walk \\ from $\delta$ to $\gamma$};
	\end{tikzpicture}
\caption{The $C_\gamma$-internal walk pairing with that deriving from $(+,\alpha,\beta,\gamma)$ when $C_{\delta_{k-1}}\cap [\alpha,\gamma)=C_{\delta_{k-2}}\cap [\alpha,\gamma)=\varnothing$ and $j=\max\{i\mid i<k\text{ and } C_{\delta_i}\cap[\alpha,\gamma)\neq\varnothing\}$. Here $\alpha'''=\min(C_{\delta_{j+1}\delta_j}\backslash\alpha)$. Marked with descent to an ``$\times$'' are some of those nodes ensured by our assumptions to be terminal. Observe the appearance, in the lower left and lower right corners, of opposite-signed instances of $(\alpha,\gamma,\delta_{k-1})$.}
\label{fig:pairs3}
\end{figure}
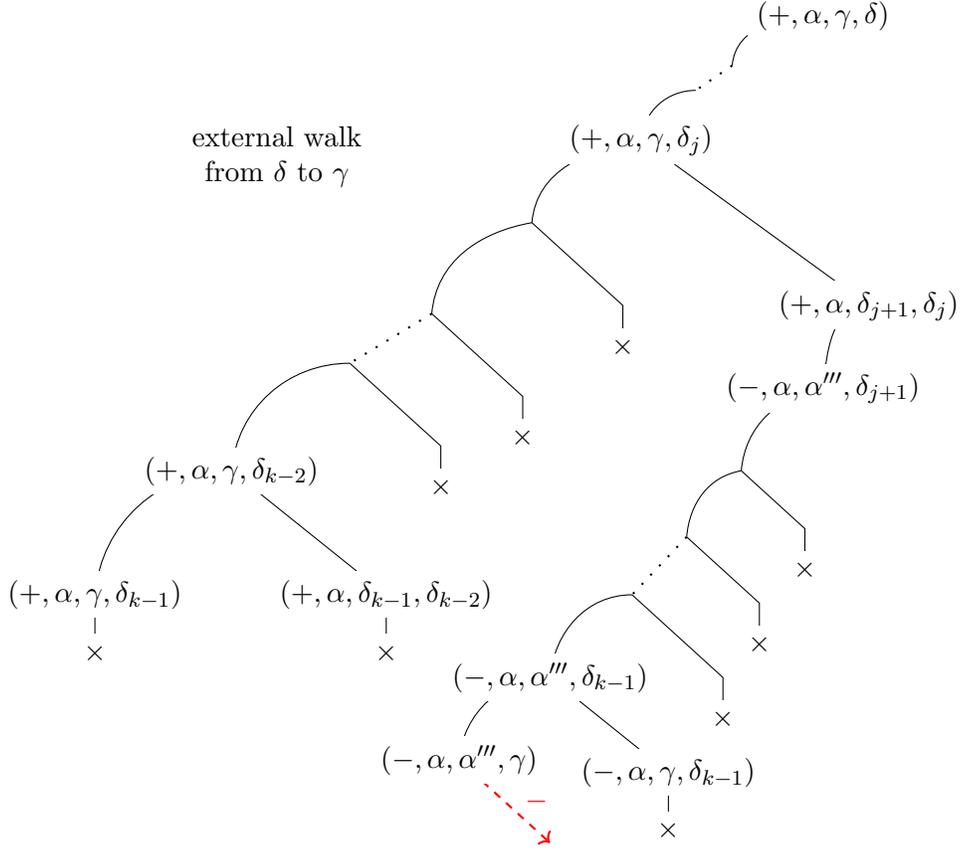

Finally, consider the possibility that $C_{\delta_j}\cap [\alpha,\gamma)=\varnothing$ for all $\delta_j\in\mathrm{Tr}(\gamma,\delta)\backslash\{\gamma\}$. In this case, $\gamma$ will fail to appear as as a last input-coordinate anywhere within the expansion of $\mathrm{Tr}_2(+,\alpha,\gamma,\delta)$. It is certain instead to appear, though, in the expansion of $\mathrm{Tr}_2(-,\alpha,\beta,\delta)$, as desired. See Figure \ref{fig:pairs4}.

\begin{figure}[h]
\centering
\begin{tikzpicture}[MyPersp,font=\large,fill=white]
	\coordinate (A) at (.5,0,2.5);
	\coordinate (A1) at (.5,0,1.85);
	\coordinate (B) at (1.35,0,3.45);
	\coordinate (B1) at (1.75,0,3.05);
	\coordinate (B2) at (1.75,0,2.75);
	\coordinate (C) at (1.85,0,3.8);
	\coordinate (C1) at (2.25,0,3.4);
	\coordinate (C2) at (2.25,0,3.1);
	\coordinate (D) at (2.8,0,4.6);
	\coordinate (D1) at (3.2,0,4.2);
	\coordinate (D2) at (3.2,0,3.9);
	\coordinate (E) at (3.2,0,5.1);
	\coordinate (E1) at (3.6,0,4.7);
	\coordinate (E2) at (3.6,0,4.4);
	\coordinate (F) at (4.1,0,5.9);
	\coordinate (F1) at (4.6,0,5.4);
	\coordinate (F2) at (4.6,0,5);
	
	\coordinate (A') at (4.5,0,2.5);
	\coordinate (A1') at (3.5,0,1.5);
	\coordinate (A2') at (4.5,0,.5);
	\coordinate (B') at (5.35,0,3.45);
	\coordinate (B1') at (5.75,0,3.05);
	\coordinate (B2') at (5.75,0,2.75);
	\coordinate (C') at (5.85,0,3.8);
	\coordinate (C1') at (6.25,0,3.4);
	\coordinate (C2') at (6.25,0,3.1);
	\coordinate (D') at (6.8,0,4.6);
	\coordinate (D1') at (7.2,0,4.2);
	\coordinate (D2') at (7.2,0,3.9);
	\coordinate (E') at (7.2,0,5.1);
	\coordinate (E1') at (7.6,0,4.7);
	\coordinate (E2') at (7.6,0,4.4);
	\coordinate (F') at (8.1,0,5.9);
	\coordinate (F1') at (8.6,0,5.4);
	\coordinate (F2') at (8.6,0,5);	

\draw (B)--(B1)--(B2);
\draw (C)--(C1)--(C2);
\draw (B2) node[rectangle,fill] {$\times$};
\draw (C2) node[rectangle,fill] {$\times$};
\draw (D)--(D1)--(D2);
\draw (E)--(E1)--(E2);
\draw (A)--(A1);
\draw (A1) node[rectangle,fill] {$\times$};
\draw (D2) node[rectangle,fill] {$\times$};
\draw (E2) node[rectangle,fill] {$\times$};
\draw (F)--(F1)--(F2);
\draw (F2) node[rectangle,fill] {$\times$};
 \draw (F) to[out=-160,in=80] (E);
 \draw (D) to[out=-175,in=85] (C) to[out=-180,in=85] (B) to[out=-180,in=85] (A);
 \draw[thick, loosely dotted] (E)--(D);
	\draw (A) node[rectangle,fill] {$(+,\alpha,\gamma,\delta_{k-1})$};
	\draw (F) node[rectangle,fill] {$(+,\alpha,\gamma,\delta)$};

\draw (B')--(B1')--(B2');
\draw (C')--(C1')--(C2');
\draw (B2') node[rectangle,fill] {$\times$};
\draw (C2') node[rectangle,fill] {$\times$};
\draw (D')--(D1')--(D2');
\draw (E')--(E1')--(E2');
\draw (A') to[out=-170,in=85] (A1');
\draw (A1') node[rectangle,fill] {$\times$};
\draw (D2') node[rectangle,fill] {$\times$};
\draw (E2') node[rectangle,fill] {$\times$};
\draw (F')--(F1')--(F2');
\draw (F2') node[rectangle,fill] {$\times$};
 \draw (F') to[out=-160,in=80] (E');
 \draw (D') to[out=-175,in=85] (C') to[out=-180,in=85] (B') to[out=-180,in=85] (A');
 \draw (A')--(5.5,0,1.5)--(5.5,0,.85);
 \draw (5.5,0,.85) node[rectangle,fill] {$\times$};
 \draw (5.5,0,1.5) node[rectangle,fill] {$(-,\alpha,\gamma,\delta_{k-1})$};
 \draw[thick, loosely dotted] (E')--(D');
	\draw (A') node[rectangle,fill] {$(-,\alpha,\beta,\delta_{k-1})$};
	\draw[red, thick, dashed, ->] (A1')--(A2');
	\draw (A1') node[rectangle,fill] {$(-,\alpha,\beta,\gamma)$};
	\draw (F') node[rectangle,fill] {$(-,\alpha,\beta,\delta)$};
		\draw (4.35,0,1.05) node[align=center] {\color{red} $-$};
	\end{tikzpicture}
\caption{The $C_\gamma$-internal walk pairing with that deriving from $(+,\alpha,\beta,\gamma)$ when $C_{\delta_j}\cap [\alpha,\gamma)=\varnothing$ for all $\delta_j\in\mathrm{Tr}(\gamma,\delta)\backslash\{\gamma\}$. In this case the external walk from $\delta$ to $\beta$ must first copy the walk from $\delta$ to $\gamma$; at the end of this copy, $\gamma$ will appear in the last coordinate position, as desired. There, in fact, the initial input $(\alpha,\beta,\gamma)$ of (\ref{eq:Tr_2_alpha_array}) reappears with opposite sign, so in this case, the elements of (\ref{eq:Tr_2_alpha_array}) completely cancel one another out.}
\label{fig:pairs4}
\end{figure}
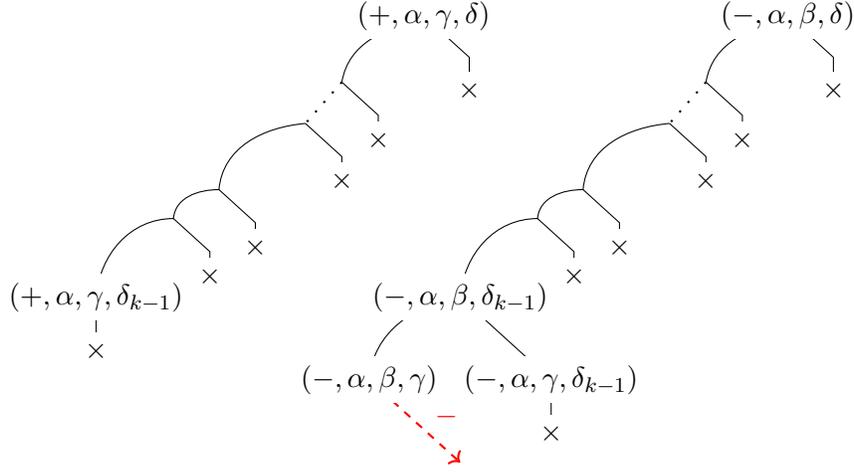

This kind of ``combinatorial peekaboo'' or exhaustion of hypotheses can grow even more elaborate deeper within the expansion of (\ref{eq:Tr_2_alpha_array}) --- but we know by Lemma \ref{lem:pairing} that it must always end happily somewhere. We have included an account of the most basic case for the further light it sheds on the general character of higher walks.

\section{A sample higher walk}
\label{sect:explorations}
This section extends the impulse of the previous subsection, which was, in essence, to expand a generic $\mathrm{Tr}_2(+,\alpha,\beta,\gamma)$ just to see what we find.
Much of what we do we record in Figures \ref{fig:app1}, \ref{fig:app2}, \ref{fig:app3}, and \ref{fig:app4}, which, to preserve the flow of text, we have collected in this paper's appendix; these in turn tend to register as a \emph{visual} argument that we are, with higher walks, dealing with something nontrivial indeed. The complexity of $\mathrm{Tr}_2$ derives broadly from two phenomena without precedent in the classical case: \emph{repetitions of outputs}, and \emph{depth}. Each calls for further investigation; our aim in the present section is simply to convey something of their significance and extent. 
\subsection{Repetitions}
\label{subsect:sample}
It's instructive to begin by adopting any of a number of simplifying assumptions. See Exercise \ref{ex:rho2unbd} below for some of the most radical such assumptions: that $\alpha\leq\beta\leq\gamma<\omega$, or that $\gamma=C_\gamma=\omega$ or $\omega_1$, for example.
Condition 2 of Section \ref{subsect:Higher_C}, more mildly, will ensure that any initial step ``into'' a $C_\xi$-internal walk (any output whose index ends in $0$, in other words) is to a successor ordinal, while assuming $\mathcal{C}$ ordertype-minimal will ensure in any case that the next internal step (corresponding to an index ending in $01$), if defined, is to the next highest element of $C_\xi$. For such $\mathcal{C}$, these are the only sorts of $C_\xi$-internal steps possible when $\mathrm{cf}(\xi)\leq\omega$, so that, in particular, the internal walks of $\mathrm{Tr}_2(+,\alpha,\beta,\gamma)$ induced by an ordertype-minimal $C$-sequence on $\omega_1$ all just stepwise descend through the relevant ``ladders'' $C_\xi$ down to $\mathrm{min}\,C_\xi\backslash\alpha$. The only effect of raising $\alpha$, in consequence, is to truncate $\mathrm{Tr}_2(+,\alpha,\beta,\gamma)$; put differently, each $\mathrm{Tr}_2(+,\alpha,\beta,\gamma)$ in this setting is simply an initial subtree of $\mathrm{Tr}_2(+,0,\beta,\gamma)$, lending $\mathrm{Tr}_2$ on the countable ordinals a somewhat $2$-dimensional character.
For this reason, we tend to think of Figures \ref{fig:app1}, \ref{fig:app2}, \ref{fig:app3}, and \ref{fig:app4} as together recording a representative expansion in, more interestingly, the $\omega_2$ setting, but they may be taken as such for any other uncountable cardinal $\kappa$ as well.
These figures' basic structure is as follows.

Figures \ref{fig:app1} and \ref{fig:app2} (which together comprise a single figure) record a $\mathrm{Tr}_2(+,\alpha,\beta,\gamma)$ expansion which is patterned on that of Figure \ref{asampletr2}, but more extensive, and this time includes the data of signs. In Figures \ref{fig:app3} and \ref{fig:app4}, only the outputs of that expansion are directly depicted, labeled by their binary string indices and arrayed within what, visibly, is a chandelier system of external walks. (Space necessitates a choice of emphasis: charts according internal walks greater primacy are, in principle, of roughly equal relevance, but are left to the motivated reader.) Some deliberate color-codings hopefully heighten the coordination of these figures, as explained in their captions.
Despite our external emphasis, for example, internal walks are visible both as the warm-colored rightwards branches of Figures \ref{fig:app1} and \ref{fig:app2} and as the sequences of indices descending any of the warm-colored clubs engaged in the course of the external walks of Figure \ref{fig:app3} (the $0001$, $00011$, $000111$ descending $C_{\beta_{000}}$, for example).
As this observation underscores, the organizing substructures of $\mathrm{Tr}_n(+,\alpha,\beta,\gamma)$ are legible at the level of indices: internal walks manifest as lengthening tails of $1$s ($\sigma^\frown 1$, $\sigma^\frown 11$, $\sigma^\frown 111$ for some $\sigma\in\rho^2_{2,\mathrm{t}}$), while external walks manifest as lengthening tails of $0$s.
Alternations of $0$s and $1$s within the indices of $\mathrm{Tr}_n(+,\alpha,\beta,\gamma)$ forms a deep subject, and is that, roughly, of the following subsection.

Readers may benefit from tracking a few steps in parallel through Figures \ref{fig:app1}, \ref{fig:app2}, \ref{fig:app3}, and \ref{fig:app4}; to aid in this, we record a few elementary observations:
\begin{enumerate}
\item The expansion of $\mathrm{Tr}_2(+,\alpha,\beta,\gamma)$ at any given input $(\pm,\alpha,\beta',\gamma')$ is a product of assumptions about $\mathcal{C}$; we are spared the unpleasantness of recording these separately, though, by the fact that the relevant assumptions may always be read off from the node
$$\frac{(\pm,\alpha,\beta',\gamma')}{\pm\beta''}$$
itself. If the signs above and below the bar agree, for example, then $\beta'\in C_{\gamma'}$ and $\beta''=\mathrm{min}\,C_{\beta'\gamma'}\backslash\alpha$.
See Figure \ref{fig:app2}'s caption for the other possibilities, for the tracking of which the equation of Example \ref{ex:Tr2} may prove handy.
\item As noted in Section \ref{sect:Trn}, both inputs' and outputs' signs are constant for the portions of leftwards branches of $\mathrm{Tr}_2(+,\alpha,\beta,\gamma)$ naturally identifying with classical walks; these constant signs are not, however, the same. Those appearing alongside these walks in Figures \ref{fig:app3} and \ref{fig:app4} are the outputs' stable signs.
\item By and large, and visibly within Figures \ref{fig:app3} and \ref{fig:app4}, the ordinals $\beta_\sigma$ descend as their indices $\sigma$ lengthen, with one critical sort of exception. Consider, for instance, the input $(+,\alpha,\beta_0,\beta_{\varnothing})$ at position $01$ in Figure \ref{fig:app2} and its output ordinal $\beta_{01}=\mathrm{min}\,C_{\beta_0\beta_\varnothing}\backslash\alpha$.
Its successor input in the leftwards direction is $(-,\alpha,\beta_{01},\beta_0)$, initiating a classical walk from $\beta_0$ down to $\beta_{01}$ which, after a step, materializes among outputs first with $\beta_{010}$; see Figure \ref{fig:app4}. That $\beta_{010}>\beta_{01}$ is a first observation. That $\mathrm{Tr}(\beta_{01},\beta_0)$ begins as the opposite-signed tail $\mathrm{Tr}(\beta,\beta_0)$ of $\mathrm{Tr}(\beta,\gamma)$ on the left-hand side of Figure \ref{fig:app4} does, departing from it only upon reaching a $\beta'$ with $C_{\beta'}\cap [\beta_{01},\beta)\neq\varnothing$ is a second one.
That we have described a basic building block of $\mathrm{Tr}_2(+,\alpha,\beta,\gamma)$ is a third observation, in the sense that within $\mathrm{Tr}_2(+,\alpha,\beta,\gamma)$, steps of internal walks give rise to external walks between their endpoints, which induce internal walks on the clubs thereby arising, and so on: Figures \ref{fig:app3} and \ref{fig:app4} embody a view of $\mathrm{Tr}_2(+,\alpha,\beta,\gamma)$ as consisting of nothing other than this rhythm.
\item Those elements of clubs falling below $\alpha$ in Figures \ref{fig:app3} and \ref{fig:app4} represent the material, or potential material, of $\mathrm{L}_2(\alpha,\beta,\gamma)$.
\end{enumerate}
What we wish above all to highlight here, though, is the issue of repetitions, and of canceling terms, walks, and even cones, within a given $\mathrm{Tr}_n(+,\alpha,\beta,\gamma)$.
This possibility was noted in Subsection \ref{subsubsect:pairings} and again in item 3 above, but it is a much more pervasive phenomenon than either might have led us to expect. Consider, for example, the $+$-signed walk from $\beta_{00001}$ down to $\beta_{000011}$ near the middle of Figure \ref{fig:app4}.
Following the horizontal guidelines to the right, we find its $-$-signed copy in the walk from $\beta_{01000}$ down to $\beta_{010001}$.
There is a reason for this: both of the earlier external walks corresponding to $\mathrm{Tr}(\beta_{01},\beta_0)$ and $\mathrm{Tr}(\beta,\beta_0)$ and discussed in item 3 above passed through $\beta_{000}=\beta_{0100}$, thereby initiating $C_{\beta_{000}}=C_{\beta_{0100}}$-internal walks down to $\alpha$ which eventually coincide, and the repetition in question is an artifact of this coincidence.

Note in contrast that the walk from $\beta_{0000}$ down to $\beta_{00001}$ has no opposite-signed counterpart, essentially for the reason that $\mathrm{max}(C_{\beta_{000}}\cap\beta)\geq\mathrm{max}(C_{\beta_\varnothing}\cap\beta)$; by this same observation, the line (dashed in Figures \ref{fig:app3} and \ref{fig:app4}) separating the canceling and non-canceling portions of $\mathrm{Tr}_2(+,\alpha,\beta,\gamma)$ is highly sensitive to the movements of $\mathrm{L}(\beta,\gamma)$.

The most immediate point, though, is this: the prospect of cancellations renders it significantly harder to reason about values like $\rho^n_2(\vec{\alpha})$ from initial portions of $\mathrm{Tr}_n(+,\vec{\alpha})$ when $n>1$ than when (as classically) $n=1$. In particular, it's much trickier to bound higher $\rho^n_2$ from below, for example, and this renders nontriviality arguments that much more elusive.
Although it's in consequence tempting and even probably at times productive to disregard these cancellations, they are so fundamental, as we have seen, to the phenomenon of $n$-coherence that there seems ultimately to be no alternative to simply understanding them better.
\subsection{Depth}
\label{subsect:depth}
Figure \ref{fig:app4} suggests a second way of regarding higher walks as trees, since $\mathrm{Tr}_2(+,\alpha,\beta,\gamma)$ is, therein, visibly a tree of classical walks.
The root of this tree is $\mathrm{Tr}(\beta,\gamma)$; its next level consists of the walks between the relevant points of $\bigcup_{\xi\in\mathrm{Tr}(\beta,\gamma)\backslash\{\beta\}}C_\xi$, and so on, and these zeroth, first, and second levels of the tree which we will denote by $T_\mathrm{w}(\alpha,\beta,\gamma)$ are colored blue, cyan, and green, respectively, within Figure \ref{fig:app4}.
It is the height of $T_\mathrm{w}(\alpha,\beta,\gamma)$ that we mean when we speak of \emph{depth}.

As noted, this structure is largely encoded by the elements $\sigma$ of $\rho^2_{2,\mathrm{t}}(\alpha,\beta,\gamma)$: if the string $01$ occurs $n$ times within $\sigma$, then $\beta_\sigma$ resides within a node of height at most $n$ within $T_\mathrm{w}(\alpha,\beta,\gamma)$.
This coding is imperfect, though, for the reason that classical walks can give rise to further classical walks without the mediation of properly internal steps, as in the sequence $\mathrm{Tr}(\beta,\gamma)\to\mathrm{Tr}(\beta_{000000},\beta)$ (see Figure \ref{fig:app4}), in the course of which no $1$ appears.
A surer measure of the height in question is given by sign changes:
\begin{definition}
Fix an order-$2$ $C$-sequence $\mathcal{C}$ on an ordinal $\varepsilon$. For any $\alpha\leq\beta\leq\gamma<\varepsilon$, define $s_{\alpha\beta\gamma}:\rho^2_{2,\mathrm{t}}(\alpha,\beta,\gamma)\to\{-1,1\}$ by
$$\sigma\mapsto\text{ the sign of }\mathrm{Tr}_2(+,\alpha,\beta,\gamma)$$
and define $o_{\alpha\beta\gamma}:\rho^2_{2,\mathrm{t}}(\alpha,\beta,\gamma)\to\omega$ by
$$\sigma\mapsto |\{k<|\sigma|\mid s_{\alpha\beta\gamma}(\sigma\restriction k)\neq s_{\alpha\beta\gamma}(\sigma\restriction k+1)\}|$$
and define the \emph{depth} $d(\alpha,\beta,\gamma)$ of $\mathrm{Tr}_2(+,\alpha,\beta,\gamma)$ by
$$d(\alpha,\beta,\gamma)=\mathrm{max}\{o_{\alpha\beta\gamma}(\sigma)\mid\sigma\in\rho^2_{2,\mathrm{t}}(\alpha,\beta,\gamma)\}.$$
\end{definition}
Depth in this formulation is clearly a notion of \emph{oscillation}, and it is a bound, by the above remarks, on the oscillation of $0$s and $1$s within the indices of $\mathrm{Tr}_2(+,\alpha,\beta,\gamma)$ as well.
But it is plainly also a notion of descent, in the sense of item 3 above, and of diameter: it is the maximal length of a ``walk'' through $\mathrm{Tr}_2(+,\alpha,\beta,\gamma)$ whose nodes are its classical walks.

The combinatorics at play in this function are fascinating, even already under such restrictive assumptions as $0=\alpha<\beta<\gamma<\omega_1$, and much more so on the domain $\omega_2$, and no small part of their charm lies in how they refresh our experience of classical walks objects.
For the moment, though, let us restrict our remarks to the following: depth represents one of the most conspicuous non-classical phenomena or ``axes'' arising in the higher trace functions $\mathrm{Tr}_n$, and is almost certainly a key to their further development.
Put differently, it represents their combinatorial dimension most alien to the case of $n=1$, and the dimension, consequently, in which one should expect to derive their most novel effects.
The first concrete questions in these directions concern the \emph{unboundedness} of the function $d$, questions recorded in our conclusion and bearing, in turn, on the questions of nontriviality organizing both Section \ref{sect:nontriviality} and \ref{subsect:strongcolorings} below.
\section{Nontriviality}
\label{sect:nontriviality}

As partially indicated, this section is arguably the domain of the most pressing, interesting, and consequential questions about higher walks.
The issue is not, though, as might once have been suspected, that nontrivial coherence phenomena, or even nontrivial coherence phenomena deriving from higher walks, do not follow from the $\mathsf{ZFC}$ axioms on cardinals $\kappa>\omega_1$.
The issue is rather that, for every uncountable $\kappa<\aleph_\omega$, they do:
the aforementioned questions are those that ensue about the extent and combinatorial mainsprings of these phenomena, and it is as much these questions as any particular theorem which we hope to communicate below.

Put differently, an analysis like that of Section \ref{sect:higher_coherence} would be a sterile exercise if the $n$-coherent families exhibited therein were all \emph{trivially} so --- as, for example, are the $\rho_2^n$-fiber maps with respect to the $(n-1)$-fold compounding of the trivial $C$-sequence $\langle C_\beta:=\beta\mid\beta<\varepsilon\rangle$ on any ordinal $\varepsilon$.
Below, in Section \ref{subsect:nontrivncoh}, we will show that one of them, namely the family of $r_2^n$-fiber maps with respect to an ordertype-minimal $C$-sequence on $\omega_n$, is \emph{not} trivial, in the mod finite sense of the following definition.
\begin{definition}
\label{def:n-triv}
Let $A$ be an abelian group. For any $n>1$, a family of functions $$\Phi=\langle\varphi_{\vec{\alpha}}:\alpha_0\rightarrow A\mid\vec{\alpha}\in [\varepsilon]^{n}\rangle$$ is \emph{trivial} (or \emph{$n$-trivial}, where the emphasis is useful) \emph{mod locally constant} if 
		\begin{align}
			\label{eq:ntriv}\sum_{i=0}^{n-1} (\text{-}1)^i\psi_{\vec{\alpha}^i}= \varphi_{\vec{\alpha}}\hspace{.8 cm} \textnormal{modulo locally constant functions}
		\end{align}
		for all $\vec{\alpha}\in [\varepsilon]^{n}$  (as before, implicit in this equation is the restriction of the function $\psi_{\vec{\alpha}^0}$ to the domain of the other functions in the sum).
Similarly for the \emph{mod finite} modulus of finitely supported functions.	
		For $n=1$ and either modulus, the triviality condition is as described in Section \ref{subsect:classical}.
		\end{definition}
This is, of course, a coboundary condition, the one complementary to the cocycle condition of $n$-coherence; observe in particular that with respect to either of the above moduli, any $n$-trivial family is $n$-coherent.
The general question is whether the converse may, on a given $\varepsilon$, fail to hold, a question equivalent, as we have hinted, to that of whether $\mathrm{H}^n(\varepsilon;\mathcal{A})\neq 0$.
Note that such a failure is an instance of \emph{incompactness}, for the reason that all properly initial segments of a nontrivial $n$-coherent family \emph{are} trivial, as readers new to this material are encouraged to verify.

In a second subsection, we turn our attention to the functions $\rho^n_2$, both for their intrinsic interest and as a way of framing some of our most fundamental questions.
The nontriviality arguments of Section \ref{subsect:nontrivncoh} are somewhat coarse, and seem to require large codomains to apply; if, in contrast, small-codomain functions like $\rho_1$ and $\rho_2$ attain nontrivial coherence on $\omega_1$, it's because deeper principles of nontriviality --- the ``Hausdorff condition'' of finite-to-one-ness and a Ramsey-theoretic unboundedness property, respectively --- are operative within them.
It is on the elaboration of higher-dimensional and higher-cardinal analogues of such principles that the fullest further development of the theory of higher walks seems largely to depend.
\subsection{Nontrivial $n$-coherent families}
\label{subsect:nontrivncoh}
We turn now to the first of two related constructions of nontrivial $n$-coherent families on $\omega_n$.
\begin{theorem}
\label{thm:nontriviality_r_2^n}
For any $n>0$ and ordertype-minimal $C$-sequence $\mathcal{C}$ on $\omega_n$, the family
$$\Phi(r_2^n):=\langle r_2^n(\,\cdot\,,\vec{\gamma}):\gamma_0\to\bigoplus_{\omega_n^{[n+1]}}\mathbb{Z}\mid\vec{\gamma}\in[\omega_n]^n\rangle$$
of $r_2^n$-fiber maps forms a mod finite nontrivial $n$-coherent family of functions.
\end{theorem}
The following function will facilitate the proof:
for any $n>0$ and $\delta<\omega_n$ let
$$\pi^\delta_n:\bigoplus_{[\omega_n]^{n+1}}\mathbb{Z}\to\bigoplus_{[\delta]^n}\mathbb{Z}$$
be the homomorphism taking $\lfloor(\vec{\gamma},\delta)\rfloor$ to $\lfloor\vec{\gamma}\rfloor$ and all other generators to $0$.
\begin{proof}
For the $n$-coherence of such families, see Lemma \ref{lem:r^n_2_coherence}.
We will argue their nontriviality by induction on $n$.
The base case of $n=1$ was established, in essence, in Theorem \ref{thm:ri}.
The key to step $n+1$ of the induction is the following claim.
\begin{claim}
For every $\delta\in S^{n+1}_n$, the family
$$\Phi(r_2^{n+1})^\delta:=\langle\pi^\delta_{n+1}\circ r_2^{n+1}(\,\cdot\,,\vec{\gamma},\delta)\big|_{C_\delta\cap\gamma_0}\mid\vec{\gamma}\in[C_\delta]^{n}\rangle$$
is nontrivially $n$-coherent.
\end{claim}
\begin{proof}[Proof of claim]
Since, by our assumptions on $\mathcal{C}$, this family is simply a $C_\delta$-internal version of $\Phi(r_2^n)$, the claim amounts to our induction hypothesis.
\end{proof}
Now suppose for contradiction that a family $$\Psi=\langle\psi_{\vec{\gamma}}:\gamma_0\to\bigoplus_{[\omega_n]^{n+1}}\mathbb{Z}\mid\vec{\gamma}\in [\omega_{n+1}]^n\rangle$$ trivialized $\Phi(r^{n+1}_2)$.
By our claim, for all $\delta\in S^{n+1}_n$ there exists an $f_1(\delta)=\vec{\gamma}\in [\delta]^{n-1}$ and $f_0(\delta)=\alpha<\gamma_0$ with $\pi_n^\delta\circ\psi_{\vec{\gamma}}(\alpha)\neq 0$.
By the Pressing Down Lemma, $(f_0,f_1)$ is constantly $(\alpha',\vec{\gamma}')$ on some stationary $S\subseteq S_n^{n+1}$ --- but this implies that $\psi_{\vec{\gamma}'}(\alpha')\notin\bigoplus_{[\omega_n]^{n+1}}\mathbb{Z}$, a contradiction, as desired.
\end{proof}
\begin{remark}
More careful argument will show that $\Phi(r_2^n)$ fails even to admit a $\big(\prod_{[\omega_n]^{n+1}}\mathbb{Z}\big)$-valued mod finite trivialization.
Observe also that Theorem \ref{thm:nontriviality_r_2^n} could, like its base case, equally well have been framed and argued for the $r^n_2$ variant taking values in $\bigoplus_{[\omega_n]^{n+1}}\mathbb{Z}/2$; this is clear since its argument so entirely flowed from its base case.
\end{remark}
The principle at work here is even plainer in our second example of a nontrivial $2$-coherent family
$$\Phi=\langle\varphi_{\alpha\beta}:\alpha\to\mathbb{Z}^{(\omega_2)}\mid\alpha<\beta<\omega_2\rangle,$$
the construction of which will reconnect us with the discussion of Section \ref{subsect:alternative}.
Much as for the nontrivial $1$-coherent family described therein, choices of minimal ordertype clubs $C_\gamma\subseteq\gamma$ for each limit $\gamma$ in $\omega_2$ will be instrumental in our construction. Here, though,  choices of minimal ordertype clubs $C_{\beta\gamma}\subseteq\beta\cap C_\gamma$ for each $\beta\in\mathrm{acc}(C_\gamma)$ will be needed as well; in aggregate, in other words, our construction will invoke some choice of order-$2$ ordertype-minimal $C$-sequence $\mathcal{C}$ on $\omega_2$.
Below, the terms \emph{$2$-coherent} and \emph{nontrivial} may be read as referencing either of our main moduli, so long as this is done consistently.

Again much as in Section \ref{subsect:alternative}, we will build up our $\Phi$ in stages $\gamma<\omega_2$; the premise at each of them is that the portion $\Phi\restriction\gamma=\langle\varphi_{\alpha\beta}\mid\alpha<\beta<\gamma\rangle$ has been constructed.
We define a $2$-coherent extension $\Phi\restriction\gamma+1$ of $\Phi\restriction\gamma$ by first defining suitable $\varphi_{\beta\gamma}$ for $\beta\in C_\gamma$, then extending these definitions to all $\beta\in\gamma$; by taking one extra step at stages $\gamma\in\mathrm{Cof}(\omega_1)$ we will ensure that $\Phi$ is nontrivial as well.
The indexing ordinals $\gamma$, after all, satisfy one of three properties:
\begin{enumerate}
\item $\gamma=\beta+1$ for some $\beta$. In this case, let
\[
\varphi_{\alpha\gamma}=
    \begin{cases}
      0 & \text{if } \alpha=\beta,\\
        \varphi_{\alpha\beta} & \text{if } \alpha<\beta.
    \end{cases}
\]
\item $\langle\gamma_i\mid i\in\omega\rangle$ is an increasing enumeration of $C_\gamma$. In this case, let $\gamma_{-1}=0$ and define the functions $\langle\varphi_{\gamma_i\gamma}\mid i\in\omega\rangle$ by induction on $i$ as follows: for any $\xi<\gamma_i$ let
\[
\varphi_{\gamma_i\gamma}(\xi)=
    \begin{cases}
      \varphi_{\gamma_{i-1}\gamma}(\xi)-\varphi_{\gamma_{i-1}\gamma_i}(\xi) & \text{if } \xi<\gamma_{i-1},\\
        0 & \text{if } \gamma_{i-1}\leq\xi<\gamma_i.
    \end{cases}
\]
These assignments $2$-cohere with those of previous stages by construction; similarly for the natural extension of these assignments to all other $\varphi_{\beta\gamma}$ given by 
$$\varphi_{\beta\gamma}(\xi)=\varphi_{\beta\gamma_{i(\beta)}}(\xi)+\varphi_{\gamma_{i(\beta)}\gamma}(\xi).$$
\item $\langle\gamma_i\mid i\in\omega_1\rangle$ is an increasing enumeration of $C_\gamma$.
Exactly as above, definitions for $\langle\varphi_{\beta\gamma}\mid\beta\in\gamma\backslash C_\gamma\rangle$ will canonically derive from definitions for $\langle\varphi_{\beta\gamma}\mid \beta\in C_\gamma\rangle$, hence we may focus our attention on the latter. These, in turn, we define in two steps: first, we define $\langle\varphi'_{\beta\gamma}\mid \beta\in C_\gamma\rangle$ by induction on the index of $\beta=\gamma_j$ to 2-cohere with the previous stages' functions; we then modify these assignments to $\langle\varphi_{\beta\gamma}\mid \beta\in C_\gamma\rangle$ in a manner which, cumulatively, will imply $\Phi$ nontrivial.

Again let $\gamma_{-1}=0$ and for successor $j$ (including $0$) let
\[
\varphi'_{\gamma_j\gamma}(\xi)=
    \begin{cases}
      \varphi'_{\gamma_{j-1}\gamma}(\xi)-\varphi_{\gamma_{j-1}\gamma_j}(\xi) & \text{if } \xi<\gamma_{j-1},\\
        0 & \text{if } \gamma_{j-1}\leq\xi<\gamma_j.
    \end{cases}
\]
For limit $j$, fix increasing enumeration $\langle\beta_i\mid i\in\omega\rangle$ of $C_{\gamma_j\gamma}$ and let
$$\varphi'_{\gamma_j\gamma}(\xi)=\varphi'_{\beta_{i(\xi)}\gamma}(\xi)-\varphi_{\beta_{i(\xi)}\gamma_j}(\xi)$$
for any $\xi<\gamma_j$.
Observe that our procedure will imply that $\mathrm{supp}(\varphi'_{\gamma_j\gamma})\subseteq\gamma$ for all $j\in\omega_1$.

To conclude the construction, fix, for example via the mechanism of Section \ref{subsect:internal_walks} or more directly via the sequence $\langle C_{\beta\gamma}\mid\beta\in C_\gamma\rangle$, a nontrivial coherent family $\Upsilon=\langle\upsilon_\beta:\beta\to\mathbb{Z}\mid\beta\in C_\gamma\rangle$, and let
$$\varphi_{\beta\gamma}(\xi)=\varphi'_{\beta\gamma}(\xi)+\upsilon_\beta(\xi)\cdot\lfloor\gamma\rfloor$$
for each $\xi$ less than $\beta\in C_\gamma$. This, together with the aforementioned extension to all other $\beta\in\gamma$, completes the cofinality-$\omega_1$ step (and hence the construction); observe that $2$-coherence is conserved.
\end{enumerate}
Suppose now for contradiction that
$$\Psi=\langle\psi_\beta:\beta\to\mathbb{Z}^{(\omega_2)}\mid\beta\in\omega_2\rangle$$
trivializes $\Phi$, and let
$$f(\gamma)=\sup_{(\alpha,\beta)\in [\gamma]^2}\mathrm{supp}(\psi_\beta(\alpha))$$
for all $\gamma\in\omega_2$.
Since $f$ is continuous and increasing with unbounded range, $C_f:=\{\gamma\mid f(\gamma)=\gamma\}$ is club in $\omega_2$; for any $\gamma\in S^2_1\cap C_f$, though, this implies that $$\langle\pi^\gamma\circ\varphi_{\beta\gamma}:\beta\to\mathbb{Z}\mid\beta\in\gamma\rangle$$
is trivial, where $\pi^\gamma:\mathbb{Z}^{(\omega_2)}\to\mathbb{Z}$ denotes the projection onto the $\gamma^{\mathrm{th}}$ coordinate. This is the contradiction desired.

The above is about as elementary a recursive $\mathsf{ZFC}$ construction of a nontrivial $2$-coherent family as may be imagined, and its generalization to higher $n$ is essentially straightforward.
We have recorded it to underscore two points.

For the first, let us pause to compute an arbitrary $\varphi_{\beta\delta}(\xi)$ from the construction above. Perhaps $\delta=\gamma+1$, for example, so that $\varphi_{\beta\delta}(\xi)=\varphi_{\beta\gamma}(\xi)$. Perhaps $\beta$ is then a successor element of $C_\gamma$, so that, letting $\beta'=\mathrm{max}\,C_\gamma\cap\beta$, we have that $\varphi_{\beta\delta}(\xi)=\varphi_{\beta'\gamma}(\xi)-\varphi_{\beta'\beta}(\xi)$. Or perhaps $\beta\notin C_\gamma$ instead, so that if $\gamma'=\mathrm{min}\,C_\gamma\backslash\beta$ then $\varphi_{\beta\delta}(\xi)=\varphi_{\gamma'\gamma}(\xi)+\varphi_{\beta\gamma'}(\xi)$. And so on: as the reader will already have perceived, the working out of $\varphi_{\beta\delta}(\xi)$ follows exactly the pattern of that of $\mathrm{Tr}_n(+,\xi,\beta,\delta)$, even at the level of signs; we thus arrive as in the $n=1$ case to the perspective that \emph{higher walks are how the most fundamental recursive constructions of higher nontrivial coherent families unwind.}\footnote{Of course, minor and less uniform variations on the above construction exist; their unwindings give rise to minor but possibly interesting variations on our proposed higher walks' forms.}

The second point is that the two nontriviality arguments above (which are effectively the same argument) each need a size-$\aleph_n$ codomain to work.
This is as much an empirical statement as a mathematical one: no one so far has succeeded \emph{in \textsf{ZFC} alone} in replicating these effects with smaller codomains; $\mathbb{Z}$-valued nontrivial $n$-coherent families do, on the other hand, exist in every $\mathsf{ZFC}$ model so far assayed for them, a point we return to in our conclusion. What this observation most immediately implies for our purposes is that a smaller-codomain $n$-coherent family may only be nontrivial by way of some other principle. Let us survey, briefly, some possibilities for $\rho_2^n$.
\subsection{Prospects for $\rho^n_2$}
\label{subsect:rho^n_2}
For simplicity (and since our discussion will be so essentially heuristic in nature), we will focus on the representative case of $\rho^2_2$.
As a sanity check, let us begin with the observation that $\rho_2^2$ is unbounded in both the positive and negative directions; these points are sufficiently simple and instructive that we record them as an exercise.
\begin{exercise}
\label{ex:rho2unbd}
Show that with respect to any ordertype-minimal $C$-sequence $\mathcal{C}$ on $\varepsilon\geq\omega$,
$$\rho^2_2(\alpha,\beta,\gamma)=1-\rho_2(\beta,\gamma)$$
for all $\alpha\leq\beta\leq\gamma<\omega$. Note also that for any infinite cardinal $\kappa<\varepsilon$, there is a canonical choice for $C_\kappa$, namely $\kappa$ itself, and, in turn, for each $C_{\beta\kappa}$, namely $C_\beta$. Show that under such choices,
$$\rho^2_2(\alpha,\beta,\kappa)=\rho_2(\alpha,\beta)$$
for all $\alpha\leq\beta<\kappa$, and conclude that under these assumptions, the range of $\rho^2_2$ on the domain $(\omega+1)^{[3]}$ is precisely $\mathbb{Z}$.\footnote{Observe further that $\rho_2$'s unboundedness carries implications, as here and in Theorem \ref{thm:rho2} below, for the unboundedness of $\rho^2_{2,\mathrm{t}}$, in the sense that the unboundedness of the postcomposition of the latter by any enumeration (i.e., map to $\omega$) of the set of finite binary trees will generally follow.}
\end{exercise}
In particular, under the above assumptions on $\kappa=\omega_1$,
\begin{align}
\label{eq:rho22_on_omega_1}
\langle\rho^2_2(\,\cdot\,,\beta,\omega_1):\beta\to\mathbb{Z}\mid\beta\in\omega_1\rangle=\langle -\rho_2(\,\cdot\,,\beta):\beta\to\mathbb{Z}\mid\beta\in\omega_1\rangle
\end{align}
is a nontrivial coherent family (mod locally constant); yet more particularly, by the $2$-coherence of $\rho^2_2$, the function $\rho^2_2(\,\cdot\,,\alpha,\beta)$
is locally constant for any countable ordinals $\alpha\leq\beta$.
This property will not, in general, persist for larger $\alpha\leq\beta$, but the above observations do suggest a useful view of $\rho^2_2$ under the assumption that $\mathcal{C}$ is an ordertype-minimal $C$-sequence on $\omega_2$, a view holding even more forcefully if $\mathcal{C}$ is also a compounding of a classical $C$-sequence.
Namely, for each $\gamma\in S^2_1$,
\begin{align*}
\langle\rho^2_2(\,\cdot\,,\beta,\gamma)\big|_{\beta\cap C_\gamma}\mid\beta\in C_\gamma\rangle
\end{align*}
is a \emph{distorted image} of a nontrivial coherent family like (\ref{eq:rho22_on_omega_1}), with the distortions stemming from the inputs $(-,\xi,\beta_{\sigma^\frown 1},\beta_\sigma)$ $(\sigma\in\rho^2_{2,\mathrm{t}}(\xi,\beta,\gamma)\cap 1^{<\omega})$ and $(-,\xi,\beta_{\varnothing},\beta)$, which in the $\gamma=\omega_1$ case would have been terminal, since in that case $\mathrm{min}\,C_{\beta_\sigma\gamma}\backslash\xi=\mathrm{min}\,C_{\beta_\sigma}\backslash\xi=\beta_{\sigma^\frown 1}$ and $\mathrm{min}\,C_{\beta\gamma}\backslash\xi=\mathrm{min}\,C_{\beta}\backslash\xi=\beta_{\varepsilon}$. More colloquially, what these distortions reflect, and what $\rho^2_2$ may be regarded as at some level recording, is the differences between the external and $C_\gamma$-internal steps from such $\beta_\sigma$ down to $\xi$, and the potential magnitude of such differences may be thought of as growing larger as $\gamma$ moves farther, in $S^2_1$, away from $\omega_1$. The much more to be said in these directions would move us beyond the terrain of an introduction; suffice it to say that the non-$1$-triviality of $\rho_2[\gamma]$ and its proximity to $\rho^2_2$ along $C_\gamma$ for each $\gamma\in S^2_1$ are plausibly significant factors in the non-$2$-triviality of $\rho^2_2$.

Here, then, is a good place to recall the classical argument for the nontriviality, mod locally constant, of (\ref{eq:rho22_on_omega_1}).
In fact the original argument is of a stronger statement, namely that no $\psi:\omega_1\to\mathbb{Z}$ differs from each $\rho_2(\,\cdot\,,\beta)$ $(\beta\in\omega_1)$ by a \emph{bounded} function: one supposes for contradiction that there exist such a $\psi$ and a $k:\omega_1\to\mathbb{N}$ such that
\begin{align}
|\psi(\alpha)-\rho_2(\alpha,\beta)|\leq k(\beta)
\end{align}
for all $\alpha<\beta<\omega_1$. By the pigeonhole principle, there then exist uncountable $A,B\subseteq\omega_1$ such that $\psi$ and $k$ are constant on $A$ and $B$ respectively. By the following theorem (a weak instance of \cite[Lem.\ 2.3.4]{todwalks}), $\rho_2$ is strongly unbounded on any such $A\times B$, and this supplies the desired contradiction.
\begin{theorem}
\label{thm:rho2}
For any $\{(\alpha_i,\beta_i)\,|\,i\in\omega_1\}\subset [\omega_1]^2$ satisfying $\beta_i<\alpha_j$ for all $i<j$ in $\omega_1$ and $\ell\in\mathbb{N}$ there exists a cofinal $E\subseteq\omega_1$ such that $\rho_2(\beta_i,\gamma_j)>\ell$ for any $i<j$ in $E$.
\end{theorem}

What would such an argument look like for $\rho_2^2$? One would suppose for contradiction that there exist a $\Psi=\{\psi_{\alpha}:\alpha\to \mathbb{Z}\mid \alpha\in \omega_2\}$ and $k:[\omega_2]^2\to\mathbb{N}$ such that 
\begin{align*}
|\psi_\gamma(\alpha)-\psi_\beta(\alpha)-\rho^2_2(\alpha,\beta,\gamma)|\leq k(\beta,\gamma)
\end{align*}
for all $\alpha<\beta<\gamma<\omega_2$. One would again then hope to
\begin{enumerate}
\item bound the function $(\alpha,\beta,\gamma)\mapsto |\psi_\gamma(\alpha)-\psi_\beta(\alpha)|+k(\beta,\gamma)$ on some large set $X$;
\item show that $\rho_2^2$ exhibits strongly unbounded behavior on any such $X$.
\end{enumerate}
Hopes for something along the lines of item 2, perhaps a $\rho^2_2$-analogue of Theorem \ref{thm:rho2} for families $\{(\alpha_i,\beta_i,\gamma_i)\mid i\in\omega_1\}$, do not seem unreasonable;\footnote{See \cite[Def.\ 1.2]{Knaster} for the general $2$-dimensional template for this sort of behavior, as well as \cite[\S 9]{todwalks}.} that the theorem's classical argument applies the $n=1$ instances of Lemmas \ref{lem:max_L<limits} and \ref{lem:end-extension}, for example, is promising.
Interestingly, in such a scenario, $\rho^2_2$ would already exhibit strong unboundedness on the domain $[\omega_1]^3$; we should underscore, though, that $\langle\rho^2_2(\,\cdot\,,\beta,\gamma)\mid (\beta,\gamma)\in [\delta]^2\rangle$ \emph{is} trivial (as any $2$-coherent family is) for any $\delta<\omega_2$, so that the distinction of $\delta=\omega_2$ in this case would lie entirely in how it induces item 1.
That item replaces the pigeonhole principle of the $n=1$ case with, loosely speaking, a Ramsey-type principle; the latter being harder to come by in $\mathsf{ZFC}$, it is natural, at least at a first pass, to pursue these questions under any additional additional assumptions that seem useful: the nontrivial coherence of $\rho^2_2$ under any of them would already be quite interesting.
Let us note that strong variants of Theorem \ref{thm:rho2} do extend to higher cardinals $\kappa$ under mild nontriviality conditions on the underlying $C$-sequence (see \cite[Thm.\ 6.3.2]{todwalks}); the utility of $\square(\kappa)$ in this context is in maintaining the coherence of $\rho_2$ alongside them.
Nontriviality conditions on higher $C$-sequences carrying similar implications for $\rho_2^n$, perhaps to be paired, in higher $\square$-like combinations, with the coherence conditions of Theorem \ref{thm:n-coherence}'s item 3, would of course be of enormous interest as well.

However, since this all remains, at present, speculative, one might reasonably worry that $\rho^2_2$ loses too much of the data of $\mathrm{Tr}_2$ to retain its nontriviality (which we may think of Theorem \ref{thm:nontriviality_r_2^n} as registering).
Similarly for any of the other small-codomain higher rho functions (to which the above discussion should be read as also broadly applying); here the maximal worry is that none of those \emph{of countable codomain} is nontrivial on $\omega_2$ under suitable $\mathsf{ZFC}$ assumptions.
This, too, is a scenario.
As noted in our conclusion, however, it seems at least as peculiar and combinatorially intricate a scenario as any other which we have entertained above.
Nor would it absolve us --- given both the known extent of the condition $\mathrm{H}^2(\omega_2;\mathbb{Z})\neq 0$ and the combinatorial richness of higher walks --- of what are the driving questions in any of these scenarios, namely the interrelationships between these phenomena and the existence of higher-dimensional combinatorial principles distinctive to $\omega_2$, or to $\omega_n$ more generally.
We touch briefly on related approaches to the latter in the following section.
\section{Higher-dimensional linear orders and colorings}
\label{sect:further}
In this section, we consider higher variants of two main areas of the classical theory only partially subsumed by the theme of nontrivial coherence: \emph{strong colorings}, and \emph{uncountable linear orders}.
\subsection{Strong colorings}
\label{subsect:strongcolorings}
The subject of strong colorings is too large to either entirely neglect or seriously explore here; we therefore confine our remarks to the portion of it connecting to themes already raised.
Most conspicuous among these is \emph{nontriviality}: intuitively, both strong colorings and nontrivial families of functions resist simplification, and their means of doing so can often overlap.
Among the first and most striking applications of the walks machinery, for example, was the provision of witnesses to the negative partition relation
\begin{align}
\label{eq:rainbow1}
\omega_1\not\to[\omega_1]_\omega^2,
\end{align}
to be read as \emph{there exists a function $f:[\omega_1]^2\to\omega$ such that for any uncountable $X\subset\omega_1$, the $f$-image of $[X]^2$ is all of $\omega$}.
Note that $\rho_2$, by Theorem \ref{thm:rho2}, already provides us with a function whose restriction to any such $[X]^2$ is \emph{unbounded} in $\omega$; the stronger property (\ref{eq:rainbow1}) is witnessed, on the other hand, by functions under no obligation to \emph{cohere}.
These witnesses more precisely derive from further functions which, without quite being rho functions, themselves derive from classical walks. These are the \emph{oscillation} function $\mathrm{osc}(\,\cdot\,,\,\cdot\,)$ and \emph{square bracket operation} $[\,\cdot\,\cdot\,]$ of \cite[\S 2.1 and \S 8]{todwalks} and \cite[\S 5 and \S 8]{todwalks}, respectively, and it is natural to ponder their higher-dimensional analogues as well. The former is probably more accurately described as a family of functions (see also \cite{LSpace}), and on this front we content ourselves, for now, with observing both that \emph{depth} is, as noted, an oscillation function and that the higher walks landscape is, by way of signs, branchings, repetitions, and so on, even richer in oscillatory phenomena than the classical one, and that these very much do call for further study.

The latter function points us to some subtler issues.
Recall first of all that $[\,\cdot\,\cdot\,]$ \emph{does} extend in \textsf{ZFC} to higher cardinals $\kappa=\omega_n$ (and beyond; see \cite[Cor.\ 8.2.14]{todwalks}), with the consequence that $\kappa\not\to [\kappa]^2_\kappa$. Observe more concretely that extending the square bracket's most distinctive feature --- \emph{for any $X\in [\kappa]^\kappa$, the $[\,\cdot\,\cdot\,]$-image of $[X]^2$ contains a club subset of $\kappa$} --- to dimensions $n>2$ is a trivial exercise: just continue to apply $[\,\cdot\,\cdot\,]$ to the first two coordinates of any $n$-tuple. What remains intriguing as $\kappa$ rises above $\omega_1$, though, is the behavior of colorings on smaller subsets of $\kappa$.
Most notably, Todorcevic applied combinations of the above functions in \cite{cubes} to show that
\begin{align}
\label{eq:rainbow2}
\omega_2\not\to[\omega_1]_\omega^3
\end{align}
(see also \cite{FeldmanRinot} for recent refinements and applications of this relation). This is an optimal $\mathsf{ZFC}$ result in the sense, for example, that the superscript $3$ cannot be reduced within it, since under the continuum hypothesis $\omega_2\to(\omega_1)^2_\omega$, by Erd\"{o}s-Rado \cite{ErdosRado}.
Nor can the number of colors be raised within the $\mathsf{ZFC}$ framework since, as Todorcevic has shown, $\omega_2\not\to [\omega_1]^3_{\omega_1}$ is equivalent to the negation of Chang's Conjecture.
The status of the latter, in fact, bears so intimately on combinatorics at this level as to form an unavoidable consideration in their analysis (see \cite[\S 9]{todwalks}).
Nevertheless, for a program of study of higher-dimensional $\mathrm{ZFC}$ combinatorics on the cardinals $\omega_n>\omega_1$, (\ref{eq:rainbow2}) is a reassuring and even inspiring result.
More particularly, its parameters so nicely align with those of higher walks as to underscore the question of their relation. Do higher walks furnish more uniform witnesses to (\ref{eq:rainbow2}) or its higher analogues $\omega_n\not\to[\omega_1]^{n+1}_\omega$? We record this question in our conclusion.
\subsection{Countryman objects and higher-dimensional linear orders}
Central to the classical theory of walks, as we have seen, are the \emph{Countryman lines} deriving from the natural branch-orderings of the trees $T(\rho_i)$.
The question of their analogues for higher walks faces us most immediately with the problem of what a \emph{higher-dimensional linear order} (or, less oxymoronically, \emph{higher-dimensional total order}) might be.\footnote{For Cantor's own studies of ``$n$-dimensional order types'' see \cite[pp.\ 157-158]{Dauben}.}
Facile answers like \emph{a plane} or \emph{a product of linear orders} aren't much help; in the analogy we're pursuing, an $n$-dimensional such structure
\begin{enumerate}
\item should derive from relations among the fiber maps $\rho^n_i(\,\cdot\,,\vec{\gamma})$ $(\vec{\gamma}\in [\omega_n]^n)$ for rho functions $\rho_i^n$, and
\item should admit some finite product which decomposes into less than $\omega_n$ such structures.
\end{enumerate}
Here we describe a framework which rather neatly fulfills item 1 and warrants further study in its own right; item 2 then ranks high on a shortlist of most immediate and intriguing next questions. More briefly, we will show that higher walks \emph{do} induce what we propose to call \emph{higher-dimensional linear orders}, and leave the questions of their classification and nature for later.

What, after all, is a \emph{linear order}? It is a directed complete graph (i.e., a complete graph whose edges are all arrows, or, more succinctly, a \emph{tournament}) whose restriction to any 3 vertices never takes the following form:

\begin{figure}[h]
\centering
\begin{tikzpicture}
[MyPersp]
	\coordinate (A) at (4,0,1);
	\coordinate (B) at (5,0,2.73);
	\coordinate (C) at (6,0,1);
	\draw[thick, -Latex] (A)--(B);
	\draw[thick, -Latex] (B)--(C);
	\draw[thick, -Latex] (C)--(A);
\end{tikzpicture}
\caption{A cycle.}
\label{fig:cyc3}
\end{figure}

Alternatively, a linear order is a directed complete graph whose restriction to any 3 vertices takes the following form:

\begin{figure}[h]
\centering
\begin{tikzpicture}
[MyPersp]
	\coordinate (A) at (4,0,1);
	\coordinate (B) at (5,0,2.73);
	\coordinate (C) at (6,0,1);
	\draw[thick, -Latex] (A)--(B);
	\draw[thick, -Latex] (B)--(C);
	\draw[thick, -Latex] (A)--(C);
\end{tikzpicture}
\caption{Transitivity.}
\label{fig:noncyc3}
\end{figure}
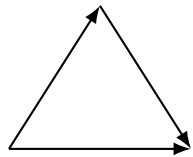

Of course, we identify such forms with their images under reflections and hence rotations: up to isomorphism of relational structures, these are the only such forms that can arise. Observe that the form of Figure \ref{fig:noncyc3} may be characterized as that induced by a numbering of the vertices, unlike that of Figure \ref{fig:cyc3}, while the property that \emph{opposite ends of arrows (heads and tails) meet at each vertex} characterizes Figure \ref{fig:cyc3}.

This readily generalizes: our proposed definition of a \emph{2-dimensional linear order} is \emph{an oriented complete 3-uniform hypergraph whose restriction to any 4 vertices is never cyclic}. Here we'll record orientatations of faces (i.e.,\ ``edges'' of $3$-uniform hypergraphs) as clockwise or counterclockwise arrows. We'll view a complete 3-uniform hypergraph on 4 vertices as from above the apex of a tetrahedron, and record the orientation of the base as an outer loop (to avoid clutter). The cycle in question then admits representation as in Figure \ref{fig:cyc4}.

\begin{figure}[h]
\centering
\begin{tikzpicture}
[MyPersp]
	\coordinate (A) at (3,0,2);
	\coordinate (B) at (7,0,2);
	\coordinate (C) at (5,0,5.46);
	\coordinate (D) at (5,0,3.15);
	\coordinate (E) at (4.85,0,5.69);
	\coordinate (F) at (5.15,0,5.69);
	\coordinate (G) at (2.8,0,1.88);
	\coordinate (H) at (7.2,0,1.88);
	\draw[thick] (A)--(B);
	\draw[thick] (B)--(C);
	\draw[thick] (A)--(C);
	\draw[thick] (A)--(D);
	\draw[thick] (B)--(D);
	\draw[thick] (C)--(D);
	\draw (5,0,2.5) node[scale=2.2] {$\circlearrowright$};
	\draw (4.4,0,3.55) node[scale=2.2] {$\circlearrowright$};
	\draw (5.6,0,3.55) node[scale=2.2] {$\circlearrowright$};
	\draw[thin, ->] (E) to[out=-175,in=120, distance=1.1cm] (G) to[out=-60,in=-120, distance=1.1cm] (H) to[out=60,in=-5, distance=1.1cm] (F);
\end{tikzpicture}
\caption{A degree-2 cycle: $\textbf{H}_4$, in the terminology of \cite{Cherlin_2021}.}
\label{fig:cyc4}
\end{figure}
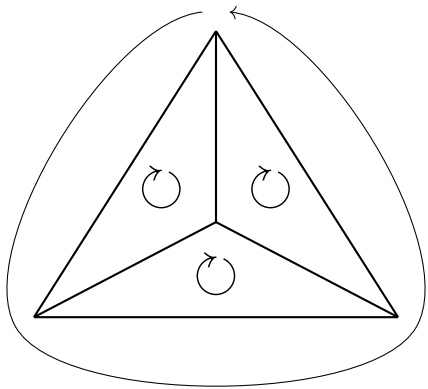

Observe in this figure that any two faces meet in an edge and that their orientations along that edge are opposite, in analogy with the characterization of Figure \ref{fig:cyc3} above.
Figures \ref{fig:cyc3} and \ref{fig:cyc4} are also exactly the \emph{homogeneous} oriented complete 2- and 3-uniform hypergraphs on 3 and 4 vertices, respectively, meaning that any isomorphism between their finite substructures (precise definitions are given below) extends to an automorphism of the whole.
Figure \ref{fig:cyc4} is also the only orientation of the complete 3-uniform hypergraph on 4 vertices which is not induced by an orientation of the tetrahedron's 6 edges. (These orientations induce a clockwise or counterclockwise orientation of a face simply by ``majority rule''; that induced in Figure \ref{fig:noncyc3}, for example, is clockwise.) Hence much as in the case of linear orders, we could equivalently define a \emph{2-dimensional linear order} as \emph{an oriented complete 3-uniform hypergraph whose restriction to any 4 vertices is induced by a orientation of its edges}.  For brevity, we shorten the term to \emph{2-order} below.

These, as it happens, form a fascinating subject; the preceding discussion is essentially that of the 2019 note \cite{HighWalksNote}, which capped several further observations with the remark: \emph{I am unaware of any source that even begins to study these objects [\dots] and find it deeply surprising that there may in fact be none.} I was accordingly delighted and relieved by the arXiv appearance thereafter of Cherlin-Hubi\v{c}ka-Kone\v{c}n\'{y}-Ne\v{s}et\v{r}il's ``Ramsey expansions of $3$-hypertournaments'' \cite{Cherlin_2021}, attesting their roughly contemporaneous study of, most centrally, what we've termed $2$-orders just above.
Those notes approach the subject from the perspective of structural Ramsey theory and are recommended to the reader; to quote from their most relevant passages:
\begin{quote}
In this abstract, an \emph{$n$-hypertournament} is a structure \textbf{A} in a language with a single $n$-ary relation $R$ such that for every set $S\subseteq A$ with $|S|=n$ it holds that the automorphism group of the substructure induced on $S$ by \textbf{A} is precisely $\mathrm{Alt}(S)$, the alternating group on $S$. This in particular means that exactly half of $n$-tuples of elements of $S$ with no repeated occurrences are in $R^{\textbf{A}}$. For $n=2$ we get standard tournaments, for $n=3$ this corresponds to picking one of the two possible cyclic orientations on every triple of vertices.
\end{quote}
In particular, \emph{3-hypertournament} is a handier name for what we've been calling an \emph{oriented complete 3-uniform hypergraph}, and we will employ it below. There are, up to isomorphism, exactly three $3$-hypertournaments on $4$ vertices, namely the cycle of Figure \ref{fig:cyc4}, which Cherlin et al.\ term $\textbf{H}_4$, and the two others of  Figure \ref{fig:ncyc4}, which they term $\textbf{C}_4$ and $\textbf{O}_4$, respectively.
We note in passing that the class of linear orders has a natural image within the class of $3$-hypertournaments whose restrictions to four vertices are all of type $\textbf{C}_4$, for the reason that any ordering of four vertices induces $\textbf{C}_4$ on the $3$-uniform hypergraph which they span.
Note also that, just as the edges between the pairs among any three points in general position within an oriented line will inherit orientations as in Figure \ref{fig:noncyc3} (and never as in Figure \ref{fig:cyc3}), the faces spanned by the triples among any four points in general position within an oriented plane will inherit orientations as in one of the $3$-hypertournaments in Figure \ref{fig:ncyc4} (and never as in Figure \ref{fig:cyc4}).
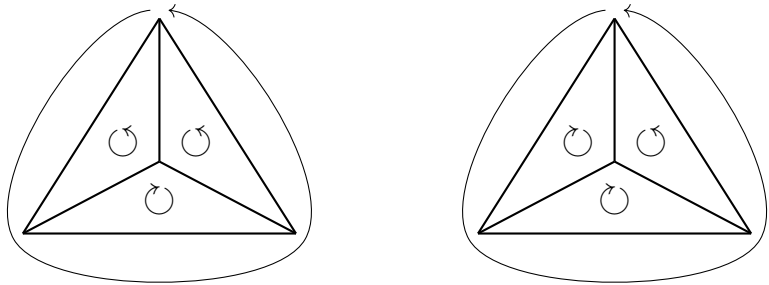
\begin{figure}[h]
\centering
\begin{tikzpicture}
[MyPersp]
	\coordinate (A) at (1.5,0,2);
	\coordinate (B) at (4.5,0,2);
	\coordinate (C) at (6.5,0,2);
	\coordinate (D) at (9.5,0,2);
	\coordinate (E) at (3,0,4.6);
	\coordinate (F) at (8,0,4.6);
	\coordinate (G) at (3,0,2.87);
	\coordinate (H) at (8,0,2.87);
	\coordinate (E0) at (2.9,0,4.7);
	\coordinate (A') at (1.4,0,1.93);
	\coordinate (B') at (4.6,0,1.93);
	\coordinate (E1) at (3.1,0,4.7);
	\coordinate (F0) at (7.9,0,4.7);
	\coordinate (C') at (6.4,0,1.93);
	\coordinate (D') at (9.6,0,1.93);
	\coordinate (F1) at (8.1,0,4.7);

	\draw[thick] (A)--(B);
	\draw[thick] (C)--(D);
	\draw[thick] (A)--(E);
	\draw[thick] (E)--(B);
	\draw[thick] (C)--(F);
	\draw[thick] (F)--(D);
	\draw[thick] (A)--(G);
	\draw[thick] (B)--(G);
	\draw[thick] (E)--(G);
	\draw[thick] (C)--(H);
	\draw[thick] (D)--(H);
	\draw[thick] (F)--(H);
	\draw (3,0,2.4) node[scale=1.6] {$\circlearrowright$};
	\draw (8,0,2.4) node[scale=1.6] {$\circlearrowright$};
	\draw (2.6,0,3.1) node[scale=1.6] {$\circlearrowleft$};
	\draw (3.4,0,3.1) node[scale=1.6] {$\circlearrowleft$};
	\draw (7.6,0,3.1) node[scale=1.6] {$\circlearrowright$};
	\draw (8.4,0,3.1) node[scale=1.6] {$\circlearrowleft$};
	\draw[thin, ->] (E0) to[out=-175,in=120, distance=.8cm] (A') to[out=-60,in=-120, distance=.8cm] (B') to[out=60,in=-5, distance=.8cm] (E1);
	\draw[thin, ->] (F0) to[out=-175,in=120, distance=.8cm] (C') to[out=-60,in=-120, distance=.8cm] (D') to[out=60,in=-5, distance=.8cm] (F1);
\end{tikzpicture}
\caption{The non-cyclic 3-hypertournaments on 4 vertices, dubbed $\textbf{C}_4$ and $\textbf{O}_4$, respectively, in \cite{Cherlin_2021}.}
\label{fig:ncyc4}
\end{figure}

The animating observation for \cite{Cherlin_2021} is that well-understood Ramsey expansions exist for all but one of the four $4$-constrained classes of finite $3$-hypertournaments forming strong amalgamation classes, but that the class for which they do not seems to pose novel challenges for structural Ramsey theory. This class, whose more ``spatial'' aspect we briefly suggested above, is precisely that of the $2$-orders, or of what Cherlin et al.\ term the \emph{$\mathbf{H}_4$-free $3$-hypertournaments}:
\begin{quote}
Note that in some sense, this generalizes the class of finite linear orders: As $\mathrm{Aut}(\textbf{H}_4)=\mathrm{Alt}(4)$, one can define $\textbf{H}_n$ to be the $(n-1)$-hypertournament on $n$ points such that $\mathrm{Aut}(\textbf{H}_n)=\mathrm{Alt}(n)$.
For $n=3$, we get that $\textbf{H}_3$ is the oriented cycle on $3$ vertices and the class of all finite linear orders contains precisely those tournaments which omit $\textbf{H}_3$.
\end{quote}
It is this that we will take as our definition of a higher-dimensional linear order: for any $n>0$, an \emph{$n$-order} is simply an $\textbf{H}_{n+2}$-free $(n+1)$-hypertournament.

Consider now the binary relation $\triangleleft_i$ on $\omega_1$ induced by a function $\rho_i:[\omega_1]^2\to\mathbb{Z}$ as follows: for any $\beta,\gamma\in\omega_1$, if
$$\Delta_i(\beta,\gamma):=\min\{\xi<\beta\cap\gamma\mid \rho_i(\xi,\gamma)-\rho_i(\xi,\beta)\neq 0\}$$
is undefined then $\beta\triangleleft_i\gamma$ if and only if $\beta<\gamma$; otherwise $\beta\triangleleft_i\gamma$ if and only if
$$\rho_i(\Delta_i(\beta,\gamma),\gamma)-\rho_i(\Delta_i(\beta,\gamma),\beta)>0.$$
That this defines a linear order follows from the fact that
\begin{align}
\label{eq:nocycle}
& \rho_i(\Delta_i(\beta,\gamma),\gamma)-\rho_i(\Delta_i(\beta,\gamma),\beta) >0 \\
\nonumber \text{and } & \rho_i(\Delta_i(\alpha,\gamma),\alpha)-\rho_i(\Delta_i(\alpha,\gamma),\gamma)  >0 \\
\nonumber \text{together imply that } & \rho_i(\Delta_i(\alpha,\beta),\alpha)-\rho_i(\Delta_i(\alpha,\beta),\beta)  >0
\end{align}
for any $\alpha,\beta,\gamma\in\omega_1$.
As we have seen, under suitable $C$-sequence assumptions, this is, for each $i\in 4$, a very special linear order indeed.

Consider next, for any $n>0$ and ordinal $\varepsilon$, the $(n+1)$-ary relation $\triangleleft_i^n$ on $\varepsilon$ induced by a function $\rho^n_i:[\varepsilon]^{n+1}\to\mathbb{Z}$ (our reasoning will apply wherever the codomain is a linearly ordered abelian group), which is defined as follows. First, ``symmetrize'' $\rho^n_i$ to a function $\hat{\rho}^n_i$ for which $$\hat{\rho}^n_i(\xi,\sigma\cdot\vec{\beta})=\mathrm{sgn}(\sigma)\cdot\rho^n_i(\xi,\vec{\beta})$$
for any $(\xi,\vec{\beta})\in [\varepsilon]^{n+1}$ and $\sigma\in\mathrm{Sym}(n)$.
Next, for any $\vec{\gamma}\in [\varepsilon]^{n+1}$, if
\begin{align}
\label{eq:Delta^n_i}\Delta^n_i(\vec{\gamma}):=\min\{\xi<\gamma_0\mid \sum_{j=0}^n(-1)^j\rho^n_i(\xi,\vec{\gamma}^j)\neq 0\}\end{align}
is undefined then declare $\triangleleft^n_i(\sigma\cdot\vec{\gamma})$ to hold for all $\sigma\in\mathrm{Alt}(n+1)$; otherwise $\triangleleft^n_i(\sigma\cdot\vec{\gamma})$ holds if and only if
$$\sum_{j=0}^n(-1)^j\hat{\rho}^n_i(\Delta^n_i(\vec{\gamma}),(\sigma\cdot\vec{\gamma})^j)>0.$$
We claim that this relation is $\mathrm{Alt}(n+1)$-invariant, and thus defines an $(n+1)$-hypertournament on $\varepsilon$.
The claim reduces to the following lemma, which ensures also that wherever $\Delta^n_i(\vec{\gamma})$ is defined, it equals $\Delta^n_i(\sigma\cdot\vec{\gamma})$ under the natural extension of equation \ref{eq:Delta^n_i}.
\begin{lemma}
\label{lem:sigmaout}
For all $(\xi,\vec{\gamma})\in [\varepsilon]^{n+2}$ and $\sigma\in\mathrm{Sym}(n+1)$,
$$\sum_{j=0}^n(-1)^j\hat{\rho}^n_i(\xi,(\sigma\cdot\vec{\gamma})^j)=\mathrm{sgn}(\sigma)\sum_{j=0}^n(-1)^j\rho^n_i(\xi,\vec{\gamma}^j).$$
\end{lemma}
\begin{proof} For any $\sigma\in\mathrm{Sym}(n+1)$ and $j\leq n$ let $\tau(\sigma,j)$ be the $\sigma$-induced bijection of the $n$ coordinates other than $\sigma^{-1}(j)$ with the $n$ coordinates other than $j$ and let $\hat{\tau}(\sigma,j)$ denote its natural image in $\mathrm{Sym}(n)$; we then have:
\begin{align*}
\sum_{j=0}^n(-1)^j\hat{\rho}^n_i(\xi,(\sigma\cdot\vec{\gamma})^j) & = \sum_{j=0}^n(-1)^j\hat{\rho}^n_i(\xi,(\tau(\sigma,j)\cdot\vec{\gamma}^{\sigma^{-1}(j)}) \\
& = \sum_{j=0}^n(-1)^{j}\mathrm{sgn}(\hat{\tau}(\sigma,j))\cdot\rho^n_i(\xi,\vec{\gamma}^{\sigma^{-1}(j)}) \\
& = \mathrm{sgn}(\sigma)\sum_{j=0}^n(-1)^{\sigma^{-1}(j)}\rho^n_i(\xi,\vec{\gamma}^{\sigma^{-1}(j)}).
\end{align*}
The last equality follows from the fact that $\mathrm{sgn}(\hat{\tau}(\sigma,j))=(-1)^{\sigma^{-1}(j)-j}\mathrm{sgn}(\sigma)$.
\end{proof}
Next, we characterize the $\textbf{H}_{n+2}$-configurations within our setting.
\begin{lemma}
\label{lem:Hnchar}
The restriction of the $(n+1)$-hypertournament $\triangleleft^n_i$ to $\vec{\gamma}\in [\varepsilon]^{n+2}$ is of $\mathrm{\mathbf{H}}_{n+2}$-type if and only if for any $\pi\in \mathrm{Sym}(n+1)\backslash\mathrm{Alt}(n+1)$ there exists an $\ell$ such that $\triangleleft^n_i(\pi^{\ell+k}\cdot\vec{\gamma}^k)$ for all $k\leq n+1$.
\end{lemma}
\begin{proof}
For any $\vec{\gamma}\in[\varepsilon]^{n+2}$, $\mathrm{Aut}(\triangleleft_i^n\restriction\vec{\gamma})$ naturally identifies with a subgroup of $\mathrm{Sym}(n+2)$; we show that $\vec{\gamma}$ is as in the lemma's conclusion if and only if $\mathrm{Aut}(\triangleleft_i^n\restriction\vec{\gamma})$ is, under this identification, exactly $\mathrm{Alt}(n+2)$.
To this end, fix a $\sigma\in\mathrm{Sym}(n+2)$. Much as in the proof of the previous lemma, for any $\sigma(k)\leq n+1$ there exists a $\tau(\sigma,\sigma(k))$ with $\hat{\tau}(\sigma,\sigma(k))\in\mathrm{Sym}(n+1)$ of sign $\mathrm{sgn}(\sigma)\cdot (-1)^{\sigma(k)-k}$ such that $$(\sigma\cdot\vec{\gamma})^{\sigma(k)}=\tau(\sigma,\sigma(k))\cdot\vec{\gamma}^k\text{;}$$thus $\sigma\in\mathrm{Aut}(\triangleleft_i^n\restriction\vec{\gamma})$ if and only if $$\tau(\sigma,\sigma(k))\cdot\big(\triangleleft^n_i\restriction\vec{\gamma}^k\big)=\triangleleft^n_i\restriction\vec{\gamma}^{\sigma(k)}$$
for all $k\leq n+1$.
This tells us that $\mathrm{sgn}(\sigma)=1$, and hence that $\sigma\in\mathrm{Alt}(n+2)$, if and only if $\triangleleft^n_i(\vec{\gamma}^k)$ converts to $\triangleleft^n_i\restriction\vec{\gamma}^{\sigma(k)}$ under the natural identification of the $n+1$ coordinates of $\vec{\gamma}^k$ and $\vec{\gamma}^{\sigma(k)}$ via the natural action of the $(\sigma(k)-k)^{\mathrm{th}}$ power of any $\pi\in\mathrm{Sym}(n+1)\backslash\mathrm{Alt}(n+1)$, and this concludes the proof.
\end{proof}
The following lemma, which generalizes the condition (\ref{eq:nocycle}) ensuring that $\triangleleft_i$ is a linear order, now ensures that $\triangleleft^n_i$ is $\textbf{H}_{n+2}$-free.
Interested readers may find it edifying to test out the $n=1$ and $n=2$ cases of both this and the preceding lemma.
\begin{lemma}
\label{lem:Hnarith}
For all $\vec{\gamma}\in [\varepsilon]^{n+2}$, if for some integer $\ell$
\begin{align}
\tag{$e_j$}
\sum_{k=0}^n(-1)^{j+k+\ell}\rho_i^n(\Delta_i^n(\vec{\gamma}^j),(\vec{\gamma}^j)^k)>0
\end{align}
for all $j\leq n$ then
\begin{align*}
\sum_{k=0}^n(-1)^{n+k+\ell}\rho_i^n(\Delta_i^n(\vec{\gamma}^{n+1}),(\vec{\gamma}^{n+1})^k)>0.
\end{align*}
In fact, the conclusion follows from the weaker assumption that $\Delta^n_i(\vec{\gamma}^j)$ is defined for at least one $j\leq n$, and that $(e_j)$ holds whenever it is.
\end{lemma}
\begin{proof}
Assume first that the terms $\Delta^n_i(\vec{\gamma}^j)$ $(j\leq n)$ are all defined.
By standard simplicial arithmetic computations, for any $\xi<\gamma_0$,
$$\sum_{j=0}^{n+1}\sum_{k=0}^n(-1)^{j+k+\ell}\rho_i^n(\xi,(\vec{\gamma}^j)^k)=0,$$
and this together with our assumptions implies that
$$\Delta^n_i(\vec{\gamma}^{n+1})\geq\xi_m:=\min\{\Delta_i^n(\vec{\gamma}^j)\mid j\leq n\}.$$
Replacing the $\Delta$-terms in each equation $e_j$ with $\xi_m$ yields a sequence of $\geq 0$ inequalities, at least one of them strict, whose sum on the left-hand side is
$$\sum_{k=0}^n(-1)^{n+k+\ell}\rho_i^n(\xi_m,(\vec{\gamma}^{n+1})^k).$$
That $\Delta^n_i(\vec{\gamma}^{n+1})=\xi_m$ immediately follows, implying our conclusion.
Next, observe that under the weaker assumptions of our lemma, $\xi_m$ is still defined, and the above reasoning continues to apply wholesale.
\end{proof}
Let us summarize these recognitions in a theorem.
\begin{theorem}
For any $n>0$ and ordinal $\varepsilon$ and function $\rho^n_i:[\varepsilon]^{n+1}\to\mathbb{Z}$, the associated $(n+1)$-ary relation $\triangleleft_i^n$ on $\varepsilon$ is an $\mathrm{\mathbf{H}}_{n+2}$-free $(n+1)$-tournament or, more succinctly, an $n$-order on $\varepsilon$. In particular, just as the classical fiber map ordering $\triangleleft_2$ associated to $\rho_2$ induces a linear ordering of $\omega_1$, the higher fiber map ordering $\triangleleft^n_2$ associated to $\rho^n_2$ defines an $n$-order on $\omega_n$.
\end{theorem}
\begin{proof}
By Lemma \ref{lem:sigmaout}, $\triangleleft^n_i$ is an $(n+1)$-tournament.
Lemma \ref{lem:Hnarith} conjoined with Lemma \ref{lem:Hnchar} shows that the $\triangleleft^n_i$-orientations of the $n$-faces of any $\vec{\gamma}\in [\varepsilon]^{n+2}$ are not of $\textbf{H}_{n+2}$-type if $\Delta_i^n(\vec{\gamma}^j)$ is defined for at least one $j\leq n+1$.
If no such $\Delta_i^n(\vec{\gamma}^j)$ is defined then those orientations are those induced by the standard ordering of the elements of $\vec{\gamma}$.
In particular, $\triangleleft^n_i(\vec{\gamma}^j)$ holds for all $j\leq n+1$, ensuring again via Lemma \ref{lem:Hnchar} that $\triangleleft^n_i$ restricted to $\vec{\gamma}$ is not of $\textbf{H}_{n+2}$-type, completing the proof.
\end{proof}

The theorem opens onto some tantalizing questions. First among these, perhaps, is what distinctions --- ideally Countryman in flavor --- the ordering $\triangleleft^2_2$ induced by an order-$2$ ordertype-minimal $C$-sequence on $\omega_2$ may carry.
At a longer range, questions of classification and even of bases for classes of higher $n$-orders, generalizing the celebrated $n=1$ case of \cite{Moore_Five}, do loom, but more elementary questions are evidently the more immediate order of the day.
Many of these may be productively approached from within, broadly speaking, the framework of \cite{Cherlin_2021}, wherein, much as in our own context, successes at the $n=1$ level --- that of (2-)tournaments or classical walks --- stand in intriguing contrast with the perplexities of the next one. One such success was the classification of countable homogenous tournaments (\cite{Lachlan}; see also \cite{Cherlin88}): as it happens, outside the finite and random cases, there are exactly two; they are $\mathbb{Q}$ (the $\textbf{H}_3$-free case) and the ``local order'' $Q^*$ (essentially a circular order, which \emph{locally} omits $\textbf{H}_3$), and their affinities for the two $1$-manifolds without boundary, the line and the circle, respectively, are striking. We mention this to record two further points:
\begin{itemize}
\item How the greater range of $n$-manifolds (this other prominent class of homogenous structures) for $n>1$ may bear on the classification of countable homogenous $(n+1)$-tournaments is a fascinating and not entirely idle question. Not unrelatedly, as hinted above, a kind of rudimentary geometry of linear and planar conditions seems operative within the class of $2$-orders.
\item Much as for higher walks, part of the intrigue of higher $n$-tournaments is in how lower-order tournaments manifest within them. This can be subtle: $\mathbb{Q}$ and $Q^*$ each induce $2$-orders, for example, but they turn out to induce the same one. Similarly, classification results for uncountable $1$-orders manifest within classification problems for uncountable $2$-orders, and this is part of the charm of the latter.
\end{itemize}
Of course, the classical study of linear orders is intimately related to that of trees \cite{TodTreesLines} (with the Countryman property, within this relation, closely tied to special-ness and coherence; see \cite[Thm.\ 4.4.1]{todwalks} and \cite{PengThesis}); we note in passing that suitable higher-dimensional variations on the tree concept \cite{TFOA}, on chain conditions \cite{TodZhang}, and even Suslin-ness (particularly by Lambie-Hanson) are areas of active inquiry as well. But our account has accumulated so many loose ends at this point that it is clearly time for us to conclude.

\section{Conclusion}
\label{sect:conclusion}
Let us conclude with a list of things which it would be wonderful to know.
\begin{question}
\label{Q1}
Fix $n>1$. Under what conditions does $\mathrm{H}^n(\omega_n;\mathbb{Z})\neq 0$ or  $\mathrm{H}^n(\omega_n;\mathbb{Z}/2)\neq 0$, and under what conditions do higher walks play the role of witnesses to these facts?
\end{question}
This is, of course, a more measured framing of the following question:
\begin{question}
Fix $n>1$. Is $\mathrm{H}^n(\omega_n;\mathbb{Z})\neq 0$ or  $\mathrm{H}^n(\omega_n;\mathbb{Z}/2)\neq 0$ a $\mathsf{ZFC}$ theorem?
\end{question}
The question dominates this area. A \emph{yes} would attest a richness of $\mathsf{ZFC}$ combinatorics above $\omega_1$ which works like this one have only begun to engage, and the grandeur of that scenario alone can seem good grounds for doubting it.
It is not clear, though, that a \emph{no} (whereupon the question reverts to Question \ref{Q1}) entails any simpler a scenario.
One would expect it, for example, to be significantly easier to establish
\begin{equation*}
\label{eq:convanishoffdiagonal}
\tag{*}
\mathrm{Con}(\mathrm{H}^n(\omega_k;\mathbb{Z})=0)\textnormal{ for some finite }k>n>1
\end{equation*}
than anywhere along the diagonal $k=n$, but until quite lately, even this had proven elusive; moreover, in the Eskew-Hayut model going farthest in this direction (\cite{DenseIdeals}, simultaneously witnessing (*) for all  such $k$ and $n$), $\mathrm{H}^n(\omega_n;\mathbb{Z})$ is nevertheless nonzero for every $n\in\omega$.
We defer a fuller discussion of these and related matters to \cite{BLHZ}, where they form a main focus.

All our questions, like those above, admit both more open and more pointed formulations; for the next ones, we will favor the former.
\begin{question}
Under what conditions does $\rho^2_2$ or $\rho_{2,\mathrm{t}}^2$ exhibit unboundedness properties recapitulating those of $\rho_2$?
Under what conditions are the fiber maps of $\rho^2_2$ nontrivial?
Under what conditions are their restrictions to $C_\beta$ for $\beta\in S^2_1$ nontrivial?
\end{question}
Similarly, of course, for higher $\rho^n_2$.
Many of our most fundamental curiosities, though, are best phrased as a research task or program.
What we hope to have shown in the foregoing is that each of the following is, potentially, a quite rewarding one; the first two could conceivably contribute to the solution of multiple questions above, for example.
\begin{task}
Clarify the conditions of the unboundedness of the depth function $d$.
\end{task}
A refinement of this task interacts with the next one: clarify the conditions and extent of the \emph{uncanceled depth} function, where the latter, within the view of $\mathrm{Tr}_n$ as a tree of walks, measures the longest chain of \emph{uncanceled} $(n-1)$-dimensional walks.
\begin{task}
Clarify the mechanisms of cancellation within the upper trace functions $\mathrm{Tr}_n$.
\end{task}

\begin{task}
Develop the classification theory of $2$-orders on $\omega$, $\omega_1$, or $\omega_2$.
\end{task}

\begin{task}
Develop the higher full lower trace function $\mathrm{F}_n$. Does it exhibit higher-dimensional subadditivity properties? What implications, Countryman-type or otherwise, does it carry for the higher linear orders $\triangleleft^n_i$?
\end{task}

\begin{task}
Develop the theory of higher-dimensional $C$-sequences alongside that of the higher walks they induce: extend the derivation, in Theorem \ref{thm:n-coherence}'s item 3, of higher coherence from $\square$-type assumptions to higher dimensions, and develop a complementary theory of nontriviality conditions alongside them.
\end{task}

\begin{task}
Derive strong colorings from higher walks or related oscillation functions, beginning with a witness to $\omega_2\not\to[\omega_1]^3_\omega$.
\end{task}

A concluding speculation is irresistible. Motifs from algebraic topology have, by set theory's standards, played an outsized role in the material of the preceding pages. If the analogy with manifold theory is strong, we should expect a number of problems to simplify beyond some threshold number of dimensions (frequently 5, as in \cite[Thm.\ 1.1]{Smale}, for example), and we cannot help but wonder if this may be the case with some of our own questions as well. The more idiosyncratic dimensions $2$, $3$, and $4$ of distinctive interest to a low-dimensional topologist correspond within our framework to the cardinals $\omega_1$, $\omega_2$, and $\omega_3$, suggesting one more perspective on the distinguished roles the latter play within set theory (many of us are \emph{low-dimensional infinitary combinatoricists}).
Alternately, if our point of reference is $n$-category theory, we might expect dimensions above $3$ to be of an ``essentially unusable'' complexity \cite[p.\ 5]{lurie}.
To put any of this on firmer mathematical footing would be wonderful as well.

\section*{Acknowledgements}
This paper's profound debt to both the mathematics and conversation of Stevo Todorcevic is hopefully clear. I owe most of my knowledge of that and a world of related mathematics, as well as the impetus to unravel its relation to Mitchell's theorem, to Justin Moore.
As Thoreau in ``Walking'' writes: \emph{I have met with but one or two persons in the course of my life who understood the art of Walking, that is, of taking walks} \cite{Thoreau}; I've had the excellent fortune of meeting more, but the inspiration of those two persons remains, for me, singular. That said, this paper's dedication should not be mistaken for their endorsement of any particular perspective or approach espoused herein.
It is a pleasure also to thank Osvaldo Guzm\'{a}n, Michael Hru\v{s}\'{a}k, Chris Lambie-Hanson, Hossein Lamei Ramandi, Assaf Rinot, Mark Schachner, Iian Smythe, Corey Switzer, and Jing Zhang for discussions bearing on this material, and to look forward to more of them. Thanks to Gregory Cherlin, Jan Hubi\v{c}ka, and Mat\v{e}j Kone\v{c}n\'{y} for discussions of, and updates on, the study of $\mathbf{H}_4$-free $3$-hypertournaments. Thanks finally to Bori\v{s}a Kuzeljevi\'{c} for the article invitation, and for his support and patience throughout its preparation.

\newpage
\section*{Appendix: Charting $\mathrm{Tr}_2$}
\label{sect:appendix}

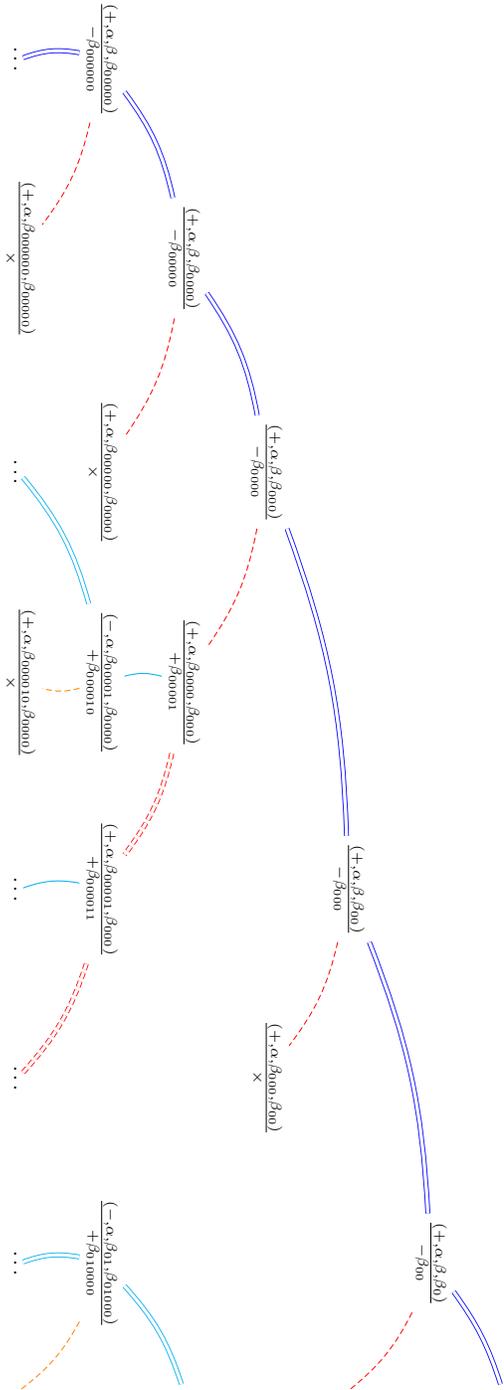
\begin{figure}[!ht]
\begin{adjustbox}{addcode={\begin{minipage}{\width}}{\caption{\Large The expansion of $\mathrm{Tr}_2(+,\alpha,\beta,\gamma)$: inputs and outputs. Depicted is the left wing of an initial portion of this expansion; it is to be viewed in conjunction with the right wing depicted in Figure \ref{fig:app2}, which conceptually precedes it, beginning, as it does, with the input $(+,\alpha,\beta,\gamma)$. These, in turn, should be read in coordination with Figures \ref{fig:app3} and \ref{fig:app4}, which, on the model of classical walks diagrams, more spatially record the ordinals appearing as outputs. The colors above and those of Figure \ref{fig:app4}, for example, deliberately align, with cooler colors (blues and green) for external trajectories and warmer colors (red through yellow) for internal trajectories, a distinction here reinforced by solid and dashed lines. These colors modulate as these trajectories appear at varying depths.
      }\label{fig:app1}\end{minipage}}, scale=.8, rotate=270, center}
      \begin{tikzcd}[ampersand replacement=\&,cramped]
	\&\&\&\&\&\&\& {} \\
	\&\&\&\&\&\& {\frac{(+,\alpha,\beta,\beta_0)}{-\beta_{00}}} \\
	\&\&\&\& {\frac{(+,\alpha,\beta,\beta_{00})}{-\beta_{000}}} \&\&\& {} \\
	\&\& {\frac{(+,\alpha,\beta,\beta_{000})}{-\beta_{0000}}} \&\&\& {\frac{(+,\alpha,\beta_{000},\beta_{00})}{\times}} \\
	\& {\frac{(+,\alpha,\beta,\beta_{0000})}{-\beta_{00000}}} \&\& {\frac{(+,\alpha,\beta_{0000},\beta_{000})}{+\beta_{00001}}} \&\&\&\& {} \\
	{\frac{(+,\alpha,\beta,\beta_{00000})}{-\beta_{000000}}} \&\& {\frac{(+,\alpha,\beta_{00000},\beta_{0000})}{\times}} \& {\frac{(-,\alpha,\beta_{00001},\beta_{0000})}{+\beta_{000010}}} \& {\frac{(+,\alpha,\beta_{00001},\beta_{000})}{+\beta_{000011}}} \&\& {\frac{(-,\alpha,\beta_{01},\beta_{01000})}{+\beta_{010000}}} \\
	\dots \& {\frac{(+,\alpha,\beta_{000000},\beta_{00000})}{\times}} \& \dots \& {\frac{(+,\alpha,\beta_{000010},\beta_{0000})}{\times}} \& \dots \& \dots \& \dots \& {} \\
	\\
	{} \&\& {}
	\arrow[bend right=12, cyan, Rightarrow, no head, from=5-8, to=6-7]
	\arrow[bend left=12, red, dashed, no head, from=2-7, to=3-8]
	\arrow[bend right=9, blue, Rightarrow, no head, from=2-7, to=3-5]
	\arrow[bend left=12, red, dashed, no head, from=3-5, to=4-6]
	\arrow[bend right=9, blue, Rightarrow, no head, from=3-5, to=4-3]
	\arrow[bend left=12, red, dashed, no head, from=4-3, to=5-4]
	\arrow[bend right=15, cyan, no head, from=5-4, to=6-4]
	\arrow[bend left=12, red, Rightarrow, dashed, no head, from=5-4, to=6-5]
	\arrow[bend right=15, cyan, Rightarrow, no head, from=6-7, to=7-7]
	\arrow[bend left=12, orange, dashed, no head, from=6-7, to=7-8]
	\arrow[bend right=15, cyan, no head, from=6-5, to=7-5]
	\arrow[bend left=12, red, Rightarrow, dashed, no head, from=6-5, to=7-6]
	\arrow[bend left=15, orange, dashed, no head, from=6-4, to=7-4]
	\arrow[bend right=12, cyan, Rightarrow, no head, from=6-4, to=7-3]
	\arrow[bend right=12, blue, Rightarrow, no head, from=4-3, to=5-2]
	\arrow[bend left=12, red, dashed, no head, from=5-2, to=6-3]
	\arrow[bend right=12, blue, Rightarrow, no head, from=5-2, to=6-1]
	\arrow[bend left=12, red, dashed, no head, from=6-1, to=7-2]
	\arrow[bend right=15, blue, Rightarrow, no head, from=6-1, to=7-1]
	\arrow[bend left=9, blue, Rightarrow, no head, from=2-7, to=1-8]
\end{tikzcd}
  \end{adjustbox}
\end{figure}

\newpage
\begin{figure}[ht]
  \begin{adjustbox}{addcode={\begin{minipage}{\width}}{\caption{\Large The expansion of $\mathrm{Tr}_2(+,\alpha,\beta,\gamma)$: inputs and outputs (continued). More transitional steps (in the sense of Section \ref{subsect:sample}) are marked with single lines, while the lines of more properly external or internal steps are doubled. This expansion depends, of course, at each input $(\pm,\alpha,\beta',\gamma')$ on assumptions about $C_{\gamma'}$ and, possibly, $C_{\beta'\gamma'}$; as noted in Section \ref{subsect:sample}, though, these assumptions may in their essentials be inferred from the output appearing below the bar. If that output $\beta''$ appears with the opposite sign as the input, for example, then $\beta''=\mathrm{min}\,C_{\gamma'}\backslash\beta'$. If $\times$ appears below the bar, then $\beta'=\mathrm{min}\,C_{\gamma'}\backslash\alpha$ and $(\pm,\alpha,\beta',\gamma')$ is terminal. An ellipsis signals that the input appearing there is not terminal; the outputs of its further expansion are then partially recorded in Figures \ref{fig:app3} and \ref{fig:app4}.
      }\label{fig:app2}\end{minipage}}, scale=.75, rotate=270, center=15cm}
      \begin{tikzcd}[ampersand replacement=\&,cramped]
\&\&\&\&\& {\frac{(+,\alpha,\beta,\gamma)}{-\beta_\varnothing}} \& \\
\&\&\& {\frac{(+,\alpha,\beta,\beta_\varnothing)}{-\beta_0}} \&\&\&\& {\frac{(+,\alpha,\beta_\varnothing,\gamma)}{\times}} \& \\
	\& {\frac{(+,\alpha,\beta,\beta_0)}{-\beta_{00}}} \& {} \&\&\& {} \& \\
\&\&\&\&\& {\frac{(+,\alpha,\beta_0,\beta_\varnothing)}{+\beta_{01}}} \& \\
\&\& {\frac{(+,\alpha,\beta_{00},\beta_0)}{\times}} \&\& {\frac{(-,\alpha,\beta_{01},\beta_0)}{+\beta_{010}}} \&\&\& {\frac{(+,\alpha,\beta_{01},\beta_\varnothing)}{+\beta_{011}}} \& \\
\&\&\& {\frac{(-,\alpha,\beta_{01},\beta_{010})}{+\beta_{0100}}} \&\& {\frac{(-,\alpha,\beta_{010},\beta_0)}{\times}} \& {\frac{(-,\alpha,\beta_{011},\beta_{01})}{+\beta_{0110}}} \& \& \dots \\
\&\& {\frac{(-,\alpha,\beta_{01},\beta_{0100})}{+\beta_{01000}}} \&\& {\frac{(-,\alpha,\beta_{0100},\beta_{010})}{\times}} \& {\frac{(-,\alpha,\beta_{011},\beta_{0110})}{+\beta_{01100}}} \&\& {\frac{(+,\alpha,\beta_{0110},\beta_{01})}{\times}} \& \\
\& {\frac{(-,\alpha,\beta_{01},\beta_{01000})}{+\beta_{010000}}} \&\& {\frac{(-,\alpha,\beta_{01000},\beta_{0100})}{-\beta_{010001}}} \& \dots \&\& {\frac{(+,\alpha,\beta_{01100},\beta_{0110})}{\times}} \& \\
\& \dots \& {\frac{(-,\alpha,\beta_{010000},\beta_{01000})}{\times}} \& \dots \& \dots \& \\
	\\
	{} \&\& {}
	\arrow[bend left=12, orange, dashed, no head, from=6-7, to=7-8]
		\arrow[bend right=15, cyan, Rightarrow, no head, from=8-2, to=9-2]
	\arrow[bend right=12, cyan, no head, from=5-8, to=6-7]
	\arrow[bend left=9, red, Rightarrow, dashed, no head, from=4-6, to=5-8]
	\arrow[bend right=12, cyan, Rightarrow, no head, from=6-7, to=7-6]
	\arrow[bend left=12, orange, dashed, no head, from=7-6, to=8-7]
	\arrow[bend right=12, cyan, Rightarrow, no head, from=7-6, to=8-5]
	\arrow[bend right=12, cyan, no head, from=4-6, to=5-5]
	\arrow[bend left=12, orange, dashed, no head, from=5-5, to=6-6]
	\arrow[bend right=12, cyan, Rightarrow, no head, from=5-5, to=6-4]
	\arrow[bend right=12, cyan, Rightarrow, no head, from=6-4, to=7-3]
		\arrow[bend right=12, cyan, Rightarrow, no head, from=7-3, to=8-2]
	\arrow[bend left=12, orange, dashed, no head, from=6-4, to=7-5]
	\arrow[bend left=12, orange, dashed, no head, from=7-3, to=8-4]
	\arrow[bend left=12, orange, Rightarrow, dashed, no head, from=8-4, to=9-5]
	\arrow[bend right=15, green, no head, from=8-4, to=9-4]
	\arrow[bend left=12, red, Rightarrow, dashed, no head, from=5-8, to=6-9]
		\arrow[bend left=12, red, dashed, no head, from=3-2, to=5-3]
				\arrow[bend left=12, orange, dashed, no head, from=8-2, to=9-3]
	\arrow[bend right=12, blue, Rightarrow, no head, from=2-4, to=3-2]
	\arrow[bend right=9, blue, Rightarrow, no head, from=1-6, to=2-4]
	\arrow[bend left=9, red, dashed, no head, from=2-4, to=4-6]
	\arrow[bend left=9, red, dashed, no head, from=1-6, to=2-8]
\end{tikzcd}
  \end{adjustbox}
\end{figure}

\newpage
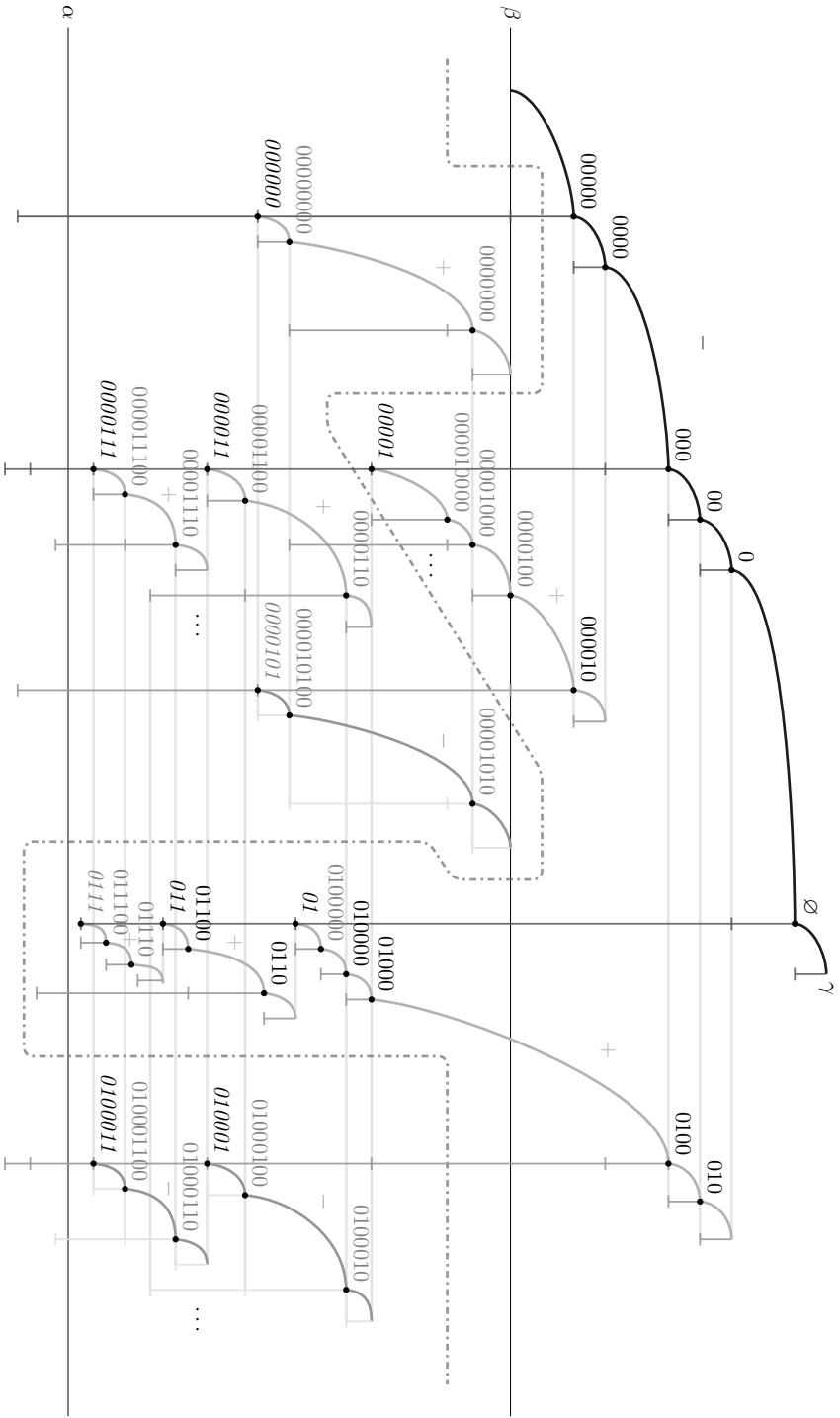
\begin{figure}[!ht]
  \begin{adjustbox}{addcode={\begin{minipage}{\width}}{\caption{\Large The expansion of $\mathrm{Tr}_2(+,\alpha,\beta,\gamma)$: a mapping of outputs (Figure \ref{fig:app4} is its color rendering). Note that a portion of the expansion remains undepicted, marked only by ellipses alongside the lower intervals of $C_{\beta_{00001000}}$,   $C_{\beta_{0000110}}$, and $C_{\beta_{0100010}}$. As the second and third of these intervals are opposite-signed and identical, their associated further expansions will for many purposes cancel; here $\rho^2_2(\alpha,\beta,\gamma)$ will be $11$ plus the contribution of the first ellipsis, for example. This is the meaning of the dashed line traversing the diagram: all of the expansion appearing below it appears in opposite-signed pairs; as noted in Section \ref{subsect:sample}, this line reflects the pattern of $\mathrm{L}(\beta,\gamma)$. To facilitate this analysis, we have marked the sign of each external walk above and to its left, and, the better to trace repeating ordinals, have included horizontal grey guidelines wherever they appear.}\label{fig:app3}\end{minipage}}, scale=.85, rotate=270, center}
     \begin{tikzpicture}
     \selectcolormodel{gray}
	\coordinate (gamma) at (16,14);
	\coordinate (beta) at (1,9);
	\coordinate (alpha) at (1,2);
	\draw (gamma) node[right] {$\gamma$};
	\draw (beta) node[left] {$\beta$};
	\draw (alpha) node[left] {$\alpha$};
	\draw (6,11.8) node[above] {{\color{blue} $-$}};
	\draw (10,9.5) node[above] {{\color{cyan} $+$}};
	\draw (8.4,3.35) node[above] {{\color{cyan} $+$}};
	\draw (17.2,10.3) node[above] {{\color{cyan} $+$}};
	\draw (4.8,7.7) node[above] {{\color{cyan} $+$}};
	\draw (8.6,5.8) node[above] {{\color{cyan} $+$}};
	\draw (15.5,4.4) node[above] {{\color{cyan} $+$}};
	\draw (15.45,2.735) node[above] {{\color{cyan} $+$}};
	\draw (19.4,3.35) node[above] {{\color{green} $-$}};
	\draw (12.3,7.7) node[above] {{\color{green} $-$}};
	\draw (19.6,5.8) node[above] {{\color{green} $-$}};
	\draw (10.5,3.9) node[above] {\dots};
	\draw (21.5,3.9) node[above] {\dots};
	\draw (9.55,7.6) node[above] {\dots};
	\coordinate (b') at (23,9);
	\coordinate (a') at (23,2);
	\draw[ultra thin] (beta)--(b');
	\draw[ultra thin] (alpha)--(a');
	
	\coordinate (01) at (15.2,5.6);
	\coordinate (011) at (15.2,3.5);
	\coordinate (0111) at (15.2,2.2);
	\coordinate (010) at (19.6,12);
	\coordinate (0100) at (19,11.5);
	\coordinate (01000) at (16.4,6.8);
	\coordinate (010000) at (16,6.4);
	\coordinate (0100000) at (15.6,6);
	\coordinate (010001) at (19,4.2);
	\coordinate (0100011) at (19,2.4);
	
	\coordinate (empty) at (15.2,13.5);
	\coordinate (0) at (9.6,12.5);
	\coordinate (00) at (8.8,12);
	\coordinate (000) at (8,11.5);
	\coordinate (0000) at (4.8,10.5);
	\coordinate (00000) at (4,10);
	\coordinate (b) at (2,9);
	
	\coordinate (000000) at (4,5);
	\coordinate (00001) at (8,6.8);
	\coordinate (000011) at (8,4.2);
	\coordinate (0000111) at (8,2.4);	
	
	\coordinate (000010) at (11.5,10);
	\coordinate (0000100) at (10,9);
	\coordinate (00001000) at (9.2,8.4);
	\coordinate (000010000) at (8.8,8);
	
	\coordinate (0000000) at (5.8,8.4);
	\coordinate (00000000) at (4.4,5.5);

	\coordinate (0000101) at (11.5,5);	
	\coordinate (00001010) at (13.3,8.4);
	\coordinate (000010100) at (11.9,5.5);
	\coordinate (0000110) at (10,6.4);
	\coordinate (00001100) at (8.5,4.8);
	\coordinate (0100010) at (21,6.4);
	\coordinate (01000100) at (19.5,4.8);
	
	\coordinate (00001110) at (9.2,3.7);
	\coordinate (000011100) at (8.4,2.9);
	\coordinate (01000110) at (20.2,3.7);
	\coordinate (010001100) at (19.4,2.9);
	\coordinate (0110) at (16.3,5.1);
	\coordinate (01100) at (15.6,3.9);
	\coordinate (01110) at (15.85,3);
	\coordinate (011100) at (15.5,2.6);
	
	\coordinate (1) at (1.5,8);
	\coordinate (2) at (3.2,8);
	\coordinate (3) at (3.2,9.5);
	\coordinate (4) at (6.8,9.5);
	\coordinate (5) at (6.8,6.1);
	\coordinate (6) at (7.35,6.1);
	\coordinate (7) at (12.8,9.5);
	\coordinate (8) at (14.5,9.5);
	\coordinate (8') at (14.5,8.1);
	\coordinate (9') at (13.9,7.7);
	\coordinate (9) at (13.9,1.3);
	\coordinate (10) at (17.3,1.3);
	\coordinate (11) at (17.3,8);
	\coordinate (12) at (22.5,8);

	\draw[very thick, rounded corners, dashdotted, magenta, opacity=.7] (1)--(2)--(3)--(4)--(5)--(6)--(7)--(8)--(8')--(9')--(9)--(10)--(11)--(12);
	
	\draw[very thick, gray, opacity=.2] (01)--(16.7,5.6);
	\draw[very thick, gray, opacity=.2] (011)--(16.1,3.5);
	\draw[very thick, gray, opacity=.2] (10,3.3)--(21,3.3);
	\draw[very thick, gray, opacity=.2] (0000111)--(19.4,2.4);
	\draw[very thick, gray, opacity=.2] (00001110)--(20.6,3.7);
	\draw[very thick, gray, opacity=.2] (000011100)--(20.2,2.9);
	\draw[very thick, gray, opacity=.2] (8,6.8)--(21.5,6.8);
	\draw[very thick, gray, opacity=.2] (0000110)--(21.5,6.4);
	\draw[very thick, gray, opacity=.2] (00001100)--(21,4.8);
	\draw[very thick, gray, opacity=.2] (000011)--(20.6,4.2);	
	\draw[very thick, gray, opacity=.2] (0000000)--(14,8.4);
	\draw[very thick, gray, opacity=.2] (00000000)--(13.3,5.5);
	\draw[very thick, gray, opacity=.2] (000000)--(11.9,5);	
	\draw[very thick, gray, opacity=.2] (0)--(20.2,12.5);
	\draw[very thick, gray, opacity=.2] (00)--(20.2,12);
	\draw[very thick, gray, opacity=.2] (000)--(19.6,11.5);
	\draw[very thick, gray, opacity=.2] (0000)--(12,10.5);
	\draw[very thick, gray, opacity=.2] (00000)--(12,10);

	\draw[thick, orange, opacity=.9] (20.2,12.5)--(20.2,12);
	\draw[thick, orange, opacity=.9] (20.1,12)--(20.3,12);
	\draw[thick, orange, opacity=.9] (19.5,11.5)--(19.7,11.5);
	\draw[thick, orange, opacity=.9] (19.6,12)--(19.6,11.5);
	\draw[thick, orange, opacity=.9] (12,10.5)--(12,10);
	\draw[thick, orange, opacity=.9] (11.9,10)--(12.1,10);
	
	\draw[thick, orange, opacity=.9] (16.1,3.5)--(16.1,3.1);
	\draw[thick, orange, opacity=.9] (01110)--(15.85,2.6);
	\draw[thick, orange, opacity=.9] (011100)--(15.5,2.2);
	\draw[thick, orange, opacity=.9] (16,3.1)--(16.2,3.1);
	\draw[thick, orange, opacity=.9] (15.75,2.6)--(15.95,2.6);
	\draw[thick, orange, opacity=.9] (15.4,2.2)--(15.6,2.2);
	
	\draw[very thick, blue] (gamma) to[out=-180,in=85, distance=.3cm] (empty);
	 \draw[very thick, blue] (empty) to[out=-180,in=85, distance=.9cm] (0);
	 \draw[very thick, blue, distance=.3cm] (0) to[out=-180,in=85] (00);
	  \draw[very thick, blue] (00) to[out=-180,in=85, distance=.3cm] (000);
	  \draw[very thick, blue] (000) to[out=-180,in=85, distance=.7cm] (0000);
	  \draw[very thick, blue] (0000) to[out=-180,in=85, distance=.3cm] (00000);
	  \draw[very thick, blue] (00000) to[out=-180,in=85, distance=.55cm] (b);
	 
		\draw[very thick, cyan] (6.5,9) to[out=-180,in=85, distance=.3cm] (0000000);
	    \draw[very thick, cyan] (0000000) to[out=-180,in=85, distance=.8cm] (00000000);
	     \draw[very thick, cyan] (00000000) to[out=-180,in=85, distance=.25cm] (000000);	 
	    
		\draw[very thick, cyan] (16.7,5.6) to[out=-180,in=85, distance=.25cm] (0110);
	    \draw[very thick, cyan] (0110) to[out=-180,in=85, distance=.5cm] (01100);
	     \draw[very thick, cyan] (01100) to[out=-180,in=85, distance=.25cm] (011);
	     
	     		\draw[very thick, cyan] (16.15,3.5) to[out=-180,in=85, distance=.28cm] (01110);
	    \draw[very thick, cyan] (01110) to[out=-180,in=85, distance=.25cm] (011100);
	     \draw[very thick, cyan] (011100) to[out=-180,in=85, distance=.25cm] (0111);
	     	 	     
	     \draw[very thick, green] (14,9) to[out=-180,in=85, distance=.3cm] (00001010);
	    \draw[very thick, green] (00001010) to[out=-180,in=85, distance=.8cm] (000010100);
	     \draw[very thick, green] (000010100) to[out=-180,in=85, distance=.25cm] (0000101);
	  	
		\draw[very thick, cyan] (20.2,12.5) to[out=-180,in=85, distance=.3cm] (010);
	    \draw[very thick, cyan] (010) to[out=-180,in=85, distance=.3cm] (0100);
	     \draw[very thick, cyan] (0100) to[out=-180,in=85, distance=1.2cm] (01000);
	     \draw[very thick, cyan] (01000) to[out=-180,in=85, distance=.25cm] (010000);
	    \draw[very thick, cyan] (010000) to[out=-180,in=85, distance=.25cm] (0100000);
	     \draw[very thick, cyan] (0100000) to[out=-180,in=85, distance=.25cm] (01);

		\draw[very thick, cyan] (10.5,6.8) to[out=-180,in=85, distance=.3cm] (0000110);
	    \draw[very thick, cyan] (0000110) to[out=-180,in=85, distance=.8cm] (00001100);
	     \draw[very thick, cyan] (00001100) to[out=-180,in=85, distance=.3cm] (000011);
	     
	     \draw[very thick, cyan] (9.6,4.2) to[out=-180,in=85, distance=.27cm] (00001110);
	    \draw[very thick, cyan] (00001110) to[out=-180,in=85, distance=.5cm] (000011100);
	     \draw[very thick, cyan] (000011100) to[out=-180,in=85, distance=.27cm] (0000111);
	     
	     	     \draw[very thick, green] (20.6,4.2) to[out=-180,in=85, distance=.27cm] (01000110);
	    \draw[very thick, green] (01000110) to[out=-180,in=85, distance=.5cm] (010001100);
	     \draw[very thick, green] (010001100) to[out=-180,in=85, distance=.27cm] (0100011);
	     
	     \draw[very thick, green] (21.5,6.8) to[out=-180,in=85, distance=.3cm] (0100010);
	    \draw[very thick, green] (0100010) to[out=-180,in=85, distance=.8cm] (01000100);
	     \draw[very thick, green] (01000100) to[out=-180,in=85, distance=.3cm] (010001);

	     \draw[very thick, cyan] (12,10.5) to[out=-180,in=85, distance=.3cm] (000010);
		\draw[very thick, cyan] (000010) to[out=-180,in=85, distance=.5cm] (0000100);
	     \draw[very thick, cyan] (0000100) to[out=-180,in=85, distance=.4cm] (00001000);
	     \draw[very thick, cyan] (00001000) to[out=-180,in=85, distance=.2cm] (000010000);
	     \draw[very thick, cyan] (000010000) to[out=-180,in=85, distance=.4cm] (00001);
	     
	\draw[thick, red, opacity=.8] (gamma)--(16,13.5);	
	\draw[thick, red, opacity=.8] (15.9,13.5)--(16.1,13.5);	
	
	\draw[thick, red, opacity=.8] (empty)--(15.2,2.2);
	\draw[thick, red, opacity=.8] (15.1,12.5)--(15.3,12.5);
	\draw[thick, red, opacity=.8] (15.1,5.6)--(15.3,5.6);
	\draw[thick, red, opacity=.8] (15.1,3.5)--(15.3,3.5);
	\draw[thick, red, opacity=.8] (15.1,2.2)--(15.3,2.2);

	\draw[thick, red, opacity=.8] (0)--(9.6,12);
	\draw[thick, red, opacity=.8] (9.5,12)--(9.7,12);		  
	  
	\draw[thick, red, opacity=.8] (00)--(8.8,11.5);	 
	\draw[thick, red, opacity=.8] (8.7,11.5)--(8.9,11.5); 
	  
	\draw[thick, red, opacity=.8] (000)--(8,1);
	\draw[thick, red, opacity=.8] (7.9,10.5)--(8.1,10.5);
	\draw[thick, red, opacity=.8] (7.9,6.8)--(8.1,6.8);
	\draw[thick, red, opacity=.8] (7.9,4.2)--(8.1,4.2);
	\draw[thick, red, opacity=.8] (7.9,2.4)--(8.1,2.4);	 
	\draw[thick, red, opacity=.8] (7.9,1.4)--(8.1,1.4);
	\draw[thick, red, opacity=.8] (7.9,1)--(8.1,1);
	
	\draw[thick, orange, opacity=.9] (000010)--(11.5,1.2);
	\draw[thick, orange, opacity=.9] (11.4,9)--(11.6,9);
	\draw[thick, orange, opacity=.9] (11.4,5)--(11.6,5);
	\draw[thick, orange, opacity=.9] (11.4,1.2)--(11.6,1.2);
	
	\draw[thick, orange, opacity=.9] (0000100)--(10,8.4);
	\draw[thick, orange, opacity=.9] (00001000)--(9.2,5.5);
	
	\draw[thick, orange, opacity=.9] (16.7,5.6)--(16.7,5.1);
	\draw[thick, orange, opacity=.9] (0110)--(16.3,1.5);
	\draw[thick, orange, opacity=.9] (01100)--(15.6,3.5);
		\draw[thick, orange, opacity=.9] (16.2,3.9)--(16.4,3.9);
	\draw[thick, orange, opacity=.9] (16.2,1.5)--(16.4,1.5);
	\draw[thick, orange, opacity=.9] (15.5,3.5)--(15.7,3.5);
	\draw[thick, orange, opacity=.9] (16.6,5.1)--(16.8,5.1);
	
	\draw[thick, orange, opacity=.9] (9.6,4.2)--(9.6,3.7);
	\draw[thick, orange, opacity=.9] (00001110)--(9.2,1.8);
	\draw[thick, orange, opacity=.9] (000011100)--(8.4,2.4);
	\draw[thick, orange, opacity=.9] (9.5,3.7)--(9.7,3.7);
	\draw[thick, orange, opacity=.9] (9.1,1.8)--(9.3,1.8);
	\draw[thick, orange, opacity=.9] (9.1,2.9)--(9.3,2.9);
	\draw[thick, orange, opacity=.9] (8.3,2.4)--(8.5,2.4);
	
		\draw[thick, yellow, opacity=1] (20.6,4.2)--(20.6,3.7);
	\draw[thick, yellow, opacity=1] (01000110)--(20.2,1.8);
	\draw[thick, yellow, opacity=1] (010001100)--(19.4,2.4);
	\draw[thick, yellow, opacity=1] (20.5,3.7)--(20.7,3.7);
	\draw[thick, yellow, opacity=1] (20.1,1.8)--(20.3,1.8);
	\draw[thick, yellow, opacity=1] (20.1,2.9)--(20.3,2.9);
	\draw[thick, yellow, opacity=1] (19.3,2.4)--(19.5,2.4);
	
	\draw[thick, orange, opacity=.9] (000010000)--(8.8,6.8);
	\draw[thick, orange, opacity=.9] (9.9,8.4)--(10.1,8.4);
	\draw[thick, orange, opacity=.9] (9.1,5.5)--(9.3,5.5);
	\draw[thick, orange, opacity=.9] (9.1,8)--(9.3,8);
	\draw[thick, orange, opacity=.9] (8.7,6.8)--(8.9,6.8);
	
		\draw[thick, orange, opacity=.9] (10.5,6.8)--(10.5,6.4);
	\draw[thick, orange, opacity=.9] (0000110)--(10,3.3);
	\draw[thick, orange, opacity=.9] (00001100)--(8.5,4.2);
	\draw[thick, orange, opacity=.9] (10.4,6.4)--(10.6,6.4);
	\draw[thick, orange, opacity=.9] (9.9,4.8)--(10.1,4.8);
	\draw[thick, orange, opacity=.9] (9.9,3.3)--(10.1,3.3);
	\draw[thick, orange, opacity=.9] (8.4,4.2)--(8.6,4.2);
	
	\draw[thick, yellow, opacity=1] (21.5,6.8)--(21.5,6.4);
	\draw[thick, yellow, opacity=1] (0100010)--(21,3.3);
	\draw[thick, yellow, opacity=1] (01000100)--(19.5,4.2);
	\draw[thick, yellow, opacity=1] (21.4,6.4)--(21.6,6.4);
	\draw[thick, yellow, opacity=1] (20.9,4.8)--(21.1,4.8);
	\draw[thick, yellow, opacity=1] (20.9,3.3)--(21.1,3.3);
	\draw[thick, yellow, opacity=1] (19.4,4.2)--(19.6,4.2);
	
	\draw[thick, orange, opacity=.9] (01000)--(16.4,6.4);
	\draw[thick, orange, opacity=.9] (010000)--(16,6);
	\draw[thick, orange, opacity=.9] (0100000)--(15.6,5.6);
	\draw[thick, orange, opacity=.9] (16.3,6.4)--(16.5,6.4);
	\draw[thick, orange, opacity=.9] (15.9,6)--(16.1,6);
	\draw[thick, orange, opacity=.9] (15.7,5.6)--(15.5,5.6);
	
	\draw[thick, red, opacity=.8] (0000)--(4.8,10);	
	\draw[thick, red, opacity=.8] (4.7,10)--(4.9,10);	
	
	\draw[thick, orange, opacity=.9] (6.5,9)--(6.5,8.4);
	\draw[thick, orange, opacity=.9] (0000000)--(5.8,5.5);
	\draw[thick, orange, opacity=.9] (00000000)--(4.4,5);	
	\draw[thick, orange, opacity=.9] (6.4,8.4)--(6.6,8.4);
	\draw[thick, orange, opacity=.9] (5.7,5.5)--(5.9,5.5);
	\draw[thick, orange, opacity=.9] (4.3,5)--(4.5,5);
	\draw[thick, orange, opacity=.9] (5.7,8)--(5.9,8);
	
	\draw[thick, yellow, opacity=1] (14,9)--(14,8.4);
	\draw[thick, yellow, opacity=1] (00001010)--(13.3,5.5);
	\draw[thick, yellow, opacity=1] (000010100)--(11.9,5);	
	\draw[thick, yellow, opacity=1] (13.9,8.4)--(14.1,8.4);
	\draw[thick, yellow, opacity=1] (13.2,5.5)--(13.4,5.5);
	\draw[thick, yellow, opacity=1] (11.8,5)--(12,5);
	\draw[thick, yellow, opacity=1] (13.2,8)--(13.4,8);
	
	\draw[thick, orange, opacity=.9] (19,11.5)--(19,1);
	\draw[thick, orange, opacity=.9] (18.9,10.5)--(19.1,10.5);
	\draw[thick, orange, opacity=.9] (18.9,6.8)--(19.1,6.8);
	\draw[thick, orange, opacity=.9] (18.9,4.2)--(19.1,4.2);
	\draw[thick, orange, opacity=.9] (18.9,2.4)--(19.1,2.4);	 
	\draw[thick, orange, opacity=.9] (18.9,1.4)--(19.1,1.4);
	\draw[thick, orange, opacity=.9] (18.9,1)--(19.1,1);
	\draw[thick, orange, opacity=.9] (18.9,11.5)--(19.1,11.5);
	
	\draw[thick, red, opacity=.9] (00000)--(4,1.2);
	\draw[thick, red, opacity=.8] (3.9,9)--(4.1,9);
	\draw[thick, red, opacity=.8] (3.9,5)--(4.1,5);
	\draw[thick, red, opacity=.8] (3.9,1.2)--(4.1,1.2);
	  
	\draw (empty) node[above left] {$\varnothing$};
	\draw (0) node[above left] {$0$};
	\draw (00) node[above left] {$00$};
	\draw (000) node[above left] {$000$};
	\draw (0000) node[above left] {$0000$};
	\draw (00000) node[above left] {$00000$};
	\draw (01) node[above left] {$\mathit{01}$};
	\draw (011) node[above left] {$\mathit{011}$};
	\draw (00001) node[above left] {$\mathit{00001}$};
	\draw (000011) node[above left] {$\mathit{000011}$};
	\draw (0000111) node[above left] {$\mathit{0000111}$};
	\draw (000000) node[above left] {$\mathit{000000}$};
	\draw (010) node[above left] {$010$};
	\draw (0100) node[above left] {$0100$};
	\draw (01000) node[above left] {$01000$};
	\draw (010000) node[above left] {$010000$};	
	\draw (0100000) node[above left] {\color{gray} $0100000$};	
	\draw (000010) node[above left] {$000010$};
	\draw (0000100) node[above left] {\color{gray} $0000100$};
	\draw (00001000) node[above left] {\color{gray} $00001000$};
	\draw (000010000) node[above left] {\color{gray} $000010000$};
	\draw (010001) node[above left] {$\mathit{010001}$};
	\draw (0100011) node[above left] {$\mathit{0100011}$};
	\draw (0000000) node[above left] {\color{gray} $0000000$};
	\draw (00000000) node[above left] {\color{gray} $00000000$};
	\draw (0000101) node[above left] {\color{gray} $\mathit{0000101}$};
	\draw (00001010) node[above left] {\color{gray} $00001010$};
	\draw (000010100) node[above left] {\color{gray} $000010100$};
	\draw (0111) node[above left] {\color{gray} $\mathit{0111}$};
	\draw (0000110) node[above left] {\color{gray} $0000110$};
	\draw (00001100) node[above left] {\color{gray} $00001100$};
	\draw (0100010) node[above left] {\color{gray} $0100010$};
	\draw (01000100) node[above left] {\color{gray} $01000100$};
	
	\draw (00001110) node[above left] {\color{gray} $00001110$};
	\draw (000011100) node[above left] {\color{gray} $000011100$};
	\draw (01000110) node[above left] {\color{gray} $01000110$};
	\draw (0110) node[above left] {$0110$};
	\draw (01100) node[above left] {$01100$};
	\draw (01110) node[above left] {\color{gray} $01110$};
	\draw (011100) node[above left] {\color{gray} $011100$};
	\draw (010001100) node[above left] {\color{gray} $010001100$};
	
	\draw (empty) node[circle,draw,fill=black, scale=.25]{};
	\draw (0) node[circle,draw,fill=black, scale=.25]{};
	\draw (00) node[circle,draw,fill=black, scale=.25]{};
	\draw (000) node[circle,draw,fill=black, scale=.25]{};
	\draw (0000) node[circle,draw,fill=black, scale=.25]{};
	\draw (00000) node[circle,draw,fill=black, scale=.25]{};
	\draw (01) node[circle,draw,fill=black, scale=.25]{};
	\draw (011) node[circle,draw,fill=black, scale=.25]{};
	\draw (00001) node[circle,draw,fill=black, scale=.25]{};
	\draw (000011) node[circle,draw,fill=black, scale=.25]{};
	\draw (0000111) node[circle,draw,fill=black, scale=.25]{};
	\draw (000000) node[circle,draw,fill=black, scale=.25]{};
	\draw (010) node[circle,draw,fill=black, scale=.25]{};
	\draw (0100) node[circle,draw,fill=black, scale=.25]{};
	\draw (01000) node[circle,draw,fill=black, scale=.25]{};
	\draw (010000) node[circle,draw,fill=black, scale=.25]{};
	\draw (0100000) node[circle,draw,fill=black, scale=.25]{};
	\draw (000010) node[circle,draw,fill=black, scale=.25]{};
	\draw (0000100) node[circle,draw,fill=black, scale=.25]{};
	\draw (00001000) node[circle,draw,fill=black, scale=.25]{};
	\draw (000010000) node[circle,draw,fill=black, scale=.25]{};
	\draw (010001) node[circle,draw,fill=black, scale=.25]{};
	\draw (0100011) node[circle,draw,fill=black, scale=.25]{};
	\draw (0000000) node[circle,draw,fill=black, scale=.25]{};
	\draw (00000000) node[circle,draw,fill=black, scale=.25]{};
	\draw (0000101) node[circle,draw,fill=black, scale=.25]{};
	\draw (00001010) node[circle,draw,fill=black, scale=.25]{};
	\draw (000010100) node[circle,draw,fill=black, scale=.25]{};
	\draw (0111) node[circle,draw,fill=black, scale=.25]{};
	\draw (0000110) node[circle,draw,fill=black, scale=.25]{};
	\draw (00001100) node[circle,draw,fill=black, scale=.25]{};
	\draw (0100010) node[circle,draw,fill=black, scale=.25]{};
	\draw (01000100) node[circle,draw,fill=black, scale=.25]{};
	\draw (00001110) node[circle,draw,fill=black, scale=.25]{};
	\draw (000011100) node[circle,draw,fill=black, scale=.25]{};
	\draw (01000110) node[circle,draw,fill=black, scale=.25]{};
	\draw (010001100) node[circle,draw,fill=black, scale=.25]{};
	\draw (0110) node[circle,draw,fill=black, scale=.25]{};
	\draw (01100) node[circle,draw,fill=black, scale=.25]{};
	\draw (01110) node[circle,draw,fill=black, scale=.25]{};
	\draw (011100) node[circle,draw,fill=black, scale=.25]{};
	\end{tikzpicture}
  \end{adjustbox}
\end{figure}

\newpage
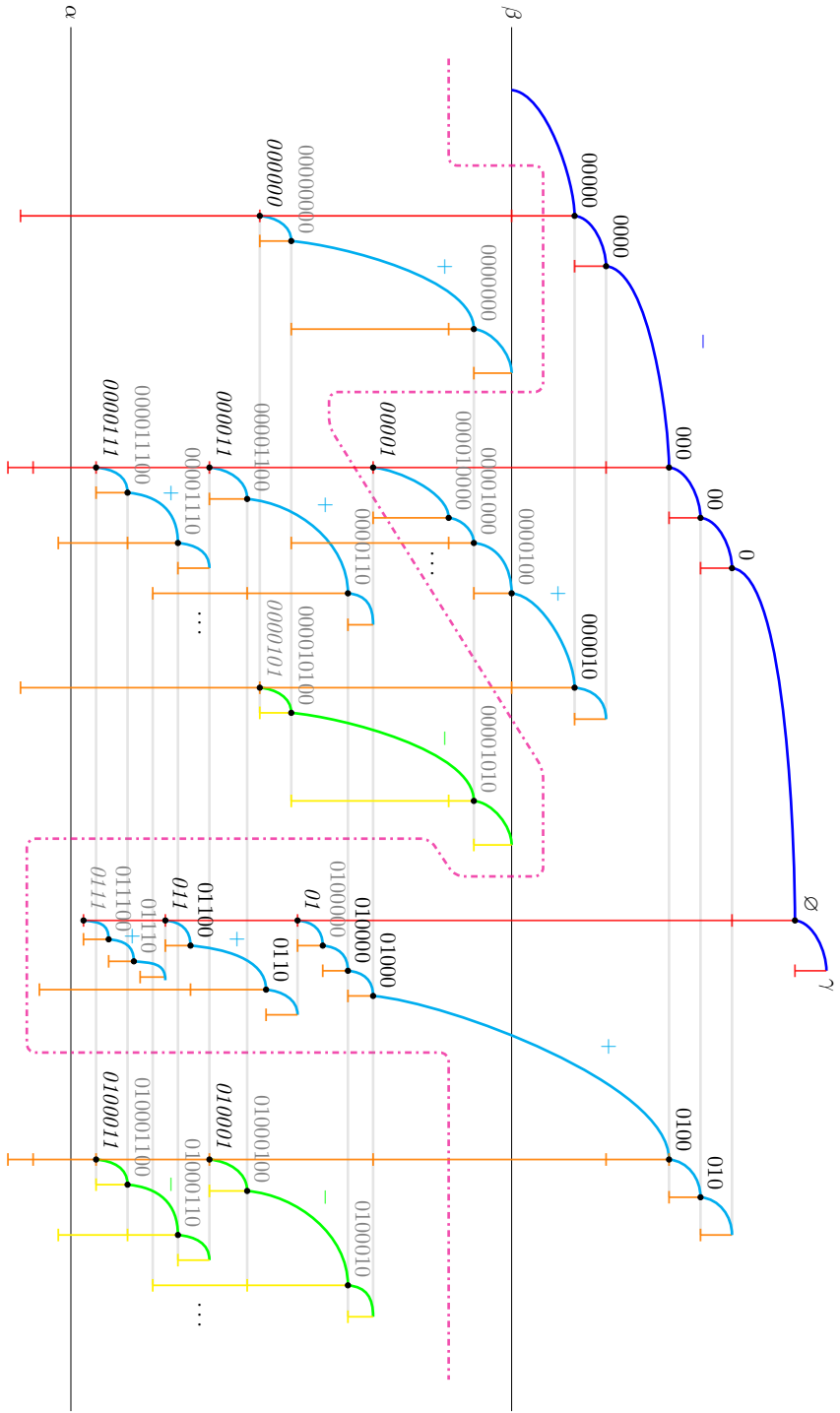
\begin{figure}[!ht]
  \begin{adjustbox}{addcode={\begin{minipage}{\width}}{\caption{\Large The expansion of $\mathrm{Tr}_2(+,\alpha,\beta,\gamma)$: a mapping of outputs (continued). As in Figures \ref{fig:app1} and \ref{fig:app2}, the steps of external walks have been colored to reflect the depth at which they appear; those from $\beta_{0000100}$ down to $\beta_{0000101}$, for example, appear at depth $3$ and are colored green accordingly (as noted, depth is legible also in the numbers of alternations $0$s and $1$s among the associated indices). The depth of internal walks (whose associated indices we have italicized), similarly, is reflected in the descending red-orange-yellow sequence of associated clubs. Omitting inputs, of course, allows us to record more outputs than in Figures \ref{fig:app1} and \ref{fig:app2}; here, only those in black font appeared among the former. Observe lastly that the top-to-bottom order in which they appear differs in some respects from that in which they appear in Figures \ref{fig:app1} and \ref{fig:app2}.} \label{fig:app4}\end{minipage}}, scale=.85, rotate=270, center}
     \begin{tikzpicture}
	\coordinate (gamma) at (16,14);
	\coordinate (beta) at (1,9);
	\coordinate (alpha) at (1,2);
	\draw (gamma) node[right] {$\gamma$};
	\draw (beta) node[left] {$\beta$};
	\draw (alpha) node[left] {$\alpha$};
	\draw (6,11.8) node[above] {{\color{blue} $-$}};
	\draw (10,9.5) node[above] {{\color{cyan} $+$}};
	\draw (8.4,3.35) node[above] {{\color{cyan} $+$}};
	\draw (17.2,10.3) node[above] {{\color{cyan} $+$}};
	\draw (4.8,7.7) node[above] {{\color{cyan} $+$}};
	\draw (8.6,5.8) node[above] {{\color{cyan} $+$}};
	\draw (15.5,4.4) node[above] {{\color{cyan} $+$}};
	\draw (15.45,2.735) node[above] {{\color{cyan} $+$}};
	\draw (19.4,3.35) node[above] {{\color{green} $-$}};
	\draw (12.3,7.7) node[above] {{\color{green} $-$}};
	\draw (19.6,5.8) node[above] {{\color{green} $-$}};
	\draw (10.5,3.9) node[above] {\dots};
	\draw (21.5,3.9) node[above] {\dots};
	\draw (9.55,7.6) node[above] {\dots};
	\coordinate (b') at (23,9);
	\coordinate (a') at (23,2);
	\draw[ultra thin] (beta)--(b');
	\draw[ultra thin] (alpha)--(a');
	
	\coordinate (01) at (15.2,5.6);
	\coordinate (011) at (15.2,3.5);
	\coordinate (0111) at (15.2,2.2);
	\coordinate (010) at (19.6,12);
	\coordinate (0100) at (19,11.5);
	\coordinate (01000) at (16.4,6.8);
	\coordinate (010000) at (16,6.4);
	\coordinate (0100000) at (15.6,6);
	\coordinate (010001) at (19,4.2);
	\coordinate (0100011) at (19,2.4);
	
	\coordinate (empty) at (15.2,13.5);
	\coordinate (0) at (9.6,12.5);
	\coordinate (00) at (8.8,12);
	\coordinate (000) at (8,11.5);
	\coordinate (0000) at (4.8,10.5);
	\coordinate (00000) at (4,10);
	\coordinate (b) at (2,9);
	
	\coordinate (000000) at (4,5);
	\coordinate (00001) at (8,6.8);
	\coordinate (000011) at (8,4.2);
	\coordinate (0000111) at (8,2.4);	
	
	\coordinate (000010) at (11.5,10);
	\coordinate (0000100) at (10,9);
	\coordinate (00001000) at (9.2,8.4);
	\coordinate (000010000) at (8.8,8);
	
	\coordinate (0000000) at (5.8,8.4);
	\coordinate (00000000) at (4.4,5.5);

	\coordinate (0000101) at (11.5,5);	
	\coordinate (00001010) at (13.3,8.4);
	\coordinate (000010100) at (11.9,5.5);
	\coordinate (0000110) at (10,6.4);
	\coordinate (00001100) at (8.5,4.8);
	\coordinate (0100010) at (21,6.4);
	\coordinate (01000100) at (19.5,4.8);
	
	\coordinate (00001110) at (9.2,3.7);
	\coordinate (000011100) at (8.4,2.9);
	\coordinate (01000110) at (20.2,3.7);
	\coordinate (010001100) at (19.4,2.9);
	\coordinate (0110) at (16.3,5.1);
	\coordinate (01100) at (15.6,3.9);
	\coordinate (01110) at (15.85,3);
	\coordinate (011100) at (15.5,2.6);
	
	\coordinate (1) at (1.5,8);
	\coordinate (2) at (3.2,8);
	\coordinate (3) at (3.2,9.5);
	\coordinate (4) at (6.8,9.5);
	\coordinate (5) at (6.8,6.1);
	\coordinate (6) at (7.35,6.1);
	\coordinate (7) at (12.8,9.5);
	\coordinate (8) at (14.5,9.5);
	\coordinate (8') at (14.5,8.1);
	\coordinate (9') at (13.9,7.7);
	\coordinate (9) at (13.9,1.3);
	\coordinate (10) at (17.3,1.3);
	\coordinate (11) at (17.3,8);
	\coordinate (12) at (22.5,8);
	
	\draw[very thick, rounded corners, dashdotted, magenta, opacity=.7] (1)--(2)--(3)--(4)--(5)--(6)--(7)--(8)--(8')--(9')--(9)--(10)--(11)--(12);
	
	\draw[very thick, gray, opacity=.2] (01)--(16.7,5.6);
	\draw[very thick, gray, opacity=.2] (011)--(16.1,3.5);
	\draw[very thick, gray, opacity=.2] (10,3.3)--(21,3.3);
	\draw[very thick, gray, opacity=.2] (0000111)--(19.4,2.4);
	\draw[very thick, gray, opacity=.2] (00001110)--(20.6,3.7);
	\draw[very thick, gray, opacity=.2] (000011100)--(20.2,2.9);
	\draw[very thick, gray, opacity=.2] (8,6.8)--(21.5,6.8);
	\draw[very thick, gray, opacity=.2] (0000110)--(21.5,6.4);
	\draw[very thick, gray, opacity=.2] (00001100)--(21,4.8);
	\draw[very thick, gray, opacity=.2] (000011)--(20.6,4.2);	
	\draw[very thick, gray, opacity=.2] (0000000)--(14,8.4);
	\draw[very thick, gray, opacity=.2] (00000000)--(13.3,5.5);
	\draw[very thick, gray, opacity=.2] (000000)--(11.9,5);	
	\draw[very thick, gray, opacity=.2] (0)--(20.2,12.5);
	\draw[very thick, gray, opacity=.2] (00)--(20.2,12);
	\draw[very thick, gray, opacity=.2] (000)--(19.6,11.5);
	\draw[very thick, gray, opacity=.2] (0000)--(12,10.5);
	\draw[very thick, gray, opacity=.2] (00000)--(12,10);

	\draw[thick, orange, opacity=.9] (20.2,12.5)--(20.2,12);
	\draw[thick, orange, opacity=.9] (20.1,12)--(20.3,12);
	\draw[thick, orange, opacity=.9] (19.5,11.5)--(19.7,11.5);
	\draw[thick, orange, opacity=.9] (19.6,12)--(19.6,11.5);
	\draw[thick, orange, opacity=.9] (12,10.5)--(12,10);
	\draw[thick, orange, opacity=.9] (11.9,10)--(12.1,10);
	
	\draw[thick, orange, opacity=.9] (16.1,3.5)--(16.1,3.1);
	\draw[thick, orange, opacity=.9] (01110)--(15.85,2.6);
	\draw[thick, orange, opacity=.9] (011100)--(15.5,2.2);
	\draw[thick, orange, opacity=.9] (16,3.1)--(16.2,3.1);
	\draw[thick, orange, opacity=.9] (15.75,2.6)--(15.95,2.6);
	\draw[thick, orange, opacity=.9] (15.4,2.2)--(15.6,2.2);
	
	\draw[very thick, blue] (gamma) to[out=-180,in=85, distance=.3cm] (empty);
	 \draw[very thick, blue] (empty) to[out=-180,in=85, distance=.9cm] (0);
	 \draw[very thick, blue, distance=.3cm] (0) to[out=-180,in=85] (00);
	  \draw[very thick, blue] (00) to[out=-180,in=85, distance=.3cm] (000);
	  \draw[very thick, blue] (000) to[out=-180,in=85, distance=.7cm] (0000);
	  \draw[very thick, blue] (0000) to[out=-180,in=85, distance=.3cm] (00000);
	  \draw[very thick, blue] (00000) to[out=-180,in=85, distance=.55cm] (b);
	 
		\draw[very thick, cyan] (6.5,9) to[out=-180,in=85, distance=.3cm] (0000000);
	    \draw[very thick, cyan] (0000000) to[out=-180,in=85, distance=.8cm] (00000000);
	     \draw[very thick, cyan] (00000000) to[out=-180,in=85, distance=.25cm] (000000);	 
	    
		\draw[very thick, cyan] (16.7,5.6) to[out=-180,in=85, distance=.25cm] (0110);
	    \draw[very thick, cyan] (0110) to[out=-180,in=85, distance=.5cm] (01100);
	     \draw[very thick, cyan] (01100) to[out=-180,in=85, distance=.25cm] (011);
	     
	     		\draw[very thick, cyan] (16.15,3.5) to[out=-180,in=85, distance=.28cm] (01110);
	    \draw[very thick, cyan] (01110) to[out=-180,in=85, distance=.25cm] (011100);
	     \draw[very thick, cyan] (011100) to[out=-180,in=85, distance=.25cm] (0111);
	     	 	     
	     \draw[very thick, green] (14,9) to[out=-180,in=85, distance=.3cm] (00001010);
	    \draw[very thick, green] (00001010) to[out=-180,in=85, distance=.8cm] (000010100);
	     \draw[very thick, green] (000010100) to[out=-180,in=85, distance=.25cm] (0000101);
	  	
		\draw[very thick, cyan] (20.2,12.5) to[out=-180,in=85, distance=.3cm] (010);
	    \draw[very thick, cyan] (010) to[out=-180,in=85, distance=.3cm] (0100);
	     \draw[very thick, cyan] (0100) to[out=-180,in=85, distance=1.2cm] (01000);
	     \draw[very thick, cyan] (01000) to[out=-180,in=85, distance=.25cm] (010000);
	    \draw[very thick, cyan] (010000) to[out=-180,in=85, distance=.25cm] (0100000);
	     \draw[very thick, cyan] (0100000) to[out=-180,in=85, distance=.25cm] (01);

		\draw[very thick, cyan] (10.5,6.8) to[out=-180,in=85, distance=.3cm] (0000110);
	    \draw[very thick, cyan] (0000110) to[out=-180,in=85, distance=.8cm] (00001100);
	     \draw[very thick, cyan] (00001100) to[out=-180,in=85, distance=.3cm] (000011);
	     
	     \draw[very thick, cyan] (9.6,4.2) to[out=-180,in=85, distance=.27cm] (00001110);
	    \draw[very thick, cyan] (00001110) to[out=-180,in=85, distance=.5cm] (000011100);
	     \draw[very thick, cyan] (000011100) to[out=-180,in=85, distance=.27cm] (0000111);
	     
	     	     \draw[very thick, green] (20.6,4.2) to[out=-180,in=85, distance=.27cm] (01000110);
	    \draw[very thick, green] (01000110) to[out=-180,in=85, distance=.5cm] (010001100);
	     \draw[very thick, green] (010001100) to[out=-180,in=85, distance=.27cm] (0100011);
	     
	     \draw[very thick, green] (21.5,6.8) to[out=-180,in=85, distance=.3cm] (0100010);
	    \draw[very thick, green] (0100010) to[out=-180,in=85, distance=.8cm] (01000100);
	     \draw[very thick, green] (01000100) to[out=-180,in=85, distance=.3cm] (010001);

	     \draw[very thick, cyan] (12,10.5) to[out=-180,in=85, distance=.3cm] (000010);
		\draw[very thick, cyan] (000010) to[out=-180,in=85, distance=.5cm] (0000100);
	     \draw[very thick, cyan] (0000100) to[out=-180,in=85, distance=.4cm] (00001000);
	     \draw[very thick, cyan] (00001000) to[out=-180,in=85, distance=.2cm] (000010000);
	     \draw[very thick, cyan] (000010000) to[out=-180,in=85, distance=.4cm] (00001);
	     
	\draw[thick, red, opacity=.8] (gamma)--(16,13.5);	
	\draw[thick, red, opacity=.8] (15.9,13.5)--(16.1,13.5);	
	
	\draw[thick, red, opacity=.8] (empty)--(15.2,2.2);
	\draw[thick, red, opacity=.8] (15.1,12.5)--(15.3,12.5);
	\draw[thick, red, opacity=.8] (15.1,5.6)--(15.3,5.6);
	\draw[thick, red, opacity=.8] (15.1,3.5)--(15.3,3.5);
	\draw[thick, red, opacity=.8] (15.1,2.2)--(15.3,2.2);

	\draw[thick, red, opacity=.8] (0)--(9.6,12);
	\draw[thick, red, opacity=.8] (9.5,12)--(9.7,12);		  
	  
	\draw[thick, red, opacity=.8] (00)--(8.8,11.5);	 
	\draw[thick, red, opacity=.8] (8.7,11.5)--(8.9,11.5); 
	  
	\draw[thick, red, opacity=.8] (000)--(8,1);
	\draw[thick, red, opacity=.8] (7.9,10.5)--(8.1,10.5);
	\draw[thick, red, opacity=.8] (7.9,6.8)--(8.1,6.8);
	\draw[thick, red, opacity=.8] (7.9,4.2)--(8.1,4.2);
	\draw[thick, red, opacity=.8] (7.9,2.4)--(8.1,2.4);	 
	\draw[thick, red, opacity=.8] (7.9,1.4)--(8.1,1.4);
	\draw[thick, red, opacity=.8] (7.9,1)--(8.1,1);
	
	\draw[thick, orange, opacity=.9] (000010)--(11.5,1.2);
	\draw[thick, orange, opacity=.9] (11.4,9)--(11.6,9);
	\draw[thick, orange, opacity=.9] (11.4,5)--(11.6,5);
	\draw[thick, orange, opacity=.9] (11.4,1.2)--(11.6,1.2);
	
	\draw[thick, orange, opacity=.9] (0000100)--(10,8.4);
	\draw[thick, orange, opacity=.9] (00001000)--(9.2,5.5);
	
	\draw[thick, orange, opacity=.9] (16.7,5.6)--(16.7,5.1);
	\draw[thick, orange, opacity=.9] (0110)--(16.3,1.5);
	\draw[thick, orange, opacity=.9] (01100)--(15.6,3.5);
		\draw[thick, orange, opacity=.9] (16.2,3.9)--(16.4,3.9);
	\draw[thick, orange, opacity=.9] (16.2,1.5)--(16.4,1.5);
	\draw[thick, orange, opacity=.9] (15.5,3.5)--(15.7,3.5);
	\draw[thick, orange, opacity=.9] (16.6,5.1)--(16.8,5.1);
	
	\draw[thick, orange, opacity=.9] (9.6,4.2)--(9.6,3.7);
	\draw[thick, orange, opacity=.9] (00001110)--(9.2,1.8);
	\draw[thick, orange, opacity=.9] (000011100)--(8.4,2.4);
	\draw[thick, orange, opacity=.9] (9.5,3.7)--(9.7,3.7);
	\draw[thick, orange, opacity=.9] (9.1,1.8)--(9.3,1.8);
	\draw[thick, orange, opacity=.9] (9.1,2.9)--(9.3,2.9);
	\draw[thick, orange, opacity=.9] (8.3,2.4)--(8.5,2.4);
	
		\draw[thick, yellow, opacity=1] (20.6,4.2)--(20.6,3.7);
	\draw[thick, yellow, opacity=1] (01000110)--(20.2,1.8);
	\draw[thick, yellow, opacity=1] (010001100)--(19.4,2.4);
	\draw[thick, yellow, opacity=1] (20.5,3.7)--(20.7,3.7);
	\draw[thick, yellow, opacity=1] (20.1,1.8)--(20.3,1.8);
	\draw[thick, yellow, opacity=1] (20.1,2.9)--(20.3,2.9);
	\draw[thick, yellow, opacity=1] (19.3,2.4)--(19.5,2.4);
	
	\draw[thick, orange, opacity=.9] (000010000)--(8.8,6.8);
	\draw[thick, orange, opacity=.9] (9.9,8.4)--(10.1,8.4);
	\draw[thick, orange, opacity=.9] (9.1,5.5)--(9.3,5.5);
	\draw[thick, orange, opacity=.9] (9.1,8)--(9.3,8);
	\draw[thick, orange, opacity=.9] (8.7,6.8)--(8.9,6.8);
	
		\draw[thick, orange, opacity=.9] (10.5,6.8)--(10.5,6.4);
	\draw[thick, orange, opacity=.9] (0000110)--(10,3.3);
	\draw[thick, orange, opacity=.9] (00001100)--(8.5,4.2);
	\draw[thick, orange, opacity=.9] (10.4,6.4)--(10.6,6.4);
	\draw[thick, orange, opacity=.9] (9.9,4.8)--(10.1,4.8);
	\draw[thick, orange, opacity=.9] (9.9,3.3)--(10.1,3.3);
	\draw[thick, orange, opacity=.9] (8.4,4.2)--(8.6,4.2);
	
	\draw[thick, yellow, opacity=1] (21.5,6.8)--(21.5,6.4);
	\draw[thick, yellow, opacity=1] (0100010)--(21,3.3);
	\draw[thick, yellow, opacity=1] (01000100)--(19.5,4.2);
	\draw[thick, yellow, opacity=1] (21.4,6.4)--(21.6,6.4);
	\draw[thick, yellow, opacity=1] (20.9,4.8)--(21.1,4.8);
	\draw[thick, yellow, opacity=1] (20.9,3.3)--(21.1,3.3);
	\draw[thick, yellow, opacity=1] (19.4,4.2)--(19.6,4.2);
	
	\draw[thick, orange, opacity=.9] (01000)--(16.4,6.4);
	\draw[thick, orange, opacity=.9] (010000)--(16,6);
	\draw[thick, orange, opacity=.9] (0100000)--(15.6,5.6);
	\draw[thick, orange, opacity=.9] (16.3,6.4)--(16.5,6.4);
	\draw[thick, orange, opacity=.9] (15.9,6)--(16.1,6);
	\draw[thick, orange, opacity=.9] (15.7,5.6)--(15.5,5.6);
	
	\draw[thick, red, opacity=.8] (0000)--(4.8,10);	
	\draw[thick, red, opacity=.8] (4.7,10)--(4.9,10);	
	
	\draw[thick, orange, opacity=.9] (6.5,9)--(6.5,8.4);
	\draw[thick, orange, opacity=.9] (0000000)--(5.8,5.5);
	\draw[thick, orange, opacity=.9] (00000000)--(4.4,5);	
	\draw[thick, orange, opacity=.9] (6.4,8.4)--(6.6,8.4);
	\draw[thick, orange, opacity=.9] (5.7,5.5)--(5.9,5.5);
	\draw[thick, orange, opacity=.9] (4.3,5)--(4.5,5);
	\draw[thick, orange, opacity=.9] (5.7,8)--(5.9,8);
	
	\draw[thick, yellow, opacity=1] (14,9)--(14,8.4);
	\draw[thick, yellow, opacity=1] (00001010)--(13.3,5.5);
	\draw[thick, yellow, opacity=1] (000010100)--(11.9,5);	
	\draw[thick, yellow, opacity=1] (13.9,8.4)--(14.1,8.4);
	\draw[thick, yellow, opacity=1] (13.2,5.5)--(13.4,5.5);
	\draw[thick, yellow, opacity=1] (11.8,5)--(12,5);
	\draw[thick, yellow, opacity=1] (13.2,8)--(13.4,8);
	
	\draw[thick, orange, opacity=.9] (19,11.5)--(19,1);
	\draw[thick, orange, opacity=.9] (18.9,10.5)--(19.1,10.5);
	\draw[thick, orange, opacity=.9] (18.9,6.8)--(19.1,6.8);
	\draw[thick, orange, opacity=.9] (18.9,4.2)--(19.1,4.2);
	\draw[thick, orange, opacity=.9] (18.9,2.4)--(19.1,2.4);	 
	\draw[thick, orange, opacity=.9] (18.9,1.4)--(19.1,1.4);
	\draw[thick, orange, opacity=.9] (18.9,1)--(19.1,1);
	\draw[thick, orange, opacity=.9] (18.9,11.5)--(19.1,11.5);
	
	\draw[thick, red, opacity=.9] (00000)--(4,1.2);
	\draw[thick, red, opacity=.8] (3.9,9)--(4.1,9);
	\draw[thick, red, opacity=.8] (3.9,5)--(4.1,5);
	\draw[thick, red, opacity=.8] (3.9,1.2)--(4.1,1.2);
	  
	\draw (empty) node[above left] {$\varnothing$};
	\draw (0) node[above left] {$0$};
	\draw (00) node[above left] {$00$};
	\draw (000) node[above left] {$000$};
	\draw (0000) node[above left] {$0000$};
	\draw (00000) node[above left] {$00000$};
	\draw (01) node[above left] {$\mathit{01}$};
	\draw (011) node[above left] {$\mathit{011}$};
	\draw (00001) node[above left] {$\mathit{00001}$};
	\draw (000011) node[above left] {$\mathit{000011}$};
	\draw (0000111) node[above left] {$\mathit{0000111}$};
	\draw (000000) node[above left] {$\mathit{000000}$};
	\draw (010) node[above left] {$010$};
	\draw (0100) node[above left] {$0100$};
	\draw (01000) node[above left] {$01000$};
	\draw (010000) node[above left] {$010000$};	
	\draw (0100000) node[above left] {\color{gray} $0100000$};	
	\draw (000010) node[above left] {$000010$};
	\draw (0000100) node[above left] {\color{gray} $0000100$};
	\draw (00001000) node[above left] {\color{gray} $00001000$};
	\draw (000010000) node[above left] {\color{gray} $000010000$};
	\draw (010001) node[above left] {$\mathit{010001}$};
	\draw (0100011) node[above left] {$\mathit{0100011}$};
	\draw (0000000) node[above left] {\color{gray} $0000000$};
	\draw (00000000) node[above left] {\color{gray} $00000000$};
	\draw (0000101) node[above left] {\color{gray} $\mathit{0000101}$};
	\draw (00001010) node[above left] {\color{gray} $00001010$};
	\draw (000010100) node[above left] {\color{gray} $000010100$};
	\draw (0111) node[above left] {\color{gray} $\mathit{0111}$};
	\draw (0000110) node[above left] {\color{gray} $0000110$};
	\draw (00001100) node[above left] {\color{gray} $00001100$};
	\draw (0100010) node[above left] {\color{gray} $0100010$};
	\draw (01000100) node[above left] {\color{gray} $01000100$};
	
	\draw (00001110) node[above left] {\color{gray} $00001110$};
	\draw (000011100) node[above left] {\color{gray} $000011100$};
	\draw (01000110) node[above left] {\color{gray} $01000110$};
	\draw (0110) node[above left] {$0110$};
	\draw (01100) node[above left] {$01100$};
	\draw (01110) node[above left] {\color{gray} $01110$};
	\draw (011100) node[above left] {\color{gray} $011100$};
	\draw (010001100) node[above left] {\color{gray} $010001100$};
	
	\draw (empty) node[circle,draw,fill=black, scale=.25]{};
	\draw (0) node[circle,draw,fill=black, scale=.25]{};
	\draw (00) node[circle,draw,fill=black, scale=.25]{};
	\draw (000) node[circle,draw,fill=black, scale=.25]{};
	\draw (0000) node[circle,draw,fill=black, scale=.25]{};
	\draw (00000) node[circle,draw,fill=black, scale=.25]{};
	\draw (01) node[circle,draw,fill=black, scale=.25]{};
	\draw (011) node[circle,draw,fill=black, scale=.25]{};
	\draw (00001) node[circle,draw,fill=black, scale=.25]{};
	\draw (000011) node[circle,draw,fill=black, scale=.25]{};
	\draw (0000111) node[circle,draw,fill=black, scale=.25]{};
	\draw (000000) node[circle,draw,fill=black, scale=.25]{};
	\draw (010) node[circle,draw,fill=black, scale=.25]{};
	\draw (0100) node[circle,draw,fill=black, scale=.25]{};
	\draw (01000) node[circle,draw,fill=black, scale=.25]{};
	\draw (010000) node[circle,draw,fill=black, scale=.25]{};
	\draw (0100000) node[circle,draw,fill=black, scale=.25]{};
	\draw (000010) node[circle,draw,fill=black, scale=.25]{};
	\draw (0000100) node[circle,draw,fill=black, scale=.25]{};
	\draw (00001000) node[circle,draw,fill=black, scale=.25]{};
	\draw (000010000) node[circle,draw,fill=black, scale=.25]{};
	\draw (010001) node[circle,draw,fill=black, scale=.25]{};
	\draw (0100011) node[circle,draw,fill=black, scale=.25]{};
	\draw (0000000) node[circle,draw,fill=black, scale=.25]{};
	\draw (00000000) node[circle,draw,fill=black, scale=.25]{};
	\draw (0000101) node[circle,draw,fill=black, scale=.25]{};
	\draw (00001010) node[circle,draw,fill=black, scale=.25]{};
	\draw (000010100) node[circle,draw,fill=black, scale=.25]{};
	\draw (0111) node[circle,draw,fill=black, scale=.25]{};
	\draw (0000110) node[circle,draw,fill=black, scale=.25]{};
	\draw (00001100) node[circle,draw,fill=black, scale=.25]{};
	\draw (0100010) node[circle,draw,fill=black, scale=.25]{};
	\draw (01000100) node[circle,draw,fill=black, scale=.25]{};
	\draw (00001110) node[circle,draw,fill=black, scale=.25]{};
	\draw (000011100) node[circle,draw,fill=black, scale=.25]{};
	\draw (01000110) node[circle,draw,fill=black, scale=.25]{};
	\draw (010001100) node[circle,draw,fill=black, scale=.25]{};
	\draw (0110) node[circle,draw,fill=black, scale=.25]{};
	\draw (01100) node[circle,draw,fill=black, scale=.25]{};
	\draw (01110) node[circle,draw,fill=black, scale=.25]{};
	\draw (011100) node[circle,draw,fill=black, scale=.25]{};
	\end{tikzpicture}
  \end{adjustbox}
\end{figure}

\newpage

\bibliographystyle{plain}
\bibliography{HighWalks}

\end{document}